\documentclass[12pt]{amsart}

\usepackage{amsmath, amssymb, amscd, verbatim, xspace,amsthm}
\usepackage{latexsym, epsfig, color}
\usepackage[hidelinks]{hyperref}

\usepackage{fullpage,pdfsync}

\usepackage{mathrsfs}
\usepackage{physics}
\usepackage{tikz-cd}

\usepackage[OT2,T1]{fontenc}
\DeclareSymbolFont{cyrletters}{OT2}{wncyr}{m}{n}
\DeclareMathSymbol{\Sha}{\mathalpha}{cyrletters}{"58}

\newcommand{\C}{\ensuremath{{\mathbb{C}}}}
\newcommand{\Z}{\ensuremath{{\mathbb{Z}}}\xspace}
\renewcommand{\P}{\ensuremath{{\mathbb{P}}}}

\newcommand{\Q}{\ensuremath{{\mathbb{Q}}}}
\newcommand{\R}{\ensuremath{{\mathbb{R}}}}
\newcommand{\F}{\ensuremath{{\mathbb{F}}}}

\newcommand{\E}{\ensuremath{{\mathbb{E}}}}

\newcommand{\ra}{\rightarrow}

\newcommand\Hom{\operatorname{Hom}}

\newcommand\Aut{\operatorname{Aut}}
\newcommand\Out{\operatorname{Out}}
\newcommand\im{\operatorname{im}}

\newcommand\Sym{\operatorname{Sym}}

\newcommand\Sur{\operatorname{Epi}}
\newcommand\SSur{\operatorname{Sur}}

\newcommand\tensor{\otimes}
\newcommand\isom{\simeq}
\newcommand\sub{\subset}

\newcommand\GL{\operatorname{GL}}

\newcommand\Pic{\operatorname{Pic}}

\renewcommand\C{\mathcal{L}}
\newcommand\W{\mathcal{W}}
\newcommand\V{\mathcal{V}}

\newcommand{\Surj}{\operatorname{Sur}}

\newcommand{\Epi}{\operatorname{Epi}}

\newcommand{\N}{\mathbb N}

\newcommand{\Prf}{\mathcal{P}}
\newcommand{\category}{C}
\newcommand{\sS}{\mathscr{S}}
\newcommand{\Diamantine}{diamond }
\newcommand{\Ggp}{G}
\newcommand{\Fgp}{F}
\newcommand{\ggp}{G}
\newcommand{\fgp}{F}
\newcommand{\barHornot }{H}

\numberwithin{equation}{section}
\newtheorem{proposition}[equation]{Proposition}
\newtheorem{theorem}[equation]{Theorem}
\newtheorem{corollary}[equation]{Corollary}
\newtheorem{example}[equation]{Example}

\newtheorem{lemma}[equation]{Lemma}

\theoremstyle{remark}
\newtheorem{remark}[equation]{Remark}

\usepackage[capitalise]{cleveref}

\newtheorem{definition}[equation]{Definition}

\newtheorem{nts}{Note to self}

\newcommand{\melanie}[1]{{\color{blue} \sf $\clubsuit\clubsuit\clubsuit$ Melanie: [#1]}}

\begin{document}

\begin{abstract} The moment problem in probability theory asks for criteria for when there exists a unique measure with a given tuple of moments. 
We study a variant of this problem for random objects in a category, where a moment is given by the average number of epimorphisms to a fixed object.
When the moments do not grow too fast, we give a necessary and sufficient condition for existence of a distribution with those moments, show that a unique such measure exists,
give formulas for the measure in terms of the moments, and prove that measures with those limiting moments  approach that particular measure.
Our result applies to categories satisfying some finiteness conditions and a condition that gives an analog of the second isomorphism theorem,
including the categories of finite groups,  finite modules, finite rings, as well as many variations of these categories.
This work is motivated by the non-abelian Cohen-Lenstra-Martinet program in number theory, which aims to calculate the distribution of random profinite groups arising as Galois groups of maximal unramified extensions of random number fields.\end{abstract}

\title{The moment problem for random objects in a category}
\author{Will Sawin}
\address{Department of Mathematics\\
Princeton University \\
Fine Hall, Washington Road \\
Princeton, NJ 08540 USA}  
\email{wsawin@math.princeton.edu}

\author{Melanie Matchett Wood}
\address{Department of Mathematics\\
Harvard University\\
Science Center Room 325\\
1 Oxford Street\\
Cambridge, MA 02138 USA}  
\email{mmwood@math.harvard.edu}

\maketitle

\section{Introduction}

The $k$th moment of a measure $\nu$ on the real numbers is defined to be $\int_X X^k d \nu$. 
The classical moment problem asks whether, given a sequence $M_k$, whether there exists a measure whose $k$th moment is $M_k$ for all $k$, and whether that measure is unique.
 For the existence part, a natural necessary and sufficient condition is known, whereas for uniqueness, multiple convenient sufficient conditions are known \cite{Schmudgen2017}. 
 An important auxiliary result allows us to apply the uniqueness result to determine the limit of a sequence of distributions from the limit of their moments. 

There have been recent  works, in number theory,  probability, combinatorics, and topology  
that have used analogs of the moment
problem to help identify distributions, but not distributions of numbers--rather distributions of algebraic objects such as groups \cite{Ellenberg2016,Wood2017,Wood2019a,Meszaros2020,Nguyen2021}
or modules \cite{Liu2019}.  Even more recent works require distributions on more complicated categories, such as groups with an action of a fixed group \cite{Boston2017,Sawin2020},  groups with certain kinds of pairings \cite{Lipnowski2020},  groups with other additional data \cite{Wood2018,Liu2022}, or tuples of groups \cite{Lee2022, Nguyen2022}.
The distributions that have been identified this way include distributions of class groups of function fields,  profinite completions of fundamental groups of $3$-manifolds,  sandpile groups of random graphs, and cokernels of random matrices.  These works have often required new proofs of moment problem analogs in each different category considered.

The goal of this paper is to prove a theorem on the moment problem for distributions of objects of a rather general category.
We expect that future works of many authors will consider distributions on new categories and we have tried to make our result  as general as possible so that it can be applied in these new settings. Already, the results of this paper have been used in subsequent work of the authors to study class groups \cite{SW-clrou} and by others to study cokernels of $p$-adic random matrices \cite{CheongHuang2023,CheongYu2023,Yan2023} and profinite groups that model absolute Galois groups \cite{AlbertsLD}.  In a future paper, we will consider profinite groups $G$ together with an action of a fixed group $\Gamma$ and a class in the group homology $H^3(G,\mathbb Z/n)$,  and will apply the main results of the current paper to understand the distributions of class groups and fundamental groups of curves over finite fields, and their influence under roots of unity in the base field. 

In the categories in which we work, 
we give a criterion for there to exist a measure on isomorphism classes of that category with a given sequence of moments, and a criterion for the measure  to be unique. We also give explicit formulas for the measure, when it exists, in terms of the moments.
This is a significant advance, as  previous work  has always relied on a priori knowing a distribution with certain moments.  Being able to construct a distribution from its moments  opens new avenues such as in our subsequent paper mentioned above.

Working at a high level of generality requires us to use some notation we expect to be unfamiliar. 
Thus, for clarity, we begin with the statement of our main theorem in the category of groups.  In many cases of interest, the probability distributions of interest are on the set of profinite groups (which includes finite groups), so we work there.

Let $\category$ be the category of finite groups.  A \emph{level} $\C$
 is a finitely generated formation of finite groups (a \emph{formation} is a set of isomorphism classes of groups closed under taking quotients and finite subdirect products). 
 For example, the group $\Z/p\Z$ generates the level of 
elementary abelian $p$-groups.  We consider the set $\mathcal{P}$ of isomorphism classes of profinite groups that have finitely many (continuous) homomorphisms to any finite group.  A group $X\in \mathcal{P}$ has a pro-$\C$ completion $X^\C$ which is the inverse limit of the (continuous) quotients of $X$ in $\C$.
We consider a topology generated on $\mathcal{P}$ generated by $\{X \,|\, X^\C\isom F\}$ as $\C$ and $F$ vary (and the associated Borel $\sigma$-algebra).
For a finite group $G$, the \emph{$G$-moment}
 of a measure $\nu$ on $\mathcal{P}$ is defined to be
$$
\int_{X\in \mathcal{P}} |\SSur(X,G)| d\nu.
$$

We give a condition, \emph{well-behaved}, below on a  real numbers $M_G$ indexed by (isomorphism classes of) finite groups
 that quantifies the notion that the $M_G$ ``don't grow too fast.''
For each surjection $\rho:G \ra F$, the quotients of $G$ (up to isomorphism) that $\rho$ factors through form a poset (\cite[Section 3.7]{Stanley2012}). We let $\mu(\rho)=\mu(F,G)$ be the M\"obius function of that poset.  For $F,G\in C$, we let $\hat{\mu}(F,G)=\sum_{\rho\in \SSur(G,F)/\Aut(F)} \mu(\rho)$.
Given such $M_G$, for each level $\C$ and finite group $F$, we define (what will be the formulas for our measures)
\begin{equation}\label{E:defv}
v_{\C,F}:=\sum_{G \in \C} \frac{\hat{\mu}(F,G)}{|\Aut(G)|} M_G.
\end{equation}

\begin{theorem}\label{T:groups}
For each finite group $G$, let $M_G$ be a real number such that the sequence $M_G$ is well-behaved.
\begin{enumerate}
\item[(Existence):] There is a measure on $\mathcal{P}$ with $G$-moment $M_G$ for each finite group $G$ if and only if the $v_{\C,F}$ defined above are non-negative for each $\C$ and $F$.
\item[(Uniqueness):] When such a measure $\nu$ exists, it is unique and is given by the formulas
$$
\nu(\{X \,|\, X^\C\isom F\}) =v_{\C,F}.
$$
\item[(Robustness):] If $\nu^t$ are measures on $\Prf$ such that for each finite group $G$,
$$
\lim_{t\ra\infty} \int_{X \in \mathcal{P}} |\SSur(X,G)| d\nu^t= M_G, 
$$
then a measure $\nu$ with moments $M_G$ exists and
the  $\nu^t$ weakly converge to $\nu$.
\end{enumerate}
\end{theorem}

If one wishes to work with only finite groups, we have an analogous result, Theorem \ref{T:general}, for each level $\C$ of finite groups. In Section~\ref{s-examples}, we give examples to show how to compute these $v_{\C,F}$ in practice.

We next describe the more general formalism of which Theorem \ref{T:groups} is a special case.

\subsection{Main result}

We first review some standard notation. Let $\category$ be a small category. 
A \emph{quotient} of $G$ is a pair $(F, \pi^G_F)$ of an object $F\in \category$ and an epimorphism $\pi^G_F:G\ra F$.  Quotients 
$(F, \pi^G_F)$ and $(F,' \pi^G_{F'})$ are isomorphic when there is an isomorphism $\rho:F\ra F'$ such that $\pi^G_{F'}=\rho \pi^G_F$.

For $G$ an object of $\category$, quotients of $G$, taken up to isomorphism, form a partially ordered set, where $ H\geq  F$ if there is $h \colon H \to F$ compatible with the morphisms from $G$. 

A \emph{lattice} is a partially ordered set where any two elements  $x,y$ have a least upper bound (join) denoted $x \vee y$, and
a greatest lower bound (meet), denoted  $x\wedge y$. A \emph{modular} lattice is one where $a \leq b $ implies $a \vee (x \wedge b) = (a \vee x) \wedge b$ for all $x$.

We now introduce some less standard notation. For $F$ a quotient of $G$, let $ [F,G]$ be the partially ordered set of (isomorphism classes of) quotients of $G$ that are $\geq F$.

An object $G\in \category$ is called \emph{minimal} if it has only one quotient (itself).

A set $\mathcal L$ of objects of $\category$, stable under isomorphism, is called \emph{downward-closed} if, whenever $H$ is a quotient of $G$, if $ G\in \mathcal L$ then $H \in\mathcal L$.  It is called \emph{join-closed} if, whenever $H$ and $F$ are quotients of $G$ and $H \vee F$ exists (in the partially ordered set of isomorphism classes of quotients of $G$), 
if $H \in \mathcal L $ and $F\in \mathcal L$ then $H \vee F\in \mathcal L$.

A \emph{formation} is a downward-closed join-closed set of isomorphism classes of $\category$ containing every minimal object of $\category$.
For sets $B,\C$ of isomorphism classes of objects of $\category$, 
we say $B$ \emph{generates} $\C$ if $\C$ is the 
smallest downward-closed join-closed set containing $B$   and every minimal object of $\category$.  
A \emph{level} $\C$ is a set of isomorphism classes of objects of  $\category$ generated by a finite set $B$.

An epimorphism $G \to F$ is \emph{simple} if it is not an isomorphism and, whenever it is written as a composition of two epimorphisms, one is an isomorphism (i.e. if the interval $[F,G]$ contains only $F$ and $G$).

We can now define the type of categories to which our main theorem applies.

\begin{definition}\label{D:diamond}
A \emph{\Diamantine category} is a small category $\category$ 
such that
\begin{enumerate}
\item \label{D:modular}
For each object $G \in \category$, the set of (isomorphism classes of) quotients of $G$ form a finite modular lattice.
\item \label{D:Aut}
Each object in $\category$ has finite automorphism group.
\item \label{D:finite}
For each level $\C$ of $\category$ and each $G\in \C$, there are finitely many elements of $\C$ with a simple epimorphism to $G$.
\item The category $\category$ has at most countably many isomorphism classes of minimal objects.
\end{enumerate}
\end{definition}

The terminology comes from the fact that (1) implies these categories satisfy the \emph{Diamond Isomorphism Theorem} (Lemma~\ref{dit}), a generalization of the Second Isomorphism Theorem.
We will see later, in Section \ref{s-examples}, that the category of finite groups, the category of finite commutative rings, and the category of finite $R$-modules for an arbitrary ring $R$, are all examples of \Diamantine categories.
Diamond categories also include more exotic examples such as finite quasigroups, finite Lie algebras, finite Jordan algebras, and more concrete examples such as finite graphs with inclusions of induced subgraphs.  Moreover, we will see in Section 6 that one can make many modifications to a 
\Diamantine category $\category$ to obtain a new \Diamantine category,  such as one of objects of $\category$ with the action of a fixed finite group,
objects of $\category$ with an epimorphism to a fixed object, or objects of $\category$ with a choice of element in some finite functorial set associated to the object.

We will now define the necessary notation to state our main theorem, which is on the moment problem for measures of isomorphism classes of a \Diamantine category $C$.
We encourage the reader to consider the case when $\category$ is the category of finite groups first when trying to understand the definitions and theorems.
Some help with this is provided in Subsection \ref{ss-translation}.

For $\pi \colon G \to F$ an epimorphism, 
we define $\mu(\pi) = \mu( F, G)$, 
where $\mu$ is the M\"obius function of the poset of quotients of $G$. We set $\hat{\mu}(F,G) = \sum_{ \pi \in\Epi(G,F)/\Aut(F)} \mu(\pi)$.

For $\pi \colon G \to F$ an epimorphism, let $Z(\pi)$ be the number of $H \in [F,G]$ that satisfy the distributive law $H \vee ( K_1 \wedge K_2)  = (H \vee K_1 ) \wedge (H \vee K_2)$ for all $K_1,K_2 \in [F,G]$. We note that $F$ and $G$ always satisfy the distributive law for all $K_1,K_2$, so $Z(\pi) \geq 2$.

As in the classical theorems of analysis and probability saying that moments not growing too fast determine a unique distribution of real numbers, we will need
a notion of moments not growing too fast.
If for each $G\in\category_{/\isom}$ (isomorphism class of $\category$),  we choose a real number $M_G$, we say the tuple $(M_G)_{G}\in \R^{\category_{/\isom}}$   is \emph{well-behaved} at a formation $\W$ if for all $F \in \W$,
$$
\sum_{G \in \W}     \sum_{ \pi \in\Epi(G,F) }\frac{{|\mu}(F,G)|}{|\Aut(G)|} Z ( \pi)^3 |M_G| <\infty.$$
 (In Section~\ref{s-calculation} we will give formulas for $\mu(F,G)$ and $Z(\pi)$ that are easy to compute in practice.)
   We say $(M_G)_{G}$ is \emph{well-behaved} if it is well-behaved at level $\C$ for every level $\C$. 
We will provide some explanation in Remark \ref{well-behaved-explanation}, after the statements of the theorems in which the notion of well-behaved is used, to motivate why  this is the rough shape of what one needs for moments ``not growing too fast''.

For any formation $\W$,  we define the \emph{$\W$-quotient} $G^{\W}$ of $G$ to be the maximal quotient of $G$ that is in $\W$, i.e. the join of all quotients of $G$ that are in $\W$.
We use the discrete $\sigma$-algebra for any measure on a level $\C$.

Then our main theorem, below, says that for well-behaved moments,  a distributions exists with those moments if and only if certain formulas in the moments are positive, and when a distribution exists with those moments, it is given by those certain formulas.  Moreover, this uniqueness is robust even through taking limits.  
We write $\Sur(F,G)$ for the set of epimorphisms from $F$ to $G$.
If $\nu$ is a distribution on $\category_{/\isom}$, then \emph{$G$-moment} of $\nu$ is defined to be
\begin{equation}\label{E:mom}
\int_{F \in \category_{/\isom}} |\Sur(F,G)| d\nu.
\end{equation}
We also define for a formation $\W$ and $F\in \W$,
\begin{equation}\label{E:defvgen}
v_{\W,F}:=\sum_{G \in \W} \frac{\hat{\mu}(F,G)}{|\Aut(G)|} M_G.
\end{equation}

\begin{theorem}[Distributions on $\C$ from moments]\label{T:general}
Let $\category$ be a \Diamantine category 
and let $\C$ be a level. 
For each $G\in \C$,  let  $M_G$ be a real number so that $(M_G)_{G}$ is well-behaved at $\C$.
\begin{enumerate}
\item[(Existence):]
A measure $\nu_\C$ on $\C$ with $G$-moment  $M_G$ for all $G\in \C$ exists if and only if $v_{\C,F}\geq 0$ for all $F\in \C$.
\item[(Uniqueness):]
If such a measure exists, then it is unique and 
$
\nu_\C(\{F\})=v_{\C,F}
$
for all $F\in\C$.
\item[(Robustness):]
If  $\nu^t_\C$ are measures on  $\C$ such that for each $G\in \C$,
$$
\lim_{t\ra\infty} \int_{F \in \C} |\Sur(F,G)| d\nu^t_\C= M_G, 
$$
then for $F\in\C$, we have $\lim_{t\ra\infty} \nu^t_\C(\{F\})=v_{\C,F}.$
\end{enumerate}
\end{theorem}

\subsection{The main result in the profinite case}

If we use $M_G$ for each $G\in \category_{/\simeq}$ to define measures $\nu_\C$ as above for each level $\C$, then by the uniqueness
in Theorem~\ref{T:general}, for levels  $\C\sub \C'$ and the functor $\C' \ra {\C}$ that sends $X$ to $X^{\C}$, the measure $\nu_{\C'}$ pushes forward to $\nu_{{\C}}$.  This compatible system of measures for each level gives rise to a measure $\nu$.  However,  in most of our main examples of interest, this measure $\nu$ is not on the category $\category$, but rather on pro-objects of $\category$.
We define a \emph{pro-isomorphism class} $X$ of $\category$, to be,  for each level $\C$, an isomorphism class  $X^\C\in \C$ such that whenever $\C\sub \C'$ we have that $(X^{\C'})^\C\cong X^\C$.  (Here $X^{\C'}\in \category$, so $(X^{\C'})^\C$ takes the  $\C$-quotient of $X^{\C'}$.) 

The reader might wonder how this definition compares to the usual notion of a pro-object. We will show in Lemma \ref{pro-object-existence} that pro-isomorphism classes are naturally in bijection with isomorphism classes of pro-objects satisfying an additional condition we call \emph{small} (Definition \ref{def-small}),
 which in the case where $\category$ is the category of finite groups, so that pro-objects are profinite groups, restricts to the usual notion of small profinite groups (i.e groups with finitely many open subgroups of index $n$ for each $n$). 
 When $R$ is a Noetherian local commutative ring with finite residue field,  and $\hat{R}$ its completion at its maximal ideal,
then  pro-isomorphism 
classes of finite $R$-modules
are naturally in bijection with isomorphism classes of finitely generated $\hat{R}$-modules (Lemma~\ref{L:prolocalRmod}).

Let $\Prf$ be the set of pro-isomorphism classes of $\category$.
We define a topology on $\Prf$ with a generated by the open sets $U_{\C,F} = \{X \in \Prf \mid X^\C \cong F\}$ for each level $\C$ and $F\in \C$. It is straightforward to check that the $U_{\C,F}$ are a basis for this topology (i.e. that finite intersections of $U_{\C,F}$ are unions of $U_{\C,F}$).

For $X$ a pro-isomorphism class,  $G\in \category,$ and $\C$ a level containing $G$,
we have that $\abs{ \Sur (X^\C, G)}$ is independent of the choice of $\C$. 
(This is because if $\mathcal{D}$ is the level generated by $G$, 
any epimorphism $X^\C \to G$ factors through $(X^{\C} )^{\mathcal{D}} \isom X^{\mathcal{D}}$, and so $\Sur (X^\C, G) \cong \Sur(X^{\mathcal{D}}, G)$.) 
We define $\abs{ \Sur (X, G)}= \abs{ \Sur (X^\C, G)}$ for any such $\C$.  

  We take a $\sigma$-algebra (according to which we consider any measures) on $\Prf$ to be generated by the $U_{\C,F}$ for each level $\C$ and $F\in \C$. 
 
 Now, we give the consequence of Theorem~\ref{T:general} for measures on $\Prf$.
 
  \begin{theorem}\label{analytic-main-measure} 
Let $\category$ be a \Diamantine category. 
For each $G\in \category_{/\simeq}$,  let  $M_G$ be a real number so that $(M_G)_{G}$ is well-behaved.
\begin{enumerate}
\item[(Existence):]
There exists a measure $\nu$ on $\Prf$
with $G$-moment $M_G$ for all $G\in \category$ if and only if for every level $\C$ and every $F\in \C$, we have $v_{\C,F}\geq 0$.
\item[(Uniqueness):]
If such a measure $\nu$ exists, then it is unique, and for every level $\C$ and $F\in\C$, we have $\nu(U_{\C,F})=v_{\C,F}$.
\end{enumerate}
\end{theorem}

If $\category$ has countably many isomorphism classes, then our $\sigma$-algebra on $\Prf$  is equivalent to the Borel $\sigma$-algebra of the topology on $\Prf$
generated by the $U_{\C,F}$  (because then there are only countably many pairs $\C,F$, so arbitrary unions
of them are the same as countable unions).

\begin{theorem}\label{T:weak-convergence}
Let $\category$ be a \Diamantine category with at most countably many isomorphism classes. 
For each $G\in \category_{/\simeq}$,  let  $M_G$ be a real number so that $(M_G)_{G}$ is well-behaved.
 If $\nu^t$ is a sequence of measures on $\Prf$ such that for all $G \in \category$ 
 \begin{enumerate}
 \item[(Robustness):] \quad\quad\quad \quad\quad\quad \quad\quad\quad  $\displaystyle{\lim_{t \to \infty} \int_{X\in \Prf}  { \abs{\Sur(X,G)}} d\nu^t  = M_G, }$
 \end{enumerate}
then a measure $\nu$ as in Theorem~\ref{analytic-main-measure} exists 
and
as $t\ra\infty$ the $\nu^t$ weakly converge to $\nu$. 
\end{theorem}

\begin{remark}\label{well-behaved-explanation} We explain some of the motivation for the definition of well-behaved. In both Theorem \ref{T:general} and Theorem \ref{analytic-main-measure}, the sum 
$$
\sum_{G \in \C} \frac{\hat{\mu}(F,G)}{|\Aut(G)|} M_G
$$
 appears as the formula for a probability. For this formula to be well-defined, the sum must be convergent, and since there is no obvious ordering of the terms, the sum must be absolutely convergent. The condition that the $M_G$ be well-behaved is a strictly stronger condition than this, because it requires the absolute convergence of the sum with an additional $Z(\pi)^3$ factor. This is technical restriction required by our method to prove the existence theorems, but we do not know if it is necessary for the statement of the theorems to hold. However, for the applications we are considering, the $Z(\pi)^3$ factor is not large enough to turn the sum from convergent to divergent, and thus causes no serious difficulty.
  \end{remark}
  
  \subsection{Motivation from number theory and prior work}
  
When $\category$ is the category of $\F_q$-vector spaces for some prime power $q$, then the moments in our sense
are  just a finite upper triangular  transformation of the classical moments of the size of the group, since $|\Hom(G,\F_q^k)|=|G|^k$.
However, the bounds on the classical moment problem have not been sufficient for many authors studying distributions on $\category$, and they have had to develop their own moment theorems.  This includes work on the distribution of  Selmer groups of congruent number elliptic curves \cite{Heath-Brown1994}, 
ranks of (non-uniform) random matrices over finite fields \cite{Alekseichuk1999,
 Cooper2000a}, 
and
$4$-ranks of class groups of quadratic fields \cite{Fouvry2006, Fouvry2006a}.  
  
Cohen and Lenstra \cite{Cohen1984} made conjectures on the distribution of class groups of quadratic number fields that have been a major motivating and explanatory force for work on class groups, including that on $4$-ranks mentioned above.  
  Ellenberg, Venkatesh, and Westerland \cite{Ellenberg2016} showed that the analog of Cohen and Lenstra's conjectures for Sylow $p$-subgroups of class groups of imaginary quadratic fields, 
when $\Q$ is replaced by $\F_q(t)$, holds as $q\ra\infty$.  To do this, they found the limiting moments of the distributions of interest and showed uniqueness and robustness in the moment problem for
the particular distribution on abelian $p$-groups conjectured by Cohen and Lenstra.
This work introduced the idea that considering averages of $\#\SSur(-,G)$ was a useful way to access distributions on finite abelian groups.
(For finite abelian groups,  these averages are a finite upper triangular transformation of classical mixed moments of a tuple of integral invariants that determine the group,
see \cite[Section 3.3]{Clancy2015}.)

This approach has been followed by many papers, seeking to show function field ($q\ra\infty$) analogs of the conjectures of Cohen and Lenstra and their many generalizations.  In \cite{Wood2018}, the second author shows such an analog for real quadratic function fields, which follows from a more refined result that gives the distribution of pairs $(\Pic^0,d_\infty)$, where $\Pic^0$ is the group of degree $0$ divisors on the associated curve and $d_\infty$ is an element of $\Pic^0$ given by the difference between the two points above $\infty$.  The quotient $\Pic^0/d_\infty$ gives the desired class group.  This work included a moment problem uniqueness theorem for the category of pairs $(A,s)$ where $A$ is an abelian $p$-group and $a\in A$.
In \cite{Boston2017}, Boston and the second author, proved function field analogs of some conjectures of Boston, Bush, and Hajir \cite{Boston2017a} on $p$-class tower groups of imaginary quadratic fields, which generalize the conjectures of Cohen and Lenstra to a non-abelian analog.  This work included a moment 
problem uniqueness theorem for the category of $p$-groups with a generator inverting automorphism.

Liu, Zureick-Brown, and the second author \cite{Liu2019} found the moments of the distributions of the fundamental groups of $\Gamma$-covers of $\P^1_{\F_q}$ split at infinity,
or equivalently of the Galois groups of the maximal unramified extensions of totally real $\Gamma$-extensions of $\F_q(t)$, as $q\ra\infty$.
For the part of these groups of order relatively prime to $q-1$ (the order of the roots of unity in $\F_q(t)$) and $q|\Gamma|$, they were able to construct an explicit distribution on profinite groups with an action of $\Gamma$ that had these moments.  In particular their distribution abelianizes to the distribution of finite abelian $\Gamma$-groups (i.e. groups with an action of $\Gamma$)
conjectured by Cohen and Martinet \cite{Cohen1990} for the class groups of totally real $\Gamma$-number fields, and thus one can apply the uniqueness and robustness results on the moment problem for finite abelian $\Gamma$-groups of Wang the the second author \cite{Wang2021} to obtain a $q\ra\infty $ analog of the Cohen-Martinet conjectures.  The first author \cite{Sawin2020} then showed a uniqueness and robustness result on the moment problem for (not necessarily abelian) $\Gamma$-groups, which then gave a $q\ra\infty$ theorem that the prime-to-$(q-1)q|\Gamma|$-part of the fundamental groups of $\Gamma$-covers of $\P^1_{\F_q}$  approached the distribution constructed in \cite{Liu2019}.   

For the part of the groups that has order not relatively prime to the roots of unity in the base field, even though moments can be found with the methods of \cite{Liu2019},  in general there is no known distribution with those moments.  For the case of quadratic extensions,  
Lipnowski, Tsimerman, and the second author \cite{Lipnowski2020} showed  a $q\ra\infty$ function field result for the class groups of quadratic function fields, even at primes dividing the order of roots of unity in the base field.
Their work considered additional pairings on the class groups (a structure they call a bilinearly enhanced group), and proved uniqueness and robustness theorem on the moment problem for  bilinearly enhanced abelian groups.

In combinatorics, work to find the distribution of Sylow-$p$ subgroups of cokernels of random integral matrices \cite{Wood2019a}, 
the distribution of the entire cokernel of such matrices \cite{Nguyen2021}, the  distribution of Sylow-$p$ subgroups of sandpile groups of Erd\H{o}s-R\'{e}nyi random graphs \cite{Wood2017}, and the
distribution of Sylow-$p$ subgroups of sandpile groups of $d$-regular random graphs \cite{Meszaros2020} has all involved
applying a uniqueness and robustness theorem for the moment problem on finite abelian groups.

Several recent works, particularly in probability, have been concerned with the distribution of tuples of finite abelian groups. Boreico \cite[Corollary 3.8.5]{Boreico2016}, while studying cokernels of random integral matrices with uniform, bounded coefficients, proves uniqueness and robustness results for $r$-tuples $(V_1,\dots,V_r)$ where $V_i$ is an $\F_{q^d_i}$-vector space for some (fixed) $d_i\in\N$.
Nguyen and Van Peski \cite[Theorem 9.1]{Nguyen2022} prove a uniqueness and robustness result for $r$-tuples of abelian groups in order to prove new universality results on the joint distribution of cokernels of products of random matrices.  Lee \cite[Theorem 1.3]{Lee2022} has also proven a uniqueness and robustness result for $r$-tuples of abelian groups, in order to apply it to several universality results for joint distributions of cokernels of random matrices perturbed in various ways.  These uniqueness and robustness results for tuples would also follow from the argument in \cite[Proof of Theorem 6.11]{Wang2021} 
 (since $(R_1\times\cdots \times R_r)$-modules correspond to tuples $(M_1,\dots,M_r)$, where $M_i$ is an $R_i$-module).
Lee \cite[Theorem 1.10]{Lee2022} also proves a uniqueness and robustness result for $r$-tuples of finite groups with an action of a fixed group $\Gamma$, with an application to the joint distribution of the quotient of free profinite  groups by Haar random relations and perturbations of those relations.  

Given the ever growing number of different categories that such moment problem results are required in, we have made every attempt here to give a result that will be applicable as broadly as possible.  All the moment problem results mentioned above
are special cases of the main theorems in this paper (except that, unlike the current paper, the moment problem results of \cite{Sawin2020} do not always require absolute convergence of the sums involved).

In all of the above discussed work, a known explicit distribution was given (or was found by trying various distributions)
with certain moments, and so uniqueness and robustness were the only aspects of the moment problem that were necessary.
Yet there remained settings where moments are known, but no distribution with those moments was known.
In one setting like this, the authors \cite{SW} found the distribution of profinite completions of $3$-manifold groups in the Dunfield-Thurston model of random Heegaard splittings, constructing the distribution itself from the knowledge of the moments.  This required working on the moment problem in the category of pairs $(G,h)$,  where $G$ is a finite group and $h\in H^3(G,\Z)$ and proving theorems on existence, uniqueness, and robustness.  

The current paper makes completely general the approach used in \cite{SW} on reconstruction of a distribution from its moments, via explicit formulas, without any a priori knowledge of a distribution with those moments.  One key motivation is our forthcoming work, which will give a $q\ra\infty$
theorem on the distribution of the part of fundamental groups of $\Gamma$-covers of $\P^1_{\F_q}$ that is relatively prime to $q|\Gamma|$ (but not necessarily $q-1$).  In particular, this will give a  correction to the conjectures of Cohen and Martinet at primes dividing the order of roots of unity in the base field, which will be provided in a short forthcoming paper \cite{SWROUCL}.  
We will obtain the desired distributions using the results of the current paper.

Before this paper,  many papers constructed  distributions as limits of distributions given by random matrix models or models of random generators and relations.
In this setting, it is usually very easy to compute the limiting moments of these distributions.  It is somewhat more challenging, but doable, to compute the limiting distributions.  As the categories got more complicated, showing that the limiting distribution had those limiting moments became increasingly more difficult.  
Now Theorem~\ref{T:general} can be used to not only compute the limiting distributions of such models in a straightforward way from their moments, but also immediately obtain that the limiting distribution has moments given by the limiting moments.  Moreover, if one's only goal is to have explicit formulas for a distribution with certain moments, one no longer has to guess possible models that may have the desired moments, since Theorem~\ref{T:general} gives those formulas directly.

\subsection{Strategy of proof}

To understand the proof of Theorem \ref{T:general}, it is helpful to first consider the special case where $(M_G)_G$ is finitely supported, i.e. $M_G=0$ for $G$ not isomorphic to any object in some finite list $G_1,\dots, G_n$. It follows that, if a measure $\nu$ with moments $M_G$ exists, then $\nu( \{G\})=0$ unless $G$ is isomorphic to one of $G_1,\dots, G_n$, since certainly $M_G \geq \nu (\{G\})$ as there is always at least one epimorphism from $G$ to itself. Thus, in this case, the problem of existence and uniqueness of a measure  can be expressed as the existence and uniqueness of a solution to the finitely many linear equations $M_{G_i}=  \sum_{j=1}^n  \abs{ \Sur(G_j, G_i) } \nu (\{G_j\})$ in finitely many variables $\nu(\{G_j\})$, together with the inequalities $\nu(\{G_j\}) \geq 0$.

The linear equations can be expressed concisely in terms of a matrix $A$ with $A_{ij} = \abs{ \Sur(G_j, G_i) } $. The necessary observation to handle this case is that $A$ is invertible. In fact, it is possible to order $G_1,\dots, G_n$ so that there are no epimorphisms from $G_j$ to $G_i$ if $j < i$. This forces $A$ to be upper-triangular, with nonzero diagonal, so an inverse always exists. Thus the linear equations always have a unique solution, and a measure exists if and only if that solution has  $\nu(\{G_j\}) \geq 0$ for all $j$.

A key observation that powers our proof is that one can explicitly calculate this inverse in great generality. In fact, the entry $A^{-1}_{ij}$ is given by the formula  $\frac{ \hat{\mu} (G_i, G_j )}{ \abs{ \Aut(G_j)}}$ (Lemma \ref{L:IE}). Using this formula, we immediately obtain a guess $v_{\C, F}$ for the unique measure with a given moments (in Equation \eqref{E:mom}). Proving that guess is correct, both that $v_{\C,F}$ produces a measure with those moments (Existence) and produces the only measure with those moments (Uniqueness) requires only that we exchange the order of summation in a certain sum.

Using our assumption that the quotients of any object form a finite modular lattice, we note that $ \hat{\mu} (G_i, G_j )$ vanishes unless there exists a semisimple epimorphism from $G_j$ to $G_i$. This allows us to restrict attention to semisimple morphisms in sums involving $\hat{\mu}$. Furthermore, we can calculate the exact value of $\hat{\mu}$ using known results on the structure of finite complemented modular lattices. 
From this we can prove the sum relevant for (Uniqueness) is absolutely convergent using the well-behavedness hypothesis, justifying the exchange of summation.

The sum relevant for (Existence), on the other hand, is trickier. It is not absolutely convergent, under well-behavedness or even much stronger assumptions, in reasonable categories like the category of finite groups. Instead we need a more intricate argument where we break each sum into a series of multiple sums and exchange them one at a time, proving absolute convergence only of certain restricted sums, which does follow from well-behavedness. 

More precisely, we can place any level $\C$ into a series of formations $\W_0 \subset \W_2 \subset \dots \subset \W_n=\C$ where $\W_0$ consists (in the category of groups) only of the trivial group or (in general) only minimal objects. We can express a sum over objects in $\C$ as a sum over objects $H_0\in \W_0$, then over objects $H_1\in \W_1$ such that $H_1^{\W_0}\isom H_0$, then over objects $H_2\in \W_2$ with $H_2^{\W_1} \cong H_1$, and so on.  Crucially, we can choose $\W_i$ to be ``close together" in the sense that for every $H \in \W_i$, the natural epimorphism $H\to H^{\W_{i-1}}$ is semisimple. This, again, lets us restrict attention to semisimple morphisms, which are controlled by the well-behavedness condition, when exchanging a given pair of sums.

The proof of (Robustness) requires, instead of an exchange of sums, the exchange of a sum with a limit. We apply a similar strategy, breaking the sum into several smaller sums, each over semisimple morphisms, and exchanging each sum with the limit in turn, using a dominated convergence argument.

The approach of attacking problems in level $\C$ by an inductive argument using a sequence of formations $\W_0,\dots, \W_n=\C$, 
where each element of $\W_i$ has a semisimple morphism to an element of $\W_{i-1}$, is something that we expect to be relevant to future problems about random groups and other random objects. As a first example, for the problem of efficiently sampling from the probability distributions we have constructed, we believe the best method will typically be to choose a random object in $\W_0$, then a random lift to $\W_1$, and so on, which reduces the problem to a series of steps of sampling from a conditional probability distribution, which will often be given by a simple probability distribution on an explicit set. For the existence problem, it is the combination of this inductive strategy with our M\"obius calculations which together allow us to restrict entirely to semisimple morphisms for the most difficult calculations in the proof.

The first step of the argument, where we find an explicit formula for the inverse matrix, is a special case of the phenomenon of M\"obius inversion in categories developed in \cite{CLL,Haigh1980,Leinster2008}. In the unifying language of \cite{Leinster2012}, $ \frac{\hat{\mu}(F,G)}{\abs{\Aut(G)}}$ is the coarse M\"obius function of the category obtained from $\C$ by keeping only the epimorphisms and ignoring all other morphisms and Lemma \ref{L:IE} is the M\"obius inversion formula. (It would be possible to derive Lemma \ref{L:IE} from \cite[Theorem 1.4]{Leinster2008} but we give a direct proof.)  Since this paper is devoted to demonstrating that this formula gives not just an inverse in the formal algebra of matrices but an inverse operator to the map sending a measure to its moments, at least in the well-behaved case, one can interpret our argument as making (in a special case) the notion of M\"obius inversion on categories analytic.

\subsection{Organization of the paper}
In Section~\ref{S:prelim},  we show basic facts about modular lattices and levels that will be used throughout the paper. 
 In Section~\ref{s-calculation}, we show how the M\"{o}bius function and $Z(\pi)$ can be given by simple formulas, and give examples in the categories of groups, rings, and modules.  In Section~\ref{S:Main}, we prove our main result, Theorem~\ref{T:general}, which operates at a single level.  In fact, we prove slightly stronger results, in which we give the minimal hypotheses separately for each piece of the theorem.  
   In Section~\ref{S:Pro}, we prove the results regarding distributions on pro-isomorphism classes (Theorems~\ref{analytic-main-measure} and \ref{T:weak-convergence}),  relate pro-isomorphism classes to classical pro-objects, and give examples in groups and modules to show what these are concretely.  In Section \ref{s-examples}, we give examples of the calculation of measures from moments, 
including showing that moments of size $O(|G|^{v})$ are well-behaved for any real $v$  in the categories of groups or $R$-modules.  
Our examples give a blueprint for how one evaluates,  in practice, the sums defining $v_{\C,F}$ and well-behavedness.
We also show many examples of diamond categories, including simple ones like (the opposites of) finite sets and finite graphs, and how various modifications to diamond categories lead to new diamond categories.

    \subsection{Further notation and conventions}

  
  We write $\category_{/\isom}$ for the set of isomorphism classes of objects of $\category$.  Throughout the paper, we sum over isomorphism classes. 
  We write $\sum_{G\in\category_{/\isom} } f(G)$, where $f$ is some function of \emph{objects} of $\category$ that is invariant under isomorphism and we intend $G$ to be an object of $\category$.
  This is a slight abuse of notation,  which more precisely means $\sum_{[G]\in\category_{/\isom} } f(G)$, where $[G]$ is an isomorphism class.
  
We write $\N$ for the non-negative integers.  
  
We write $f(x) \ll g(x)$ to mean  that there exists a constant $C$ such that $f(x)\leq Cg(x)$ for all values of $x$.  The implicit constant may depend on certain parameters, and we will say what it depends on.
  
We refer to quotients by the object, and leave the morphism implicit, whenever this does not create ambiguity.
If $F$ is a quotient of $G$, we denote the implicit epimorphism by $\pi^G_F$. We refer to the isomorphism classes of quotients by a choice of one of their elements.

Our notation $[F,G]$ for the partially ordered set of quotients of $G$ that are $\geq F$ agrees with the usual notion of an interval in lattice theory. For $H\in [F,G]$,
note that there is a unique map $H\ra F$ implied by $H\geq F$ that is compatible with the $G$-quotient structure, and this map is an epimorphism.
A priori the notation $[F,H]$ might either refer to  the interval in $[F,G]$ or the the partially ordered set of quotients of $H$ that are $\geq F$, but
we can see these two options are in natural correspondence, as a quotient of $H$ is also a quotient of $G$ (using the given map $G\ra H$), and the notions of isomorphism correspond. 
For $F,H\in \category$, we only write $F\vee H$ or $F\wedge H$ when we already have a $G$-quotient structure on $F$ and $H$ for some $G$,  and
we mean to take these operations in the partially ordered set of (isomorphism classes) of quotients of $G$.  
If we need to remind the reader what $G$ is, we write  $F\vee_G H$ or $F\wedge_G H$.

The joins and meets in the lattice guaranteed by Definition~\ref{D:diamond} \eqref{D:modular} have an interpretation in terms of products and coproducts, when the latter exist. 
When $F$ and $H$ are quotients of $G$,  and the pushout
of $F$ and $H$ over $G$ exists, then it is the meet $F\wedge H$.
When the meet $F\wedge H$ exists and the fiber product of $F$ and $H$ over $F \wedge H$ exists and is a quotient of $G$, the fiber product is the join $F \vee H$.
However, we will sometimes want to consider categories where fiber products may not exist.

Whenever we say a sum of terms is finite, that always includes the claim that there are only countably many non-zero terms in the sum.

We never consider any morphisms in our category $C$ that are not epimorphisms, so for convenience we may remove non-epimorphisms from our sets of morphisms.

\subsection{Translation of terms for the category of groups}\label{ss-translation}
We now summarize the meaning of some of the notation and results in this section and the next section in the special case where $\category$ is the category of finite groups, as a reference for readers.

For $G \to F$ a morphism of groups, the kernel is a group with an action of $G$ by conjugation, i.e. a \emph{$G$-group}. We say a $G$-group is \emph{simple} if it has no nontrivial proper normal $G$-invariant subgroup. Using this notation, we see that a morphism $G \to F$ of finite groups is {\bf simple} if and only if the kernel is a finite simple $G$-group.

A morphism $\pi \colon G \to F$ is {\bf semisimple} if and only if the kernel is a product of finite simple $G$-groups, i.e. a product of finite normal subgroups with no $G$-invariant normal subgroup. Here the ``only if" direction is because for $F$ a $G$-invariant normal subgroup of $\ker \pi$ and $H$ a complement in the sense that $F \cap H =\{1\}$ and $FH =\ker \pi$, we have $\ker \pi \cong F \times H$ as $G$-groups, and we can repeat this process until no nontrivial proper $G$-invariant normal subgroups are left, thereby writing $\ker \pi$ as a product of finite simple $G$-groups and the ``if" direction follows from Lemma \ref{semisimple-equivalence}.
Lemma \ref{semisimple-equivalence} also includes the fact that a sub-$G$-group or quotient $G$-group of a product of finite simple $G$-groups is itself the product of finite simple $G$-groups.

The {\bf dimension} of a group $G$ is the length of its Jordan-H\"older filtration.
The {\bf dimension} of a surjection $G \to F$ is the length of the initial segment of a Jordan-H\"older filtration of $G$  which includes the kernel of the surjection as one of its members. The dimension of a semisimple morphism is the number of finite simple $G$-groups we can write the kernel as a product of.

The {\bf diamond isomorphism theorem} Lemma \ref{dit} in the case of groups follows from the second isomorphism theorem.

\subsection{Acknowledgments}
The authors would like to thank Akshay Venkatesh, John Baez, Ofer Gabber, and Jiwan Jung for helpful conversations related to this project and comments on earlier versions of this manuscript.
Will Sawin was supported by NSF grant DMS-2101491 while working on this paper.
Melanie Matchett Wood was partially supported by a Packard Fellowship for Science and Engineering,  NSF grant DMS-1652116,   NSF Waterman Award DMS-2140043, and a MacArthur Fellowship.  Wood was also partially supported by the Radcliffe Institute for Advanced Study at Harvard University as a Radcliffe Fellow.

\section{Preliminaries on lattices and categories}\label{S:prelim}

Throughout this section we let $\category$ be a \Diamantine category.

The fundamental result about modular lattices is the following.

\begin{lemma}[Diamond Isomorphism Theorem, {\cite[Ch.  I, Thm. 13 on p. 13]{Birkhoff1967}}]\label{dit} If $G$ and $H$ are two elements of a modular lattice, then the intervals $[ G\wedge H, H]$ and $[G, G\vee H]$ are isomorphic as lattices. The isomorphism is given by the map $\cdot \vee G \colon  [ G\wedge H, H]\to [G, G\vee H]$ and the inverse map $\cdot \wedge H \colon  [G, G\vee H]\to  [ G\wedge H, H]$. \end{lemma}

We say an epimorphism $G \to F$ is \emph{semisimple} if the lattice $[F,G]$ is complemented (i.e. for every $H$ in $[F,G ]$, there exists $K\in [F, G]$ such that $H \vee K = G$ and $H \wedge K = F$, i.e. $K$ is a complement of $H$).   (When we refer to a morphism as simple or semisimple, we always mean that it is an epimorphism.)

\begin{remark} The definition of a complemented lattice is a natural generalization of the notion of a semisimple representation, which can be defined as representations for which each invariant subspace $V$ has an invariant complement $W$, i.e. an invariant subspace such that $V \cap W =\{0\}$ and $V + W$ is the whole space. \end{remark}

A key advantage of the modular lattice property is that it gives the set of semisimple morphisms a number of useful properties.
 A basic fact in lattice theory is the following.

\begin{lemma}\label{complemented-equivalence} Let $L$ be a finite modular lattice. Let $1$ be the maximum element and $0$ the minimum element of $L$. The following are equivalent:

\begin{enumerate}

\item $L$ is complemented.

\item Every interval in $L$ is complemented.

\item $1$ is the join of the minimal non-zero elements of $L$.

\item $0$ is the meet of the maximal non-one elements of $L$.

\end{enumerate} 

\end{lemma}

\begin{proof}
For the first three, this is \cite[VIII, Corollary to Theorem 1 on p. 114]{Birkhoff1948}, and the equivalence between (1) and (4) follows from the equivalence between (1) and (3) and the fact that the complemented modular lattices are manifestly stable under duality. \end{proof}

Translated into morphisms,  Lemma~\ref{complemented-equivalence} immediately gives the following.

\begin{lemma}\label{semisimple-equivalence} Let $\pi \colon G \to F$ be an epimorphism in $\category$. The following are equivalent:

\begin{enumerate}

\item $\pi$ is semisimple.

\item For every factorization $\pi = \pi_1 \circ \pi_2$ with $\pi_1,\pi_2$ epimorphisms, $\pi_1$ and $\pi_2$ are semisimple.

\item $G$ is the join of $F_1,\dots, F_n$ with $F_i \to F$ simple.

\item $F$ is the meet of $G_1,\dots, G_n$ with $G \to G_i$ simple.

\end{enumerate}

\end{lemma}

For fixed $H\leq G,$ (i.e. for $H$ a quotient of $G$ via $\pi^G_H$),
\begin{equation}\label{E:Mobdef}
\sum_{F\in [H,G]} \mu(H,F)=\sum_{F\in [H,G]} \mu(F,G)=\begin{cases}
1 &\textrm{if $G=H$}\\
0 &\textrm{if $G\ne H$},
\end{cases}
\end{equation}
by the defining property of the M\"{o}bius function and another well-known property (\cite[Proposition 3.7.1]{Stanley2012} for $f$ the characteristic function of $H$).

Semisimple morphisms play a crucial role in the study of the M\"obius function.

\begin{lemma}[Philip Hall]\label{non-semisimple-vanishing}
Let $G$ be an object of $\category$ and $F$ a quotient of $G$.  If $G\ra F$ is not semisimple, then $\mu(F,G)=0$. 
\end{lemma}

\begin{proof} 
This follows from \cite[Corollary p.  349]{Rota1964}
 but we give the simple proof below.

  Let $\Lambda$ be the sublattice of  quotients $H$ of $G$ such that $G\ra H$ is semisimple.  
  By Lemma~\ref{semisimple-equivalence}, we have that $\Lambda$ is meet-closed, and thus for every quotient $H$ of $G$, 
  there exists $H^{\operatorname{ss}},$ the minimal semisimple quotient of $G$ that is $\geq H$.  
Then we have
\begin{align*}
0 &=\sum_{H\in \Lambda} \mu(H,G) \sum_{K \in [F,H]} \mu(F,K) &\textrm{(each interior sum vanishes by defn. of $\mu$)}\\
 &=\sum_{K \in [F,G]}\mu(F,K)  \sum_{\substack{H\in \Lambda \cap [K,G]\\\textrm{i.e., } H\in [K^{\operatorname{ss}},G]}} \mu(H,G)   &\textrm{(rearrange sum)}\\
  &=\mu(F,G),
\end{align*}  
where the last equality holds because the inner sum vanishes unless $K^{\operatorname{ss}}=G$
by \eqref{E:Mobdef},
 and  $K^{\operatorname{ss}}=G$ if and only if $K=G$
  (as if $K<G$ then $K$ is less than some simple quotient of $G$).
 \end{proof}

Let $\pi \colon G\to F$ be an epimorphism. Let $H$ be the join of all the minimal non-$F$ elements in the interval $[F,G]$. We call the epimorphism $G \to H$ the \emph{radical} of $\pi$. Since $H$ is the join of the minimal non-$F$ elements, $H \to F$ is semisimple, and it is the maximal semisimple morphism through which $\pi$ factors.  We call $H\ra F$ the \emph{semisimplification} of $\pi$.

The finiteness in Definition~\ref{D:diamond} \eqref{D:modular}  
ensures that we can write any epimorphism $\pi$ as a finite composition of simple morphisms. The modularity assumption implies that, when we do this, the number of simple morphisms depends only on $\pi$ \cite[V, Theorem 3 on p. 68]{Birkhoff1948}. 

\begin{definition} We define the \emph{dimension} of $\pi$ to be the number of simple morphisms it is a composition of 
(or $0$ if $\pi$ is an isomorphism)
and the dimension of $G$ to be the dimension of its morphism to its unique minimal quotient.\end{definition}

Note the dimension of an epimorphism $G \to H$ is the difference between the dimension of $G$ and the dimension of $H$. 
Also, if $A< B$, them $\dim A <\dim B$.

We say that a formation $\W$ is \emph{narrow} if for each $G\in\W$, there are finitely many elements of $\W$ with a simple epimorphism to $G$.  Since we work in a diamond category, a level is a narrow formation. 
As an example, if $C$ is the category of finite abelian groups,  the formation of finite abelian $p$-groups is narrow, but the formation of all finite abelian groups is not narrow.

We will use many inductive arguments, sometimes needing to induct on dimension.

\begin{lemma}\label{dim-C-finite} For a narrow formation $\W$, and $F\in\W$, and $d$ a natural number, there are finitely many $G\in \W$ with 
an
 epimorphisms $G \to F$ of dimension $d$. \end{lemma}

\begin{proof} This follows by induction on $d$, using the definition of narrow and the fact that any epimorphism $G \to F$ of dimension $d$ factors (in at least one way) as $G \to G' \to F$ with $G \to G'$ of dimension 1 (i.e. simple) and $G' \to F$ of dimension $d-1$. \end{proof}

The following will be very useful in inductive arguments, in which we will write morphisms as compositions of semisimple morphisms.

\begin{lemma}\label{one-less-semisimple} 
Let $\W$ be a formation.
The set $\W'$ of isomorphism classes of $G$ such that $G \to G^{\W}$ is semisimple is a formation.\end{lemma}

\begin{proof} 
First we show that $\W'$ is downward-closed.  
Let $G \to  G^{\W}$ be semisimple and $H$  a quotient of $G$.  We have $ G^{\W} \wedge H \in {\W}$ and is a quotient of $H$, and thus is $\leq H^{\W}$. So  $[H^{\W}, H ]$ is an interval of of  $[G^{\W} \wedge H, H] $. By  Lemma \ref{dit}, 
$[G^{\W} \wedge H, H] $ is isomorphic to $[G^{\W} , G^{\W} \vee H]$, which is an interval of the complemented modular lattice $[G^{\W},G]$. Since intervals in complemented modular lattices are complemented (Lemma \ref{complemented-equivalence}),  $[H^{\W}, H ] $ is complemented and thus $H \to H^{\W}$ is semisimple.

Next we show that $\W'$ is join-closed.
Suppose $G \to  G^{\W}$ and $H \to H^{\W}$ are semisimple. and $G$ and $H$ are quotients of some common object.  We have $G^{\W} \vee H^{\W}$ is in $\W$ and is a quotient of $G \vee H$, and thus is $\leq  (G \vee H)^{\W}$. Thus by Lemma \ref{semisimple-equivalence}, it suffices to show $G \vee H \to G^{\W} \vee H^{\W}$ is semisimple. 

We have $[ G^{\W} \vee H^{\W},  H^{\W} \vee G] = [  G^{\W} \vee H^{\W},  G^{\W} \vee H^{\W}\vee G]$, which, by Lemma \ref{dit} is isomorphic to $[ ( G^{\W} \vee H^{\W}) \wedge G, G]$.
The interval $[ ( G^{\W} \vee H^{\W}) \wedge G, G]$
 is an interval in $[G^{\W}, G]$ and thus is complemented. 
Thus $[ G^{\W} \vee H^{\W},  H^{\W} \vee G], $ and similarly 
 $[ G^{\W} \vee H^{\W},  G^{\W}\vee H]$ are complemented. So both of $G\vee H^{\W}$ and $G^{\W}\vee H$ are the join of minimal non-$G^{\W} \vee H^{\W}$ elements (by Lemma \ref{complemented-equivalence}), which all lie in $[ G^{\W} \vee H^{\W}, G\vee H]$, and thus the join of all such elements is $\geq (G\vee H^{\W}) \vee (G^{\W}\vee H) = G\vee H$ and thus must be equal to $G \vee H$. This makes $[ G^{\W} \vee H^{\W}, G\vee H]$ complemented and $G \vee H \to G^{\W} \vee H^{\W}$ semisimple by Lemma \ref{complemented-equivalence}.

Clearly, all minimal $G$ have $G \to G^{\W}$ semisimple, and this proves the lemma.
\end{proof}

We have some basic facts about the interaction of lattice operations of taking the $\W$ quotient.

 \begin{lemma}\label{L:Cvee}
 For any formation $\W$ and $A\leq B$ objects of $\category$, we have 
\begin{enumerate}
\item $A^\W \leq B^\W$,
\item $(A\vee B^\W)^\W=B^\W$,  and
\item $A \wedge B^\W =A^\W$.
\end{enumerate}
 \end{lemma}
\begin{proof}
Since $A^\W \leq A \leq B$, and $A^\W\in\W$, we have $A^\W \leq B^\W$ by definition of $B^\W$.
Thus $B^\W=(B^\W)^\W \leq (A\vee B^\W)^\W \leq B^\W,$ which gives the second statement.
We have $A^\W\leq A,B^\W$, so $A^\W\leq A \wedge B^\W$.  Since $A \wedge B^\W\leq B^\W$, we have that 
$A \wedge B^\W\in\W$, using that $\W$ is downward-closed.  Then since $A \wedge B^\W \leq A$, we have that  $A \wedge B^\W \leq A^\W$,
and the third statement follows.
\end{proof} 
 
 \begin{definition}
For $G\in \category$, we define the \emph{complexity} of $G$ to be the minimal $k$ such that
 the morphism $G\ra M$ for a minimal object $M$ can be written as a concatenation of $k$ semisimple morphisms.
 If $G$ is minimal, then it has complexity $0$.  We say a formation $\W$ has \emph{finite complexity} if there is a finite upper bound on the complexity of the objects in $\W$, and \emph{complexity $k$} if $k$ is the maximum complexity of objects in $\W$.
 \end{definition}
 
 For example if $C$ is the category of finite abelian groups, then $\Z/p^n\Z$ for a prime $p$ is complexity $n$, while $(\Z/p\Z)^n$ has complexity $1$.
 
 \begin{lemma}\label{L:complexity}
The set $\V_k$ of isomorphism classes of complexity $\leq k$ is a formation.
 For $G\in\category$ and $k\geq 1$, we have $G\in \V_k$ if and only if $G\ra G^{\V_{k-1}}$ is semisimple.
 \end{lemma}
\begin{proof}
We prove the statement by induction on $k$ and it is clear for $k=0$.  Suppose the lemma is true for $k$.

Let $G\in \V_{k+1}$.  Then by definition, $G$ has a semisimple morphism $G\ra H$ to an $H$ of complexity $k$.
Since $H$ is a quotient of $G$ and $h\in \V_k$, we have that $G\ra H$ factors through $G\ra G^{\V_k}$.
Thus by Lemma~\ref{semisimple-equivalence} (2), we have that $G\ra G^{\V_k}$ is semisimple.

If $G\ra G^{\V_{k}}$ is semisimple, then $G$ by definition has complexity $\leq {k+1}$.  Thus the second statement of the lemma is true for $k+1$.  Then by Lemma~\ref{one-less-semisimple} we have that $\V_{k+1}$ is a formation, completing the induction.
\end{proof} 
 
 \begin{corollary}\label{C:levelcom}
Any level $\C$ has finite complexity.
 \end{corollary}
 \begin{proof}
 If $\C$ is generated by objects $G_1,\dots, G_n$, let $k$ be the maximal complexity of any $G_i$.
 Then all $G_i\in \V_k$, and since $\V_k$ is a formation by Lemma~\ref{L:complexity}, we have $\C\sub\V_k$.
 \end{proof}
 
 Lastly, we have a simple observation about the concept of a well-behaved tuple.
 
 \begin{lemma}\label{big-O-wb} Let $(M_G)_G$ and $(M_G')_G$ be tuples of nonnegative real numbers such that $M_{G'} = O (M_G)$ for all $G$, i.e. there exists a constant $c$ such that $M_{G'} \leq c M_G$ for all $G$.
 
 If $(M_G)_G$ is well-behaved at level $\C$, then $(M_G')_G$ is as well, and if $(M_G)_G$ is well-behaved, then $(M_G')_G$ is as well. 
 \end{lemma}
 
\begin{proof} This follows immediately from the definition since
\[  \sum_{G \in \C}     \sum_{ \pi \in\Epi(G,F) }\frac{\abs{{\mu}(F,G)}}{|\Aut(G)|} Z ( \pi)^3 M_G' \leq   \sum_{G \in \C}     \sum_{ \pi \in\Epi(G,F) }\frac{\abs{{\mu}(F,G)}}{|\Aut(G)|} Z ( \pi)^3 c M_G \] \[= c \sum_{G \in \C}     \sum_{ \pi \in\Epi(G,F) }\frac{\abs{{\mu}(F,G)}}{|\Aut(G)|} Z ( \pi)^3  M_G < \infty. \qedhere\] \end{proof}

 \section{Calculation of the M\"obius function}\label{s-calculation}

In this section, we explain how to calculate the M\"obius function so as to obtain explicit formulas from our main theorems in particular cases. We also explain how to estimate the sum appearing in the definition of well-behaved, so as to check the well-behavedness assumption.  In view of Lemma \ref{non-semisimple-vanishing}, we may restrict attention to semisimple morphisms $\pi \colon G \to H$. For such morphisms, $[H,G]$ is by definition a finite complemented modular lattice. There is a powerful classification theorem for finite complemented modular lattices. We begin by reviewing the classification theorem, then explain how to compute the M\"obius function in each case, and finally show how to apply these formulas in the categories of finite groups, finite modules, and finite rings.

We begin by giving some examples of finite complemented modular lattices.

\begin{example}\label{subset-example} The simplest example of a complemented modular lattice is the lattice of subsets of a set, with the ordering given by inclusion. The complements here are the usual notion of the complement of the set. The subsets of a finite set form a finite complemented modular lattice.\end{example}

\begin{example}\label{subspace-example} Another example of a complemented modular lattice is the subspaces of a vector space, again ordered by inclusion. The complements of a subspace are the usual notion of complement in linear algebra. The subspaces of a finite-dimensional vector space over a finite field 
form a finite complemented
modular lattice.\end{example}

There is a relation between Examples \ref{subset-example} and \ref{subspace-example}: A one-dimensional vector space over a field has only two subspaces, itself and the zero subspace, so the lattice of subspaces has two elements. The product of $n$ copies of this two-element lattice is the lattice of subsets of an $n$-element set. A construction of finite complemented modular lattices which generalizes both of the previous two cases is to take a finite product of a sequence of lattices, each element of which is the lattice of subspaces of some vector space over some finite field.

Not all finite complemented modular lattices arise this way, so we need two additional constructions.

\begin{example} First, note that the lattice of subspaces of $\mathbb F_q^2$ consists of a minimal element (the zero-dimensional subspace), a maximal element (the two-dimensional subspace), and $q+1$ incomparable elements (the one-dimensional subspaces). We can generalize this by taking the lattice consisting of a minimal element, a maximal element, and $q+1$ incomparable elements, where $q$ is not necessarily a prime power.\end{example}

\begin{example}\label{plane-example} Second, note that the lattice of subspaces of $\mathbb F_q^3$ consists of a minimal element, a maximal element, and elements corresponding to one-dimensional and two-dimensional subspaces. The one-dimensional subspaces are naturally in bijection with points of the finite projective plane $\mathbb P^2(\mathbb F_q)$, and the two-dimensional subspaces are naturally in bijection with lines in this plane. We can generalize this by considering finite projective planes which do not necessarily arise from vector spaces over finite fields. Since these are exactly the finite projective planes that don't satisfy Desargues's theorem, they are called non-Desarguesian planes. The smallest examples have $91$ points and $91$ lines.

We form a lattice from a finite projective plane by choosing the set of elements of the lattice to consist of one minimal element, one maximal element, the set of points of the plane, and the set of lines of the plane, with a partial ordering defined so that the minimal element is $\leq$ every other element, the maximal element is $\geq$ every other element, a point is $\leq$ a line if the point lies on the line, and every other pair of elements is incomparable.

Every finite projective plane, including non-Desarguesian ones, has a number $q$, called the order, so that every point lies on $q+1$ lines, every line contains $q+1$ points, the number of points is $q^2+q+1$, and the number of lines is $q^2+q+1$ \cite[Theorem 1 on p. 17]{FPP}. \end{example}

These constructions, combined with subspaces of a finite projective space, can be produce to construct every finite complemented modular lattice.

\begin{theorem}[{\cite[VIII, Theorem 6 on p. 120]{Birkhoff1948}}]\label{fcml-classification} Every finite complemented modular lattice can be expressed uniquely as a finite product of lattices of one of the following three forms:

\begin{itemize}

\item The lattice of linear subspaces of an $n$-dimensional space over a finite field $\mathbb F_q$, for some $n \geq 1$ and prime power $q$.

\item A lattice consisting of one minimal element, one maximal element, and $q+1$ incomparable elements, for $q \geq 2$ not a prime power.

\item A lattice consisting of one minimal element, one maximal element, the set of points in $P$, and the set of lines in $P$, for $P$ a non-Desarguesian finite projective plane, with $\leq$ defined as in Example \ref{plane-example}.

\end{itemize}

\end{theorem}

For $\pi \colon G \to H$ semisimple, let $\omega (\pi)$ be the number of terms appearing in this product representation of $[H,G]$, and we call
these the \emph{simple factors} of the lattice.

\begin{lemma}\label{L:muvalues}
 Let $L$ be a finite complemented modular lattice, and write $L$ as a product $\prod_{i=1}^{\omega} L_i$ as in Theorem \ref{fcml-classification}.

Define $q_1,\dots, q_{\omega}$ and $n_1,\dots, n_\omega$ as follows: For each $i$, if $L_i$ is the lattice of linear subspaces of $\mathbb F_{q^n}$, set $q_i=q$ and $n_i=n$. (We can take $q_i$ arbitrary if $n_i=1$.) If $L_i$ has one minimal element, one maximal element, and $q+1$ incomparable elements, set $q_i=q$ and $n_i=2$. If $L_i$ is obtained from a finite projective plane of order $q$, set $q_i=q$ and $n_i=3$.

Then for $0$ and $1$ the minimal and maximal elements of $L$, we have
\[ \mu(0,1) = \prod_{i=1}^\omega  (-1)^{n_i} q_I^{\binom{n_i}{2}} .\]
\end{lemma}

Note that if $n_i=1$ then the formula does not depend on $q_i$.

\begin{proof} Since M\"obius function is multiplicative in a product of posets, we may reduce to the case $\omega=1$, so assume $L$ is one of the three forms of Theorem \ref{fcml-classification}. If $L$ is the lattice of subspaces of a linear space, this is a classical result of Philip Hall \cite[Equation on p. 352]{Rota1964}.

In the other two cases, this is a straightforward explicit calculation, and is likely also classical.
\end{proof}

Let us now see how this classification applies to the categories of groups, rings, and modules (which we will check in Section \ref{s-examples} each form a \Diamantine category). 

\begin{lemma}\label{module-semisimple} Let $M$ be a module for a ring $R$. The lattice of submodules of $M$ is complemented if and only if $M$ is a product of simple $R$-modules.

In this case, there must exist finite simple $R$-modules $M_1,\dots, M_\omega$, pairwise non-isomorphic, and natural numbers $n_1,\dots,n_\omega \geq 1$, such that $M \cong \prod_{i=1}^\omega M_i^{n_i}$. Let $q_i$ be the order of the field of endomorphisms of $M_i$. Then the lattice of submodules of $M$ is the product from $i$ from $1$ to $\omega$ of the lattice of subspaces of $\mathbb F_{q_i}^{n_i}$. The dual of the lattice of submodules is also isomorphic to the same product. \end{lemma}

\begin{proof}By Lemma \ref{complemented-equivalence}, the lattice of quotients of $M$ is complemented if and only if $M$ is a join of simple modules, i.e. a quotient of the product of simple modules, but every quotient of a (finite) product of simple modules is itself a product of simple modules by induction.

If $M$ is a product of simple modules, then we can take one representative from each isomorphism class appearing in the product to write $M \cong \prod_{i=1}^\omega M_i^{n_i}$. We can calculate the lattice of submodules of $\prod_{i=1}^\omega M_i^{n_i}$ by noting that, by Jordan-H\"older, each submodule is isomorphic to $ \prod_{i=1}^\omega M_i^{e_i}$ for some $e_i \leq n_i$, that injections $ \prod_{i=1}^\omega M_i^{e_i} \times  \prod_{i=1}^\omega M_i^{n_i}$ are parameterized by tuples of full rank $n_i \times e_i$ matrices over $\mathbb F_{q_i}$, and two matrices give the same submodule if and only if they can be related by elements of $\prod_{i=1}^\omega GL_{e_i}(\mathbb F_{q_i})$, i.e. if and only if they have the same image.

Finally, it is easy to check that this product lattice is isomorphic to its dual, by picking a nondegenerate bilinear form on $\mathbb F_q^{e_i}$.\end{proof}

\begin{lemma}\label{module-semisimple-description} Let $\category$ be the category of finite modules over a fixed ring $R$. Then a morphism $\pi \colon M \to N $ of finite $R$-modules is semisimple if and only if $\ker \pi$ is a product of simple $R$-modules.

In this case, there must exist finite simple $R$-modules $M_1,\dots, M_\omega$, pairwise non-isomorphic, and natural numbers $n_1,\dots,n_\omega \geq 1$, such that $\ker \pi \cong \prod_{i=1}^\omega M_i^{n_i}$. Let $q_i$ be the order of the field of endomorphisms of $M_i$. Then the lattice $[N,M]$ is the product from $i$ from $1$ to $\omega$ of the lattice of subspaces of $\mathbb F_{q_i}^{n_i}$. Thus $ \mu(N,M) = \prod_{i=1}^\omega  (-1)^{n_i} q_i^{\binom{n_i}{2}} $.

\end{lemma}

\begin{proof} Every module  $K \in [N,M]$ defines a submodule of $\ker \pi$, the kernel of the natural morphism $M \to K$. Conversely, every submodule $ S\subseteq\ker \pi$ gives a module $M/S \in [N,M]$. This bijection sends inclusion of submodules to the $\geq$ relation on $[N,M]$, so $\pi$ is semisimple if and only if the lattice of submodules of $\ker \pi$ is complemented. By Lemma \ref{module-semisimple}, this occurs only if $\ker \pi$ is a product of finite simple $R$-modules. The lattice $[N,M]$ is dual to the lattice of submodules of $\ker \pi$, which by Lemma \ref{module-semisimple} has the stated description. The M\"obius formula then follows from Lemma \ref{L:muvalues}

\end{proof}

\begin{lemma}\label{ring-semisimple-description} Let $\category$ be the category of finite  
rings. Then a morphism $\pi \colon R \to S$ of finite
rings is semisimple if and only if $\ker \pi$ is a product of finite simple $R$-bi-modules.

In this case, there must exist finite simple $R$-bi-modules $M_1,\dots, M_\omega$, pairwise non-isomorphic, and natural numbers $n_1,\dots,n_\omega \geq 1$, such that $\ker \pi \cong \prod_{i=1}^\omega M_i^{n_i}$. Let $q_i$ be the order of the field of endomorphisms of $M_i$. Then the lattice $[S,R]$ is the product of $i$ from $1$ to $\omega$ of the lattice of subspaces of $\mathbb F_{q_i}^{n_i}$.  Thus $ \mu(S,R) = \prod_{i=1}^\omega  (-1)^{n_i} q_i^{\binom{n_i}{2}} $.

In particular, if $\pi$ is a local morphism of finite local commutative rings with residue field $\kappa$, then $\pi$ is semisimple if and only if $\ker \pi \cong \kappa^n$ for some $n$, and, then $[S,R]$ is isomorphic to the lattice of subspaces of $\kappa^n$, so $\mu(S,R) = (-1)^n \abs{\kappa}^{ \binom{n}{2}}$.

\end{lemma}

\begin{proof} Any ring $T \in [S,R]$ defines a sub-$R$-bi-module of $\ker \pi$, its kernel, and conversely, any sub-$R$-bi-module of $\ker \pi$ gives a two-sided ideal of $R$, whose quotient is a ring in $[S,R]$. This bijection sends inclusion of submodules to the $\geq$ relation on $[S,R]$, so $\pi$ is semisimple if and only if the lattice of sub-bi-modules of $\ker \pi$ is complemented. Since a bi-module is simply an $R \otimes R^{op}$-module, we can apply Lemma \ref{module-semisimple} to show that this happens if and only if $\ker \pi$ is a product of simple bi-modules. The lattice $[S,R]$ is dual to the lattice of submodules of $\ker \pi$, which by Lemma \ref{module-semisimple} has the stated description, and the M\"obius formula then follows from Lemma \ref{L:muvalues}

For commutative rings, the left and right actions of $R$ on $R$ are identical, so the same is true for $\ker \pi$, so we can express $\ker \pi$ and every submodule of it as modules rather than bimodules. For local commutative rings with residue field $\kappa$, the unique simple module is $\kappa$, and so the expression $\prod_{i=1}^\omega M_i^{n_i}$ simplifies to $\kappa^n$, and similarly for the M\"obius expression.\end{proof}

\begin{lemma}\label{group-semisimple-description} Let $\category$ be the category of finite groups. Then a morphism $\pi \colon G \to H$ of finite groups is semisimple if and only if $\ker \pi$ is a product of finite simple $G$-groups.

In this case, there must exist pairwise non-isomorphic finite simple abelian $G$-groups, $V_1,\dots, V_n$,  and finite simple non-abelian $G$-groups $N_1,\dots, N_m$,  as well as  exponents $e_1,\dots, e_n \geq 1$, such that $\ker \pi \cong \prod_{i=1}^n  V_i^{e_i} \times \prod_{i=1}^m N_i$. Let $q_i$ be the field of endomorphisms of $V_i$ for $i$ from $1$ to $n$. Then the lattice $[H,G]$ is the product for $i$ from $1$ to $n$ of the lattice of subspaces of $\mathbb F_{q_i}^{e_i}$ together with $m$ copies of the two-element lattice.

Thus $\mu(H, G) = (-1)^m \prod_{i=1}^n (-1)^{e_i} q_i^{ \binom{e_i}{2}}$.\end{lemma}

\begin{proof} 
The first claim is shown in Section~\ref{ss-translation}.

To obtain an expression in terms of a product $\prod_{i=1}^n  V_i^{e_i} \times \prod_{i=1}^m N_i$, we take one representative from each isomorphism class of simple $G$-groups appearing in the product, and divide into abelian and nonabelian. The nontrivial fact to check is that each isomorphism class of nonabelian $G$-groups can only occur once. Indeed, if $K_1 \times K_2$ is a sub $G$-group of $\ker \pi$ with $K_1,K_2$ isomorphic, then since the action of $K_1\subset G$ by conjugation on $K_2$ is trivial, then the action of $K_1$ on $K_1$ by conjugation must be trivial, forcing $K_1$ to be abelian, so $K_2$ is abelian as well. 

To calculate the lattice, we note that every sub-$G$-group of  $\prod_{i=1}^n  V_i^{e_i} \times \prod_{i=1}^m N_i$ is the product for each $i$ from $1$ to $n$ of a sub $G$-group of $V_i^{e_i} $ with the product for each $i$ from $1$ to $m$ of either $N_i$ or the trivial group \cite[Lemma 5.5]{Liu2020}. It follows the lattice is the product of, for each $i$ from $1$ to $n$, the lattice of sub $G$-groups of $V_i^{e_i} $, with $m$ copies of the two-element lattice. By Lemma \ref{module-semisimple} applied to $\mathbb Z[G]$, the lattice of sub-$G$-groups of $V_i^{e_i} $ is the lattice of subspaces of $\mathbb F_{q_i}^{e_i}$, verifying the stated description of the lattice.

The M\"obius formula then follows from Lemma \ref{L:muvalues}.
\end{proof}

We now explain the meaning of $Z(\pi)$ and how it can be used to bound the sum we will need.

\begin{lemma}\label{Z-omega}  For $\pi \colon G \to H$ a semisimple epimorphism, we have $Z(\pi) = 2^{\omega(\pi)}$. \end{lemma}

\begin{proof}
In a product of lattices $L_1 \times L_2$, an element $(F_1,F_2) \in L_1 \times L_2$ satisfies the distributive law in the definition of $Z(\pi)$ if and only if $F_1$ and $F_2$ do, since both joins and meets in a product are computed elementwise.   Thus it suffices to show that $Z(\pi)=2$ for each of the lattices appearing in Theorem~\ref{fcml-classification}.  This can easily be checked as $0$ and $1$ satisfy the law, but for any other $H$ we can take $K_1$ and $K_2$ to be two different complements of $H$ and this violates the distributive law.
\end{proof}

\begin{lemma}\label{Z-to-convolution}

Let $\category$ be a \Diamantine category.

 \begin{enumerate}
\item \label{I:ssuseful}
 For $\pi \colon G \to H$ semisimple,    
we have
\[  \sum_{K \in [H,G]} \abs{\mu(K, G) }  \leq \abs{\mu(H, G) } Z(\pi)^3.\]

\item  If $\W$ is a formation such that for any semisimple morphism $G \to H$ from $G\in\W$,  the number of simple factors in the lattice $[H,G]$ is bounded by a bound depending only on $H$ and $\W$, then for $\pi \colon G \to H$ with $G\in \W$, we have 
\[  \sum_{K \in [H,G]} \abs{\mu(K, G) }  \ll \abs{\mu(H, G) } \]
with the implicit constant depending only on $H$ and $\W$.

\item  If $\W$ is a formation such that for any semisimple morphism $G \to H$ from $G\in W$, the number of simple factors in the lattice $[H,G]$ which are not isomorphic to the two-element lattice is bounded by a bound depending only on $H$ and $\W$, then for $\pi \colon G \to H$ with $G\in W$, we have 
\[  \sum_{K \in [H,G]} \abs{\mu(K, G) }  \ll \abs{\mu(H, G) } Z(\pi)  \]
with the implicit constant depending only on $H$ and $\W$.

\end{enumerate}
\end{lemma}

From Lemmas \ref{module-semisimple-description} and \ref{mod-diamond}, and 
Lemma  \ref{ring-semisimple-description}, 
one can check that the condition of case (2) happens for finite modules over a ring 
 and finite local commutative rings. Similarly, we will check below in Lemmas \ref{weird-fg} and \ref{weird-fr} that case (3) holds for arbitrary finite rings and for finite groups:

The advantage of these statements is that they allow us to lower the exponent of $Z(\pi)$ in the hypotheses of our existence result.

\begin{definition}
For a formation $\W$, we say a tuple $(M_G)_{G}\in \R^{\W}$   is \emph{behaved} at $\W$ if for all $F \in \C$, 
$$
\sum_{G \in \W}     \sum_{ \pi \in\Epi(G,F) }\frac{|{\mu}(F,G)|}{|\Aut(G)|} Z ( \pi)^a |M_G| <\infty,$$ 
where either $a=3$ (in which case we say it it also \emph{well-behaved});  $\W$ satisfies the hypotheses of Lemma~\ref{Z-to-convolution} (2) and $a=0$; 
or $\W$ satisfies the hypotheses of Lemma~\ref{Z-to-convolution} (3) and $a=1$.
\end{definition}

\begin{proof}We handle case (1) first. Both sides depend only on the lattice $[H,G]$ and are multiplicative under direct products of lattices. Since $[H,G]$ is complemented modular, we may reduce to checking each of the three forms of Theorem \ref{fcml-classification}.

In the case of the lattice of subspaces of a $n$-dimensional vector space over $\mathbb F_q$, the left side is the sum over $d$ from $0$ to $n$ of the number of $(n-d)$-dimensional subspaces of $\mathbb F_q$ times $q^{ \binom{d}{2}}$, and the right side is $2^3 q^{\binom{d}{2}}$ by Lemmas \ref{L:muvalues} and \ref{Z-omega}. Since the number of subspaces of dimension $d$ is $\binom{n}{d}_q$, using the $q$-binomial theorem, the left side is

\[ \sum_{d=0}^n \binom{n}{d}_q q^{ \binom{d}{2}} = q^{ \binom{n}{2} } \prod_{j=0}^{n-1}  (1 + q^{-j}) \leq q^{ \binom{n}{2} } \prod_{j=0}^{\infty}  (1 + q^{-j}) \leq   q^{ \binom{n}{2} } \prod_{j=0}^{\infty}  (1 + 2^{-j})  <   q^{ \binom{n}{2} } 8 ,\]
giving the desired inequality. 

For the other two cases, we note that the number of points of the lattice of each dimension is given by the same formula in terms of $n=2,3$ and $q$, and the M\"obius function is given by the same formula in terms of $q$ and the dimension, so the same calculation applies. This finishes case (1).

In case (2), our assumption implies that $Z(\pi)$ is bounded in terms of $H$ and $\C$, and so (2) follows immediately from (1).

In case (3), we return to the argument of (1). We above showed that for any simple complemented modular lattice,
$$\sum_{x\in[0,1]} |\mu(x,1)|\leq 8 \abs{\mu(0,1)}.$$
For the two-element lattice,
 the  statement above is true with a constant of $2$ instead of $8$. Multiplying these two statements, we see that
$$  \sum_{K \in [H,G]} \abs{\mu(K, G) }  \leq \abs{\mu(H, G) } Z(\pi) 4^b,  $$ 
where $b$ is number of simple factors of $[H,G]$ with more than two elements.  By assumption, $4^b$ can be absorbed into our implicit constant.
\end{proof}

 \begin{lemma}\label{weird-fg} Case (3) of Lemma \ref{I:ssuseful} holds for the category of finite groups. \end{lemma}
 
 \begin{proof} Applying Lemma \ref{group-semisimple-description} and adopting its notation, one sees that the number of simple factors of $[H,G]$ which are not isomorphic to the two-element lattice is at most the number $n$ of isomorphism classes of simple abelian $G$-groups $V_i$ arising as direct factors of $\ker \pi$. The action of $G$ on $V_i$ is by conjugation which since $V_i$ is an abelian direct factor of $\ker \pi$ implies that $\ker \pi$ acts trivially on $V_i$ and thus $V_i$ is an $H$-group. Since $V_i$ is a subgroup of $G$ and $G\in \C$, the characteristic of $V_i$ must divide the order of some generator of $\C$ and thus must lie in a finite set. Since the number of isomorphism classes of irreducible finite simple abelian $H$-groups with characteristics in a finite set is bounded depending only on $H$ and that set of characteristics, condition (3) is satisfied. \end{proof}
 
  \begin{lemma}\label{weird-fr} Case (3) of Lemma \ref{I:ssuseful} holds for the category of finite rings. \end{lemma}

\begin{proof} Applying Lemma \ref{ring-semisimple-description} and adopting its notation, one observes that the number of isomorphism classes of finite simple $R$-bimodules that factor through $S$ is the number of finite simple $S$-bimodules and thus is bounded depending only on $S$.  The finite simple $R$-bimodules $M_i$ that occur with multiplicity $1$, i .e. satisfy $n_i=1$, contribute the lattice of subspaces of $\mathbb F_{q_i}$ which is a two-element lattice and can be ignored. So it suffices to show that these two classes contain all the $M_i$. 
 
 Observe that if we write $\ker \pi$ as a direct product $A_1 \times B$ of $R$-bimodules then for $a\in A, b\in B$, we have $ab \in A$ since $A \times B$ is a direct product of right modules and $ab\in B$ since $A \times B$ is a direct product of right modules so $ab=0$. Thus the action of $B$ on the module $A$ is trivial. So if we write $\ker \pi$ as a direct product $A_1 \times A_2 \times B$, the action of $A_2 \times B$ on $A_1$ is trivial and the action of $A_1 \times B$ on $A_2$ is trivial. But now if $A_1$ are both isomorphic to a fixed module $M$, then the actions of $A_1 \times B$ and $A_2 \times B $   on $ M$ are both trivial, so the action of $A_1 \times A_2 \times B= \ker \pi$ on $M$ is trivial, i.e. $M$ factors through $S$. It follows that each $M_i$ with $n_i \geq 2$ factors through $S$, as desired. \end{proof}

The assumption of part (3) of Lemma \ref{Z-to-convolution} is satisfied for $\W$ a level in
 most examples arising in ordinary mathematics, such as groups, rings, and modules. However, there are cases where it is not satisfied. We give an example:

\begin{example} Let $\category$ be the category whose objects are pairs consisting of a finite-dimensional vector space $V$ over $\mathbb F_2$ and a set $S$ of nontrivial subsets of $V$ such that the natural map $\bigoplus_{W\in S} W \to V$ is an isomorphism, where morphisms $(V_1,S_1) \to (V_2, S_2)$ are given by maps $f \colon V_1 \to V_2$ such that $f(W) \in S_2 \cup\{ 0\}$ for all $W \in S$ and if $f(W)=f(W')$ then $W=W'$ or $f(W) =0$.

We can think of these as graded vector spaces, but instead of graded by integers, they are graded by an arbitrary finite set. 

It is not so hard to check a morphism is an epimorphism if and only if $f$ is surjective, and the poset of quotients of $(V,S)$ is given by the product over $W \in S$ of the poset of quotients of $W$, and thus is a finite modular lattice. One can similarly check the other assumptions: 
 automorphism groups are subgroups of $GL_n(\mathbb F_2)$ and thus finite, and a simple morphism either raises the dimension of one element of $S$ by one or adds a new one-dimensional summand in $S$, so there are finitely many simple morphisms.

By our description of the lattice of quotients, we say every object is a meet of its one-dimensional quotients, and every one-dimensional object of $\category$ is isomorphic to $(\mathbb F_2,\{\mathbb F_2\})$, so every object is in $\C$ for $\C = \{ (\mathbb F_2,\{\mathbb F_2\})\}$. Similarly, this meet property implies every morphism is semisimple.

Thus the assumption of part (3) of Lemma \ref{Z-to-convolution} is not satisfied even for $H = (0,\emptyset)$, since $[G,H]$ can be a product of an arbitrary finite number of lattices of quotients of finite-dimensional vector spaces over $\mathbb F_2$. 

\end{example}

 \section{Proof of the main theorem}\label{S:Main}
 
 In this section, we prove Theorem \ref{T:general}.  For some parts of the theorem, we state and prove slightly stronger versions.   
 
Throughout this section, we assume $C$ is a \Diamantine category.  Any further assumptions will be stated explicitly in our results, and in the  theorems we will remind the reader of the  \Diamantine assumption.

To compactify our formulas below, we define $T_{F,H}:=\hat{\mu}(H,F)$ (note the order switches), 
and for $\pi \in \Sur(G,F)$ we define $T_\pi:=\mu(\pi)/|\Aut(F)|$.
We also let $S_{G,F}:=|\Sur(G,F)|/|\Aut(F)|$, and $I:=\category/_\isom$ (isomorphism classes).
Finally, whenever we consider $(M_G)_G\in \R^I$, we let $m_G:=M_G/|\Aut(G)|$.

The main step in the proof of (Uniqueness) is the following lemma.
\begin{lemma}\label{L:IE}
Let $G,H\in \category$, then 
$$\sum_{F\in I} T_{F,H}S_{G,F}=
\sum_{F\in I} T_{G,F}S_{F,H}=
\begin{cases}
1 &\textrm{if $G=H$}\\
0 &\textrm{if $G\ne H$}.
\end{cases}
$$
\end{lemma}
\begin{proof}

The quotients of $G$ that are (in $C$) isomorphic to $F$
correspond exactly to the $\Aut(F)$ orbits on $\Sur(G,F)$ (which all have size $\abs{\Aut(F)}$).  
Thus, given $G,H\in C$, we have an epimorphism
\begin{align*}
\{F\in I, \pi^G_F\in \Sur(G,F), \rho\in \Sur(F,H)  \} &\ra 
\{\pi^G_H\in \Sur(G,H),  F\in [H,G]  \} \\
(F, \pi^G_F,\rho)&\mapsto (\rho\pi^G_F, (F, \pi^G_F)),
\end{align*}
where the fiber of $(\rho\pi^G_F, (F, \pi^G_F))$ has size $|\Aut(F)|$.
Thus,
\begin{align*}
\sum_{F\in I} S_{G,F} T_{F,H}&=
\sum_{F\in I} \sum_{\pi^G_F\in \Sur(G,F)}  \frac{1}{|\Aut(F)|}  \frac{1}{\abs{\Aut(H)}} \sum_{\rho\in \Sur(F,H)} \mu(\rho) \\
&=
\sum_{\pi^G_H\in \Sur(G,H)}  \frac{1}{|\Aut(H)|} \sum_{F\in [H,G]} \mu(H,F) \\
&=
\begin{cases}
1 &\textrm{if $G=H$}\\
0 &\textrm{if $G\ne H$},
\end{cases}
\end{align*}
where the last equality is by \eqref{E:Mobdef}.
The second statement follows similarly, using the other part of \eqref{E:Mobdef}.
\end{proof}

We are now ready to prove (Uniqueness) with a slightly weaker hypothesis, and allowing us to work with formations more general than levels.

\begin{theorem}\label{amc-1}
Let $\category$ be a \Diamantine category, $\W$ be a downward-closed subset of isomorphism classes of $\category$, and
$(M_G)_{G}\in \R^I$ be so that for all $F\in \W$, we have 
\begin{equation}\label{E:weakwell}
\sum_{G \in \W}     \sum_{ \pi \in\Sur(G,F) }{|{\mu}(F,G)m_G|}<\infty.
\end{equation}

Suppose for all $F\in \W$ we have values $p_{\W,F} \geq 0$, such that for all $G\in {\W}$ we have
$$
\sum_{F\in \W} S_{F,G} p_{\W,F} =m_G
$$
Then, for all $F\in C$,
$$
p_{\W,F} = \sum_{G\in \W} T_{G,F} m_G
$$

\end{theorem} 

\begin{proof} If we can exchange the order of summation, we have
\begin{align*}
\sum_{G\in \W} T_{G,F}m_G= \sum_{G\in \W} T_{G,F} \sum_{H\in \W} S_{H,G} p_{H,\W}  =\sum_{H\in\W}  p_{H,\W} \sum_{G\in \W} T_{G,F}  S_{H,G} = p_{\W,F}
\end{align*} where the first equality is by assumption, the last is from Lemma~\ref{L:IE} (using that anything with an epimorphism from $H$ is in $\W$).
 To verify that we can exchange the order of summation, as in the middle equality, it suffices to check that the sum $\sum_{G\in \W} T_{G,F} \sum_{H\in \W} S_{H,G} p_{H,\W} $ is absolutely convergent.
Since $S_{H,G} p_{H,\W}\geq 0$, it suffices to show that $ \sum_{G\in \W} T_{G,F}m_G$ is absolutely convergent, which is precisely our assumption. 
\end{proof} 

Note that (Uniqueness) implies the ``only if'' direction of (Existence).
Similarly, we could prove the ``if'' direction of  (Existence),  conditional on absolute convergence of the below sums,
\begin{align}\label{E:FakeExist}
 \sum_{F\in \W}  S_{F,H} \left(\sum_{G\in \W} T_{G,F}m_G \right) =\sum_{G\in {\W}} 
m_G
 \sum_{F\in \W} T_{G,F}S_{F,H} =m_H
 .
\end{align}

Unfortunately,  in most cases of interest the absolute convergence of the above sums will not hold (see Example~\ref{no-absolute-convergence}).  This requires us to make more delicate arguments to exchange the order of summation one piece at a time, relying on further identities
satisfied by the $T_{G,F}$ as well as bounds on their size.

\begin{example}\label{no-absolute-convergence} 
One of the nicest examples of moments are that of the distribution $\mu$ on profinite groups from \cite{Liu2020}, which has $G$-moment $M_G=1$ for every finite group $G$.  (We will see in Lemma \ref{group-card-wb} that these are well-behaved.) 
The category $\category$ here is the  category of finite groups.
However, even in this case, the sum
 \[ \sum_{F \in {\C}} \sum_{G \in {\C}} T_{G,F}  m_G=
 \sum_{F \in {\C}} \sum_{G \in {\C}} T_{G,F}  \frac{1}{|\Aut(G)|} \] 
 does not converge absolutely when $\C$ is the level generated by $\{\F_3,S_3\}$
(we write $\F_2,\F_3$ for the cyclic groups of order 2,3, respectively). (This sum is the left side of \eqref{E:FakeExist} in the special case $H=1$.)
Indeed, to check that \[\sum_{F \in {\C}} \sum_{G \in {\C}}  |T_{G,F}|  |\Aut(G)|^{-1} =\infty,\] it suffices to prove a restricted sum is infinite, over only $F = \mathbb F_2^n$ and $G = \mathbb F_3^m \rtimes \mathbb F_2^n$ for some $n$ and $m$ and any action of $\mathbb F_2^n$ on $\mathbb F_3^m$.  
Every representation of $\mathbb F_2^n$ over $\mathbb F_3$ is a sum of the $2^n$ one-dimensional characters, and from this we can check every such $G$ is in $\C$.

For such a $G$ and $F$, any semisimple morphism $G\to F$ is the obvious projection  $\mathbb F_3^m \rtimes \mathbb F_2^n \ra \mathbb F_2^n$ composed with an automorphism of $\mathbb F_2^n$, so \[T_{G,F} =  
  \mu(F,G) .\]

For a given $n$, we can consider the set  $\Hom ( \mathbb F_2^n, \GL_m (\mathbb F_3))$.  We have that $\GL_n(\F_2)\times \GL_m( \F_3)$ acts on this set,
by precomposition for the first factor, and conjugation for the second factor.  
The orbits correspond to isomorphism classes of groups of the form $\F_3^m \rtimes \F_2^n$, and the stabilizers $S_a$ correspond to automorphisms of 
$\F_3^m \rtimes_a \F_2^n$
that fix the subgroup $\F_2^n$ (setwise). 
For $a\in \Hom ( \mathbb F_2^n, \GL_m (\mathbb F_3))$,
the automorphisms of $\F_3^m \rtimes_a \F_2^n$ act on the Sylow $2$-subgroups transitively, and there are $[\F_3^m: (\F_3^m)^{\im a}] $ Sylow $2$-subgroups.
Thus the stabilizer $S_a$ has size $|\Aut(\F_3^m \rtimes \F_2^n)|/[\F_3^m: (\F_3^m)^{\im a}]$.
It follows by the orbit-stabilizer theorem that for each $n$,
\[ \sum_{\substack{ G \in {\C}\\ G \cong \mathbb F_3^m \rtimes \mathbb F_2^n\\
\textrm{ for some $m$}}}  |T_{G,F}|  |\Aut(G)|^{-1}=  \sum_{m=0}^{\infty} \sum_{ a \in \Hom ( \mathbb F_2^n, \GL_m (\mathbb F_3))} \frac{ |\mu(F_2^n , F_3^m \rtimes_a F_2^n)|}{[\F_3^m: (\F_3^m)^{\im a}] | \GL_m( \mathbb F_3)| | \GL_n(\mathbb F_2)|}  
. \]
We add the subscript $a$ to the $\rtimes$ to clarify which action we mean.

The number elements of $\Hom ( \mathbb F_2^n, \GL_m (\mathbb F_3))$
giving multiplicity vector $m_1,\dots, m_{2^n}$ when decomposed into irreducible representations of $\mathbb F_2^n$ over $\mathbb F_3$ 
is $ | \GL_m (\mathbb F_3)|/|Z_{\GL_m (\mathbb F_3)} (\im a) |,$
where $a\in \Hom ( \mathbb F_2^n, \GL_m (\mathbb F_3))$ is any element with that multiplicity vector and 
$Z_{\GL_m (\mathbb F_3)} (\im a)$ is the centralizer of  $\im a$ in $\GL_m (\mathbb F_3)$.  Further, $Z_{\GL_m (\mathbb F_3)} (\im a)$ has size
${\prod_{i=1}^{2^n} \prod_{j=0}^{m_i-1} (3^{m_i} - 3^j) }$.
The  value of $\mu(F_2^n , F_3^m \rtimes_a F_2^n)$ for such an action is given by $\prod_{i=1}^{2^n} (  3^{ \frac{m_i (m_i-1)}{2} } ) $ by Lemma \ref{group-semisimple-description}. Thus 
\begin{align*}
\sum_{\substack{ G \in {\C}\\ G \cong \mathbb F_3^m \rtimes \mathbb F_2^n\\
\textrm{ for some $m$}}}  |T_{G,F}|  |\Aut(G)^{-1}|
&= \sum_{m_1,\dots, m_{2^n}=0}^{\infty} \frac{1}{ | \GL_n(\mathbb F_2)| 3^{\sum_{i>1} m_i}}   \frac{\prod_{i=1}^{2^n} (  3^{ \frac{m_i (m_i-1)}{2} } )  }{\prod_{i=1}^{2^n} \prod_{j=0}^{m_i-1} (3^{m_i} - 3^j) } \\
&> \frac{1}{| \GL_n(\mathbb F_2)|}    \left( \sum_{m=0}^{\infty}  \frac {   3^{ \frac{m (m-1)}{2} }  }{3^m \prod_{j=0}^{m-1} (3^{m} - 3^j) } \right)^{2^{n}} ,
\end{align*}
 so
\[\sum_{F \in {\C}} \sum_{G \in {\C}}  |T_{G,F}|  |\Aut(G)|^{-1}  > \sum_{n=0}^{\infty}   \frac{1}{ | \GL_n(\mathbb F_2)|}    \left( \sum_{m=0}^{\infty}  \frac {  3^{ \frac{m (m-1)}{2} }  }{3^m\prod_{j=0}^{m-1} (3^{m} - 3^j) } \right)^{2^n}  = \infty\] since 
\[ c:=\sum_{m=0}^{\infty}  \frac {  3^{ \frac{m (m-1)}{2} }   }{3^m\prod_{j=0}^{m-1} (3^{m} - 3^j) }  = 1+ \frac{1}{6} + \dots > 1\] and 
$| \GL_n(\mathbb F_2)|\leq 2^{n^2}$, and hence for large $n$ we have that $c^{2^n}/| \GL_n(\mathbb F_2)| \geq 1$.
\end{example}

\subsection{Proof of (Existence)}
Example~\ref{no-absolute-convergence} shows that our proof of (Existence) will need to be somewhat more complicated
than the easy argument in \eqref{E:FakeExist} that requires interchanging the order of summation between a sum over $F\in \W$ and a sum over $G\in \W$.
We will prove (Existence)
as a sequence of smaller identities, each based on rearranging a shorter sum that does converge absolutely for well-behaved moments.
The key is not to exchange a sum in $\Ggp$ with a sum over all $\Fgp$,  but only a sum over $\Fgp$ with a fixed $\bar{\W}$-quotient, for a level $\bar{\W}$ just barely smaller than the level $\W$ under consideration.
  In Example~\ref{no-absolute-convergence}, this would mean we only consider the sums  for a fixed $n$, instead of summing over all $n$ at once.
After we do this more limited exchange of summation, we will be able to cancel some terms before proceeding further with the exchange of order of summation.
We will see that our notion of ``well-behaved'' moments is sufficient to allow us to do these more restrained exchanges of the order of summation.

We will give an existence result that is stronger than given in the introduction, allowing us to use formations other than levels for more flexibility.
We say that an epimorphism $\Ggp \ra E$ is 
\emph{relative-$\W$} if the induced map $\Ggp\ra \Ggp^\W\vee_\Ggp E$ is an isomorphism.  

\begin{lemma}\label{L:relE}
If $\W$ is a formation  and $E\in \W$, then $\Ggp\ra E$ is relative-$\W$ if and only if $\Ggp\in \W$.
\end{lemma}
\begin{proof}
If $\Ggp\in \W$, then $\Ggp^\W\vee_\Ggp E=\Ggp \vee E=\Ggp.$  If $\Ggp=\Ggp^\W\vee_\Ggp E$, since $E\in \W$,
we have that $\Ggp\in\W$. 
\end{proof}

We now give an identity that we will use for the canceling of terms mentioned above.
\begin{lemma}\label{L:relC}
Let $ \W$ be a  formation.  
Let $\Ggp,E\in \category$ and let $\phi:\Ggp\ra E$ be an epimorphism that is not relative-$\W$.
Let $\barHornot\in {\W}$. Then
$$
\sum_{\substack{\Fgp\in I \\ \Fgp^{\W}\isom \barHornot  }} \sum_{\rho\in \Sur(\Fgp,E)} \sum_{\substack{\pi\in \Sur(\Ggp,\Fgp)\\ \rho\pi=\phi}}T_\pi=
\sum_{\substack{\Fgp\in I \\ \Fgp^{\W}\isom \barHornot  }} \sum_{\rho\in \Sur(\Fgp,E)} \sum_{\substack{\pi\in \Sur(\Ggp,\Fgp)\\ \rho\pi=\phi}} \frac{\mu(\pi)}{|\Aut(\Fgp)|}=0.
$$
\end{lemma}
\begin{proof}
Note that the sums here are finite because they are all over quotients of a fixed $G$.
We have that $\Ggp\ra E$ not relative-$\W$ implies that $\Ggp> \Ggp^\W \vee E$, and hence for  $K\leq \Ggp^\W$, also that $\Ggp>K \vee E$.
Also, by Lemma~\ref{L:Cvee} (3), for $\Fgp\leq \Ggp$, we have $\Fgp\wedge \Ggp^\W=\Fgp^\W$.
Thus, for any epimorphism $\Ggp^\W\ra \barHornot $,
\begin{align*}
0=&\sum_{\substack{K\in [\barHornot ,\Ggp^{\W}] }} \mu(\barHornot ,K) \sum_{\Fgp\in [K\vee E,\Ggp]} \mu(\Fgp,\Ggp) &\quad \textrm{ (inner sums vanish since $\Ggp>K\vee E$)}\\
=&\sum_{\substack{\Fgp\in [\barHornot \vee E,\Ggp] }} \mu(\Fgp,\Ggp) \sum_{\substack{K\in [\barHornot ,\Fgp\wedge \Ggp^\W] }} \mu(\barHornot ,K) &\quad \textrm{ (change order of summation)}\\
=&\sum_{\substack{\Fgp\in [\barHornot \vee E,\Ggp] \\\Fgp^\W=\barHornot  }} \mu(\Fgp,\Ggp) &\quad \textrm{ (inner sum vanishes unless $\Fgp^\W=\barHornot $)}.
\end{align*}
We will see that, summed over the structures on $\barHornot $ as a quotient of $\Ggp$ (i.e. over $\Sur(\Ggp,\barHornot )/\Aut(\barHornot )$), this final sum is precisely the sum in the lemma.

There is an $|\Aut(\Fgp)|-1$ correspondence between pairs $(\Fgp,\pi)$ of $\Fgp\in I$ and $\pi\in \Sur(\Ggp,\Fgp)$ and $\Fgp \leq \Ggp$.
Given $\pi$ and $\phi$, when $\Fgp\geq E$ then there is one choice of $\rho\in \Sur(\Fgp,E)$ such that $\rho\pi=\phi$, and otherwise there is no possible $\rho$ such that $\rho\pi=\phi$.  Finally, $\Fgp^\W=\barHornot $ in the sum above means that $\Fgp^\W$, as a quotient of $\Ggp$, is isomorphic to $\barHornot $, and so obtain the sum over all
$\Fgp\in I$ with $\Fgp^\W$ isomorphic to $\barHornot $ in $\category$, as in the lemma, we need to sum over all $\Ggp$-quotient structures on $\barHornot $ up to isomorphism.
\end{proof}

Now we will see that well-behaved moments will provide absolute convergence on more limited sums.
\begin{proposition}\label{P:makelocal}
Let $\bar{\W}\sub \W$ be narrow formations such that for every $G\in \W$, we have that $G\ra G^{\bar{\W}}$ is semisimple.
Let $E\in \category$ 
  and $\barHornot \in {\bar{\W}}$.
If  $(M_\Ggp)_\Ggp$ is behaved 
at a narrow formation $\tilde{\W}$ containing $\W$ and $E$, then
$$
\sum_{\Ggp\in I} \sum_{\substack{\Fgp\in I \\ \Fgp^{\bar{\W}}\isom\barHornot   }} \sum_{\pi\in \Sur(\Ggp,\Fgp)} \sum_{\substack{ \rho\in \Sur(\Fgp,E) \\ \rho\pi \textrm{ relative-$\W$} }}
 |T_\pi   m_G|
<\infty.
$$  
\end{proposition}
\begin{proof}
Given $\Ggp,\Fgp,\pi,\rho$ as in the sum,
by Lemma~\ref{non-semisimple-vanishing},  $\mu(\pi)=0$ unless $\pi$ is semi-simple, so we only consider terms with $\pi$ semi-simple.
We will work in the lattice of (isomorphism classes of) quotients of $\Ggp$, and our first goal is to show that $\Ggp \ra \Fgp^{\bar{\W}} \vee E$ is semisimple. 
By assumption, 
$\Ggp^\W\ra \Ggp^{\bar{\W}}$ is semisimple, and  Lemma~\ref{semisimple-equivalence} then implies
$\Ggp^\W\ra (E\vee  \Ggp^{\bar{\W}})\wedge \Ggp^\W$ is semisimple, since $\Ggp^{\bar{\W}}\leq (E\vee  \Ggp^{\bar{\W}})\wedge \Ggp^\W$.
We have $(E\vee \Ggp^{\bar{\W}}) \vee \Ggp^\W= E\vee \Ggp^\W=\Ggp$, where the last equality is from $\rho\pi$  relative-$\W$.
Then by the Diamond Isomorphism Theorem applied to $E\vee \Ggp^{\bar{\W}}$ and $\Ggp^\W$, 
we have $[ E\vee  \Ggp^{\bar{\W}},\Ggp]=[(E\vee  \Ggp^{\bar{\W}})\wedge \Ggp^\W,\Ggp^\W ]$, and thus $\Ggp\ra E\vee  \Ggp^{\bar{\W}}$ is semisimple.
Since $E\vee  \Ggp^{\bar{\W}}$ and $\Fgp$ are both (as $\Ggp$-quotients) the meet of quotients whose morphism from $\Ggp$ is simple
(by  Lemma~\ref{semisimple-equivalence}), the same is true for $(E\vee  \Ggp^{\bar{\W}}) \wedge \Fgp$.
We have
\begin{align*}
\Fgp^{\bar{\W}}  \vee E &= (\Ggp^{\bar{\W}} \wedge \Fgp) \vee E &\textrm{(Lemma~\ref{L:Cvee} (3))}\\
&= (E \vee \Ggp^{\bar{\W}}) \wedge \Fgp &\textrm{(modular property)}.
\end{align*}
Thus $\Ggp \ra \Fgp^{\bar{\W}} \vee E$ is semisimple.

We claim that, given $E$ and $\barHornot $, there are only finitely many (isomorphism classes) of objects that can be $\Fgp^{\bar{\W}} \vee_\Ggp E$ for some $G$, 
with $G,F,\pi,\rho$ as in the sum.
We can consider the level $\C$ generated by $\{\barHornot ,E\}$, and note any $\Fgp^{\bar{\W}} \vee_\Ggp E\in \C$.
By the Diamond Isomorphism Theorem, we have that $\dim (\Fgp^{\bar{\W}} \vee_\Ggp E \ra E) \leq
\dim \Fgp^{\bar{\W}} =
 \dim \barHornot $.
By Lemma~\ref{dim-C-finite}, there are at most finitely many (isomorphism classes of) $K\in  \C$ with 
an epimorphism $K\ra  E$ of dimension at most $\dim \barHornot $, which proves the claim. 
Let $K_1,\dots,K_t$ be objects of $\category$, ones from each such isomorphism class of $K$, and note that $\Sur(K_i,E)$ is finite.

Since $\rho\pi$ is relative-$\W$, we have that  $\Ggp\in \tilde{\W}.$
Then
{\allowdisplaybreaks\begin{align*}
&\sum_{\Ggp\in I} \sum_{\substack{\Fgp\in I \\ \Fgp^{\bar{\W}}\isom\barHornot   }} \sum_{\pi\in \Sur(\Ggp,\Fgp)} \sum_{\substack{ \rho\in \Sur(\Fgp,E) \\ 
\rho\pi \textrm{ relative-$\W$} }}
|T_\pi m_G| \\
\leq &
\sum_{i=1}^t 
\sum_{\psi \in \Sur(K_i,E)}
\sum_{\Ggp\in \tilde{\W}} 
\sum_{\substack{\Fgp\in I   \\ \Fgp^{\bar{\W}}\isom\barHornot  }} 
\sum_{\pi\in \Sur(\Ggp,\Fgp)} \sum_{\substack{ \rho\in \Sur(\Fgp,E) \\
G\ra \Fgp^{\bar{\W}} \vee_G E \textrm{ semisimple}}}
\sum_{\substack{\phi \in \Sur(G, K_i) 
\\ G\ra \Fgp^{\bar{\W}} \vee_G E \textrm{ and } \phi
\\ \textrm{isom $G$-quotients},\\
\psi \phi =\rho\pi
}}
|T_\pi m_G| \\
\leq &
\sum_{i=1}^t 
\sum_{\psi \in \Sur(K_i,E)}
\sum_{\Ggp\in \tilde{\W}} 
\sum_{\substack{\phi \in \Sur(G, K_i) \\
\phi \textrm{ semisimple}
}}
\sum_{\substack{\Fgp\in I    }} 
\sum_{\substack{\pi\in \Sur(\Ggp,\Fgp)\\
\phi=\kappa\pi \textrm{ for some } \kappa}} 
|T_\pi m_G| \\
= &
\sum_{i=1}^t 
\sum_{\psi \in \Sur(K_i,E)}
\sum_{\Ggp\in \tilde{\W}} 
\sum_{\substack{\phi \in \Sur(G, K_i) \\
\phi \textrm{ semisimple}
}}
\sum_{\substack{\Fgp\in [K_i,G]    }} 
\frac{ \left| \mu(F,G)M_\Ggp\right|}{|\Aut(\Ggp)|}\\
\ll &
\sum_{i=1}^t 
\sum_{\psi \in \Sur(K_i,E)}
\sum_{\Ggp\in \tilde{\W}} 
\sum_{\substack{\phi \in \Sur(G, K_i) \\
\phi \textrm{ semisimple}
}}
\frac{ \left| \mu(K_i,G)Z(\phi)^aM_\Ggp\right|}{|\Aut(\Ggp)|}.
\end{align*}}
In the second inequality, we can drop the sum over $\rho$ because $\rho$ is determined as $\psi \kappa$ (and $\kappa$ is unique since $\pi$ is an epimorphism).
The last inequality is by Lemma~\ref{Z-to-convolution}, where $a$ is as in the definition of behaved.
The implicit constant in the $\ll$ depends on the $K_i$ and $\tilde{\W}$, which in turns depends on only $\W$, $H$, and $E$.
The final sum converges by the definition of behaved and the fact that the first two sums are finite.
\end{proof}

We combine Lemma~\ref{L:relC} and Proposition~\ref{P:makelocal} to give a general result from which (Existence) will follow.
\begin{proposition}\label{P:twist}
Let $\W$ be a narrow formation of finite complexity.

  Let $E\in \W$.
If  $(M_G)_G$ is behaved at $\W$ and  $v_{{\W},H}\geq 0$  for all $H\in \W$, then
$$
 \sum_{\substack{H\in \W}} |\Sur(H,E)| v_{\W,H}=M_E.
$$
\end{proposition}
\begin{proof}
For all $i$, let $\W_i$ be the formation of elements of $\W$ of complexity $\leq i$ (see Lemma~\ref{L:complexity}).
Let $t$ be such that $\W_t=\W$.
For any $i$ and $j=i,i+1$, we let
$$
\sS_{H,i,j}:=\sum_{\ggp\in I} \sum_{\substack{\fgp\in I \\ \fgp^{\W_i}\isom H  }} 
\sum_{\rho\in \Sur(\fgp,E)} \sum_{\substack{ \pi\in \Sur(\ggp,\fgp)\\ \rho\pi \textrm{ relative-}\W_j }}
T_\pi m_G.
$$
We have
\begin{align*}
\sum_{\substack{H\in \W_t\\H^{\W_0}\isom E^{\W_0}}} \sS_{H,t,t}&=
\sum_{\substack{H\in \W\\H^{\W_0}\isom E^{\W_0}}}
\sum_{\ggp\in \W} \sum_{\substack{\fgp\in I \\ \fgp^{\W}\isom H  }} 
\sum_{\rho\in \Sur(\fgp,E)} \sum_{\substack{ \pi\in \Sur(\ggp,\fgp)}}
T_\pi m_G\\
&=
\sum_{\substack{H\in \W}} |\Sur(H,E)|
\sum_{\ggp\in \W}
 \sum_{\substack{ \pi\in \Sur(\ggp,H)}}
T_\pi m_G\\
&= \sum_{\substack{H\in \W}} |\Sur(H,E)| v_{\W,H}.
\end{align*}
by Lemma~\ref{L:relE}, and since $G\in \W$ implies that $F^\W=F$ for a quotient $F$ of $G$, and since $H^{\W_0}\not\isom E^{\W_0}$ implies that $ |\Sur(H,E)|=0$.
One the other hand, 
\begin{align*}
\sum_{\substack{H\in \W_0 \\H^{\W_0}\isom E^{\W_0}}}  \sS_{H,0,0}&= 
\sum_{\ggp\in I} \sum_{\substack{\fgp\in I \\ \fgp^{\W_0}\isom E^{\W_0}  }} 
\sum_{\rho\in \Sur(\fgp,E)} \sum_{\substack{ \pi\in \Sur(\ggp,\fgp)\\ \rho\pi \textrm{ isom}}}
T_\pi m_G=M_E
\end{align*}
since a morphism is relative-$\W_0$ if and only if it is an isomorphism.
Thus to prove the proposition, it suffices to show that
$$
\sum_{\substack{H\in \W_t\\H^{\W_0}\isom E^{\W_0}}} \sS_{H,t,t} =
\sum_{\substack{H\in \W_0 \\H^{\W_0}\isom E^{\W_0}}}  \sS_{H,0,0}.$$

If $G\in \W_{i+1}$, then $G\ra G^{\W_{i}}$ is semisimple by Lemma~\ref{L:complexity}.
Then by Proposition~\ref{P:makelocal}, we have that the sum defining $\sS_{H,i,j}$ is absolutely convergent.
Using Lemma~\ref{L:relC}, we can subtract from $\sS_{H,i,i+1}$ the terms such that 
$\rho\pi $ is not relative-$\W_{i}$ and conclude that 
\begin{equation}\label{E:leveldown}
\sS_{H,i,i+1}=\sS_{H,i,i}.
\end{equation}
Since the sum defining $\sS_{H,i,i+1}$ is absolutely convergent, we can rearrange the order of summation to obtain
\begin{equation*}
\sum_{\ggp\in I} \sum_{\substack{\fgp\in I \\ \fgp^{\W_{i}}\isom H  }} 
\sum_{\rho\in \Sur(\fgp,E)} \sum_{\substack{ \pi\in \Sur(\ggp,\fgp)\\ \rho\pi \textrm{ relative-}\W_{i+1} }}
 T_\pi   m_\ggp=
\sum_{\substack{H'\in {\W_{i+1}}\\{H'}^{\W_{i}}\isom H}}
\sum_{\ggp\in I} \sum_{\substack{\fgp\in I \\ \fgp^{\W_{i+1}}\isom H'  }} 
\sum_{\rho\in \Sur(\fgp,E)} \sum_{\substack{ \pi\in \Sur(\ggp,\fgp)\\ \rho\pi \textrm{ relative-  }\W_{i+1} }}
 T_\pi   m_\ggp,
 \end{equation*}
i.e. 
\begin{equation}\label{E:moveF}
\sS_{H,i,i+1}= \sum_{\substack{H'\in {\W_{i+1}}\\(H')^{\W_{i}}\isom H}} \sS_{H',i+1,i+1}.
\end{equation}
We have
\begin{align*}
\sS_{H,t,t}&=\sum_{\ggp\in {\W}} \sum_{\substack{\fgp\in I \\ \fgp^{\W}\isom H  }} 
\sum_{\rho\in \Sur(\fgp,E)} \sum_{\substack{ \pi\in \Sur(\ggp,\fgp)}}
  T_\pi  m_\ggp  &\textrm{(Lemma~\ref{L:relE})}\\
  &=\sum_{\substack{\fgp\in I \\ \fgp^{\W}\isom H  }} 
\sum_{\rho\in \Sur(\fgp,E)} \sum_{\ggp\in {\W}} 
  T_{\ggp,\fgp}  m_\ggp  &\textrm{(absolute convergence})\\
  &\geq 0 &(v_{{\W},F}\geq 0).
\end{align*}
Using ~\eqref{E:moveF} and \eqref{E:leveldown}, we inductively 
(down from $t$)
see that for all $i$, we have $\sS_{H,i,i},\sS_{H,i,i+1}\geq 0$.
This non-negativity  allows us to rearrange the following sum and obtain, 
\begin{equation}\label{E:delH}
  \sum_{\substack{H'\in {\W_{i+1}}\\(H')^{{\W_0}}\isom E^{\W_0}}}
\sS_{H',i+1,i+1}=
   \sum_{\substack{H\in {\W_{i}}\\H^{{\W_0}}\isom E^{\W_0}}}
\sum_{\substack{H'\in {\W_{i+1}}\\(H')^{\W_{i}}\isom H}}
\sS_{H',i+1,i+1}.
\end{equation}
If we expand  $\sS_{H',i+1,i+1}$ with its definition, the resulting sums in \eqref{E:delH} are not necessarily absolutely convergent, as we can see in Example~\ref{no-absolute-convergence}. This requires us to take a great deal of care in this argument.  However,  the sums in \eqref{E:delH} are sums of the same non-negative terms
$\sS_{H',i+1,i+1}$,
and thus the sums both either equal some non-negative number or $\infty$.
Combining \eqref{E:delH},  \eqref{E:moveF},  \eqref{E:leveldown},  we obtain
\begin{equation}\label{E:indstep}
  \sum_{\substack{H'\in {\W_{i+1}}\\(H')^{{\W_0}}\isom E^{\W_0}}}
\sS_{H',i+1,i+1}=
   \sum_{\substack{H\in {\W_{i}}\\H^{{\W_0}}\isom E^{\W_0}}}
\sS_{H,i,i}.
\end{equation}
From \eqref{E:indstep}, we see inductively that for any $n$, we have 
\begin{equation}\label{E:indcomp}
  \sum_{\substack{H\in {\W_{n}}\\H^{{\W_0}}\isom E^{\W_0}}}
\sS_{H,n,n}=
   \sum_{\substack{H\in {\W_{0}}\\H^{{\W_0}}\isom E^{\W_0}}}
\sS_{H',0,0},
\end{equation}
which completes the proof of the proposition.
\end{proof}

We now can prove the ``if'' part of (Existence), which we generalize here.
\begin{theorem}\label{amc-23}
Let  $\category$ be a \Diamantine category,  $\W$ be a narrow formation of finite complexity, and  $(M_G)_G\in \R^\W$  be behaved at $\W$ 
with $v_{\W,F}\geq 0$ for all $F\in \W$.
Then     $ \nu_\W(\{F\})=v_{\W,F}$ defines a measure $\nu_\W$ on $\W$ (for the discrete $\sigma$-algebra)
 with moments $(M_G)_G$.
 Further, for any formation $\W'\sub \W$,  there exists a measure on $\W'$ (for the discrete $\sigma$-algebra)
 with moments $(M_G)_{G}\in \W'$.
\end{theorem}

\begin{proof} Specifying $\nu_{\W}( \{F\})$ for each $F$ in fact specifies the measure $\nu_\W$ since the underlying set is countable.  
Proposition~\ref{P:twist} immediately implies that $\nu_\W$ has moments $(M_G)_G$.
Since for $E\in \W'$ and $G\in W$, we have $|\Epi(G,E)|=|\Epi(G^{\W'},E)|$, the pushforward of 
$\nu_\W$ from $\W$ to $\W'$ along the map $ G\mapsto G^{\W'}$ gives a measure on $\W'$
with moments $(M_G)_{G}\in \W'$.
\end{proof}

Here is an example showing that the ``finite complexity" hypothesis is necessary in Theorem \ref{amc-23}.

\begin{example} Let $\C$ be the category of finite abelian groups and let $\W$ be the formation of abelian $p$-groups, which is narrow but not of finite complexity. By $\mathbb Z_p$, we mean the $p$-adic integers, viewed as a profinite group under addition. Let $M_G = |\Epi(\mathbb Z_p, G)|$ for all $G$, i.e. $M_{\mathbb Z/p^n} = p^n-p^{n-1}$ and $M_G =0$ for all non-cyclic $G$. Then the tuple $M_G (G)$ is well-behaved (in fact, the sum appearing in the definition of well-behaved is finite). However, $v_{\W, F}=0$ for all $F$, so the measure $ \nu_\W(\{F\})=v_{\W,F}$  has moments vanishing rather than matching $M_G$. \end{example}

\subsection{Proof of (Robustness)}

Now we prove (Robustness), which will eventually follow from (Uniqueness) after some delicate limit switching arguments.  Note that for Robustness
it is necessary that we work at a  formation with some boundedness conditions
and not the whole category (see \cite[Example 6.14]{Wang2021}).

We begin with a simple inequality, which we next use to prove a  technical lemma to provide us bounds we will need for exchanging limits and sums.

\begin{lemma}\label{semisimple-lower-bound} Let $\pi \colon G\to F$ be a semisimple morphism of dimension $d$. For any $k$ from $0$ to $d$, the number of $H\in [F, G]$ with the dimension of $H \to F$ equal to $k$ is at least $\binom{d}{k}$. \end{lemma}

\begin{proof} This follows from the classification Theorem \ref{fcml-classification}, because if the lower bound is true for lattices $L_1,L_2$ then it is true for their product, so it suffices to handle the three cases of Theorem \ref{fcml-classification}, which each follow from the fact that the $q$-binomial coefficients bound the ordinary binomial coefficients for $q \geq 1$. 
 \end{proof}

\begin{remark}
 Lemma \ref{semisimple-lower-bound} could in many cases be replaced by a more exact calculation.  For example, in the 
  case of groups one can exactly count the number of normal subgroups intermediate between a kernel of a semisimple morphism $\pi $ and the trivial subgroup. 
In this case, Lemma \ref{semisimple-lower-bound} is sharp when the kernel is a product of nonabelian finite simple $G$-groups, but is usually not sharp in the abelian case.
\end{remark}

\begin{lemma}\label{dominance}  Let $\W$ be a narrow formation.
For all $F \in \W$ and $n \in \mathbb N$,  let $p_{\W,F}^n \geq 0$  be a real number. Suppose that, for all $G \in \W$, we have
\begin{equation}\label{mom-soft-bound}
 \sup_{n\in \mathbb N} \sum_{F\in \W} S_{F,G} p_{\W,F}^n <\infty.
\end{equation}
Then, for all $H\in \W$, we have
$$
 \sum_{d \in \mathbb N}  \sup_{n \in \mathbb N} \Bigl( \sum_{F \in\W}    \sum_{  \substack{ \rho\in \Sur(F , H) \\ \dim  \operatorname{ss}(\rho ) =d }} 
 p_{ \W,F}^n
 \Bigr) < \infty,$$
 where $\operatorname{ss}(\rho)$ denotes the semisimplification of $\rho$, the maximal semisimple $K\ra H$ through which $\rho$ factors.
\end{lemma}

\begin{proof} 
Let $H\in \W$.
Summing \eqref{mom-soft-bound} over the set of $G \in {\W}$ and maps $\phi \in \Sur(G,H)$ which are semisimple of dimension $\leq 2$ (which is finite by Lemma \ref{dim-C-finite}), we obtain
\[ \infty>    \sum_{G \in \W} \sum_{\substack{\phi\in \Sur(G,H) \\ \textrm{semisimple} \\ \dim \phi \leq 2}} \sup_{n\in \mathbb N}   \sum_{F\in \W} S_{F,G} p_{\W,F}^n\]
\[ \geq \sup_{n \in \mathbb N}  \sum_{G \in \W} \sum_{\substack{\phi\in \Sur(G,H) \\ \textrm{semisimple} \\ \dim \phi \leq 2}} \sum_{F\in \W} S_{F,G} p_{\W,F}^n\]
\[ =  \sup_{n \in \mathbb N}  \sum_{G \in \W}  \sum_{F \in\W} \sum_{\substack{\phi\in \Sur(G,H)\\ \textrm{semisimple} \\ \dim \phi \leq 2}}  \sum_{\substack{ \pi\in\Sur(F,G)}} \frac{ p_{\W,F}^n}{ \abs{\Aut(G)}} \]
\[ =  \sup_{n \in \mathbb N} \sum_{F  \in \W}   \sum_{ \substack{ \rho\in \Sur(F,H)}}  p_{\W,F}^n \sum_{G \in \W}   \sum_{\substack{\phi\in \Sur(G,H)  \\ \textrm{semisimple} \\ \dim \phi \leq 2 }}  \sum_{\substack{ \pi\in\Sur(F,G) \\ \phi  \pi= \rho}} \frac{1}{ \abs{\Aut(G)} }. \]
Note that given $\rho$ as in the sum, if for $\pi\in \Sur(F,G)$ there exists a $\phi$ such that $\phi\pi=\rho$, then that $\phi$ is  unique.
Thus we can evaluate the  inner sum above by
$$ \sum_{G \in \W}  \sum_{\substack{\phi\in \Sur(G,H)  \\ \textrm{semisimple} \\ \dim \phi \leq 2 }}  \sum_{\substack{ \pi\in\Sur(F,G) \\ \phi  \pi= \rho}} \frac{1}{ \abs{\Aut(G)} } =|\{ G\in [H,F]\,|\, G\ra H \textrm{ semisimple of dim. $\leq$ 2} \}.|$$ 

Let $s_\rho:R_\rho \ra H$ be the semisimplification of $\rho$.  Then the  above 
 is equal to the number of $G\in [H, R_\rho]$ of dimension $\leq 2$ over $H$, since an element of $[H,F]$ is semisimple over $H$ if and only if it is $\leq R_\rho$. By Lemma \ref{semisimple-lower-bound}, this is at least \[ \binom{\dim (s_\rho)}{2} +  \binom{\dim (s_\rho)}{1} +   \binom{\dim (s_\rho)}{0} =  \binom{\dim (s_\rho)+1}{2} +1 .\]

This gives
\[ \infty >  \sup_{n \in \mathbb N} \sum_{F  \in \W}   \sum_{ \substack{ \rho\in\Sur(F,H)}}  p_{\W,F}^n  \left(  \binom{\dim (s_\rho)+1}{2} +1 \right) \] \[=\sup_{n \in \mathbb N}\sum_{d \in \mathbb N}   \sum_{F \in \W}    \sum_{  \substack{ \rho\in\Sur(F,H) \\ \dim s_\rho=d }} p_{ \W,F}^n  \left( \binom{d+1}{2} +1 \right) .\]

Let $N$ be the sum above.
Then by restricting the sum over $d$ in the definition of $N$ to a particular value, we obtain for any $d\in \mathbb N$ that
$$ \sup_{n \in \mathbb N}  \sum_{F \in \W}    \sum_{  \substack{\rho\in\Sur(F,H) \\ \dim s_\rho=d }} p_{\W,F}^n  \leq \frac{N }{  \binom{d+1}{2} +1 } .$$

This gives 

$$ \sum_{d \in \mathbb N}  \sup_{n \in \mathbb N} \Bigl( \sum_{F \in \W}    \sum_{  \substack{ \rho\in\Sur(F,H) \\ \dim s_\rho=d }} {p_{\W,F}^n} \Bigr)  \leq  \sum_{d\in \mathbb N }  \frac{N }{\left(  \binom{d+1}{2} +1\right) }  \ll { N} < \infty,$$ as desired.\end{proof}

If we have an epimorphism $F\ra R_0$, we can take its radical $F\ra R_1$, and then take the radical of $F\ra R_1$, etc., which leads to the following definition.
Given a sequence of maps $F\stackrel{\phi}{\ra} R_k \stackrel{\rho_k}{\ra} R_{k-1} \cdots \stackrel{\rho_1}{\ra} R_0$, we say $\phi,\rho_k,\dots,\rho_1$ is a \emph{radical sequence} if
for all $1\leq i\leq k$, we have that $\rho_{i+1} \circ \dots \circ \rho_k \circ \phi: F\ra R_i$ is a radical of $ \rho_{i} \circ \rho_{i+1} \circ \dots \circ \rho_k \circ \phi: F\ra R_{i-1}.$
Lemma~\ref{dominance} will allow us to inductively prove the following result using the  Fatou-Lebesgue theorem, moving the limit further and further past the sum as $k$ increases.

\begin{lemma}\label{radical-exchange-limit} 

 For some narrow formation $\W$, for all $F \in \W$ and $n \in \mathbb N$  let $p_{\W,F}^n \geq 0$  be a real number. Suppose that, for all $G \in \W$, we have that \eqref{mom-soft-bound} holds.
Then, for any $R_0\in \W$,  for all natural numbers $k$ we have
$$
\limsup_{n \ra\infty} \sum_{F\in\W} S_{F,R_0} p_{\W,F}^n
\leq  \sum_{\substack{ d_1,\dots,d_k\in\N}} \limsup_{n \to \infty} \sum_{ F,R_1,\dots,R_k \in \W} \sum_{ \substack { 
 \phi,\rho_k,\dots,\rho_1 
\textrm{ rad. seq.} 
\\\phi\in \Sur(F, R_k)\\
\rho_i \in \Sur(R_i, R_{i-1}) \\  \dim(\rho_i)=d_i
}} \frac{ p_{\W,F}^n }{ \prod_{i=0}^k \abs{ \Aut(R_i) }}.$$
\end{lemma} 

\begin{proof}
The proof is by induction on $k$.
The case $k=0$ is trivial.  Now we assume the lemma is true for $k$.
An epimorphism $F\ra R_k$ has a radical $F\ra R_{k+1}$, where the object $R_{k+1}$ is unique up to isomorphism,  the
map $F\ra R_{k+1}$ is unique up to composition with elements of $\Aut(R_{k+1})$, and the resulting map $R_{k+1}\ra R_k$ is semisimple.  Thus we have
$$\sum_{ F,R_1,\dots,R_k \in \W} \sum_{ \substack {\phi,\rho_k,\dots,\rho_1 
\textrm{ rad. seq.} \\ \phi\in \Sur(F, R_k)\\
\rho_i \in \Sur(R_i, R_{i-1})  \\   \dim(\rho_i)=d_i
}}\frac{ p_{\W,F}^n }{  \prod_{i=0}^k \abs{ \Aut(R_i) }}=  \sum_{d_{k+1} \in \N}   \sum_{ F,R_1,\dots,R_{k+1} \in \W} \sum_{ \substack { 
 b,\rho_{k+1},\dots,\rho_1 
\textrm{ rad. seq.} 
\\\phi' \in \Sur(F, R_{k+1})\\
\rho_i \in \Sur(R_i, R_{i-1}) \\  \dim(\rho_i)=d_i
}} \frac{ p_{ \W,F}^n }{  \prod_{i=0}^{k+1} \abs{ \Aut(R_i) }} .$$
To complete the induction, it suffices to show that, for fixed $d_1,\dots,d_k$, 
$$
\limsup_{n \to \infty}
 \sum_{d_{k+1} \in \N}
\sum_{\substack{F,R_1,\dots,R_{k+1}\\
 \phi',\rho_{k+1},\dots,\rho_1 }} 
  \frac{ p_{\W,F}^n }{  \prod_{i=0}^{k+1} \abs{ \Aut(R_i) }}
\leq  \sum_{d_{k+1} \in \N} \limsup_{n \to \infty} \sum_{\substack{F,R_1,\dots,R_{k+1}\\
 \phi',\rho_{k+1},\dots,\rho_1 }}   \frac{ p_{\W,F}^n }{  \prod_{i=0}^{k+1} \abs{ \Aut(R_i) }},
$$
(where the sums are over the same sets as in the previous equation).
Given $d_1,\dots,d_k$, Lemma~\ref{dim-C-finite} says there are only finitely many possible isomorphism classes of objects that can be $R_1,\dots,R_k$
in the above sums, as $\dim(\rho_i)=d_i$.  Thus the sum over $R_1,\dots,R_k,\rho_1,\dots,\rho_k$ can be freely exchanged with the $\limsup$.  
Fixing $R_1,\dots,R_k,\rho_1,\dots,\rho_k$, we apply the Fatou-Lebesgue Theorem to the remaining sum, and the hypothesis is checked by applying
Lemma~\ref{dominance} with $\rho=\rho_{k+1} \phi' $ and $H=R_{k}$.
The condition that $\phi'$ is the radical of $\rho_{k+1}\phi'$ determines $R_{k+1}$, $\phi'$
(up to simply transitive $\Aut(R_{k+1})$ orbit), and then $\rho_{k+1}$
 from $\rho$.  
The Fatou-Lebesgue theorem thus gives the inequality above, and we conclude the  lemma by induction.
\end{proof}

Lemma~\ref{radical-exchange-limit} will be useful because it turns out that at finite complexity, radical sequences cannot go on arbitrarily long without stabilizing.  

\begin{lemma}\label{radical-tower-bound} 
Let $\W$ be a formation of finite complexity $k$. 
 For all tuples $R_0,\dots, R_k \in \W$ with semisimple $\rho_i\colon R_i \ra R_{i-1}$ and $F \in \W$ with $\phi \in \Sur(F, R_k)$, if $\phi,\rho_k\dots\rho_1$ is a radical sequence, then $\phi$ is an isomorphism. 
\end{lemma}

\begin{proof} 
Let $\W_i$ be the formation of elements of $\W$ with complexity $\leq i$.
We will show that $R_i \geq F^{\W_i}$ (as quotients of $F$) for all $i$, by induction on $i$. It will follow that $R_k \geq F^{\W_k} = F$ so $R_k \isom F$, as desired.

The case $i=0$ is automatic because $F^{\W_0} $ is the minimal quotient of $F$ and thus is $\leq R_0$.  

For the induction step, assume $R_i \geq F^{\W_i}$.  By Lemma \ref{dit}, $[R_i, F^{\W_{i+1}} \vee R_i]$ is isomorphic to $[F^{\W_{i+1}} \wedge R_i, F^{\W_{i+1}}]$, and the latter is a subinterval of $[F^{\W_{i}} , F^{\W_{i+1}}]$ since $F^{\W_{i}} \leq F^{\W_{i+1}}$ and $F^{\W_{i}} \leq R_i$.   By Lemma \ref{L:complexity}, the natural morphism $F^{ \W_{i+1}} \to F^{\W_{i}}$ is semisimple for all $i$, so $[F^{\W_{i}} , F^{\W_{i+1}}]$ is a complemented modular lattice, and thus $[F^{\W_{i+1}} \wedge R_i, F^{\W_{i+1}}]$ is complemented by Lemma \ref{complemented-equivalence}, and hence $[R_i, F^{\W_{i+1}} \vee R_i]$ is as well.  Because $F^{\W_{i+1}} \vee R_i \to R_i$ is semisimple, we have $R_{i+1} \geq F^{\W_{i+1}} \vee R_i$,  since $R_{i+1}$ is the maximal quotient that $F\ra R_i$ factors through that is semisimple over $R_i$. 
Thus, $R_{i+1} \geq F^{\W_{i+1}}$, verifying the induction step.
\end{proof}

\begin{remark}
With the notation of the above proof,  if $R_0$ is a minimal object then  $R_i\in \W_i$ and thus
$ F^{\W_i}\geq R_i,$ implying that $F^{\W_i}=R_i$, as quotients of $F$.
Thus,  when $\phi: F \ra R_0$ is the minimal quotient of $F$, the complexity of $F$ is the minimal $k$ such that 
$\phi,\rho_k,\dots\rho_1$ is a radical sequence with $\phi$ an isomorphism,  i.e.  $R_k=F$ as quotients of $F$.
\end{remark}

Combining Lemmas \ref{radical-exchange-limit} and \ref{radical-tower-bound}, we show that under the assumptions  \eqref{mom-soft-bound} 
(``sup of moments is finite'') and that a limit distribution exists, we can take a limit before or after taking moments (``limit moments exist and are moments of the limit'').

\begin{lemma}\label{robustness-exchange}  
Let $\W$ be a narrow formation of finite complexity and for all $F \in \W$ and $n \in \mathbb N$  let $p_{\W,F}^n \geq 0$  be a real number. 
Suppose that,  for each $F\in \W$, we have that the limit $\lim_{n\ra\infty} p_{\W,F}^n$ exists and
for all $G \in \W$, we have
\begin{equation}
 \sup_{n}  \sum_{F\in \W} S_{F,G} p_{\W,F}^n <\infty
\end{equation}
Then for all $G\in \W$ we have
\begin{equation}
  \lim_{n \to \infty}  \sum_{F\in \W} S_{F,G} p_{\W,F}^n = \sum_{F\in \W} S_{F,G}  \lim_{n \to \infty} p_{\W,F}^n.
\end{equation}
\end{lemma}

\begin{proof} Let $k$ be the complexity of $\W$.
Lemma \ref{radical-exchange-limit} with $R_0=G$ gives
$$
\limsup_{n \ra\infty} \sum_{F\in\W} S_{F,G} p_{\W,F}^n
\leq  \sum_{\substack{ d_1,\dots,d_k\in\N}} \limsup_{n \to \infty} \sum_{ F,R_1,\dots,R_k \in \W} \sum_{ \substack { 
 \phi,\rho_k,\dots,\rho_1 
\textrm{ rad. seq.} 
\\\phi\in \Sur(F, R_k)\\
\rho_i \in \Sur(R_i, R_{i-1}) \\  \dim(\rho_i)=d_i
}} \frac{ p_{\W,F}^n }{ \prod_{i=0}^k \abs{ \Aut(R_i) }}.$$
Lemma \ref{radical-tower-bound} says that is in the above sum,  $F$ is determined by $R_k$ and $\phi$ must be an isomorphism.
Then, Lemma~\ref{dim-C-finite} implies that, given $d_1,\dots,d_k$,  the remaining sums over $F,R_1,\dots,R_k,\phi,\rho_k,\dots,\rho_1$ only have finitely many terms and can be exchanged with the $\limsup$.
Thus the right-hand side above is equal to
$$
\sum_{\substack{ d_1,\dots,d_k\in\N}}  \sum_{ F,R_1,\dots,R_k \in \W} \sum_{ \substack { 
 \phi,\rho_k,\dots,\rho_1 
\textrm{ rad. seq.} 
\\\phi\in \Sur(F, R_k)\\
\rho_i \in \Sur(R_i, R_{i-1}) \\  \dim(\rho_i)=d_i
}} 
\frac{1}{ \prod_{i=0}^k \abs{ \Aut(R_i) }}
\lim_{n\ra\infty} { p_{\W,F}^n }=\sum_{F\in \W} \sum_{\rho\in \Sur(F,G)} \frac{\lim_{n\ra\infty} { p_{\W,F}^n }}{ \abs{ \Aut(G) }}.
$$
The equality above follows because if we take $\rho=\rho_1\circ\dots\circ\rho_k\circ\phi$, then, inductively each
$R_{i-1}$ and $\rho_i\circ\dots\circ\rho_k\circ\phi$ for $2\leq i \leq k+1$ is determined, up to simply transitive $\Aut(R_{i-1})$ orbit, from the radical sequence condition, which then determines $\rho_{i-1}$ (from $\rho$ and $\rho_i\circ\dots\circ\rho_k\circ\phi$).
Combining, we obtain
$$
\limsup_{n \ra\infty} \sum_{F\in\W} S_{F,G} p_{\W,F}^n
\leq  \sum_{F\in \W} S_{F,G} \lim_{n\ra\infty} p_{\W,F}^n  .
$$
Fatou's lemma gives 
\begin{align*}
 \liminf_{n \to \infty} \sum_{F\in\W} S_{F,G} p_{\W,F}^n &\geq \sum_{F\in \W} S_{F,G}  \lim_{n \to \infty} p_{\W,F}^n,
\end{align*}
and the lemma follows.
\end{proof}

Finally, this allows us to prove (Robustness) conditional on (Uniqueness).

\begin{theorem}\label{T:Robust}
Let $\category$ be a \Diamantine category and $\W$ be a narrow formation of finite complexity.
 Let $(M_G)_G\in \R^\W$ such that the following implication holds:
\begin{itemize}
\item For any $(q_{\W,F})_F\in [0,\infty)^\W$, if for all $G\in \W$
we have $\sum_{F\in \W} S_{F,G} q_{\W,F} =m_G$, then for all $F\in C$, we have
$
q_{\W,F} = \sum_{G\in \W} T_{G,F} m_G.
$
\end{itemize}
For all $F \in \W$ and $n \in \mathbb N$  let  $p_{\W,F}^n \in [0,\infty)$. 
Suppose that
for each $G \in \W$, we have
\begin{equation}\label{E:Momassump}
  \lim_{n \to \infty} \sum_{F\in \W} S_{F,G} p_{\W,F}^n =m_G.
\end{equation}
Then for all $F\in \W$ we have
\begin{equation}\label{E:RobustCon}
  \lim_{n \to \infty} p_{\W,F}^n = \sum_{G\in \W} T_{G,F} m_G.
\end{equation}
\end{theorem}
\begin{proof}
Each object of $C$ only maps to one minimal object, since any two quotients have a greatest lower bound.
We can easily reduce to the case when $C$ has only 1 minimal object, since all of the sums in the theorem ever only involve at once objects all of whom have the same minimal object as a quotient.  

If $H$ is the minimal object of $C$, then \eqref{E:Momassump} with $G=H$ implies that there is some
upper bound for the $p_{\W,F}^n$ over all $F\in\W$ and all $n$. 
By a diagonal argument the $p_{\W,F}^n$ have a weak limit $p_{\W,F}^\infty$ in $n$, i.e. there is a subsequence $n_j$ of $\N$ such that for all $F\in \W^*$,
we have $\lim_{j\ra\infty }p_{\W,F}^{n_j}=p_{\W,F}^\infty.$ By Lemma~\ref{robustness-exchange} and \eqref{E:Momassump},
  for all $G\in \W$, 
$$\sum_{F\in \W} S_{F,G} p_{\W,F}^\infty=m_G.$$  By the assumed implication, this implies  that for all $F\in \W$,
$$
p_{\W,F}^\infty= \sum_{G\in \W} T_{G,F} m_G.
$$
The fact that in every weak limit of $p_{\W,F}^n$ has $p_{\W,F}^\infty$ equal to the same  value $\sum_{G\in \W} T_{G,F} m_G$ implies that
$$
  \lim_{n \to \infty} p_{\W,F}^n = \sum_{G\in \W} T_{G,F} m_G.
$$
\end{proof}

Finally, we record how Theorem \ref{T:general} follows from the three theorems of this section.

\begin{proof}[Proof of Theorem \ref{T:general}]
By definition and Corollary~\ref{C:levelcom}, every level is a narrow formation of finite complexity.
Well-behavedness implies the assumption \eqref{E:weakwell} of Theorem~\ref{amc-1}, which gives (Uniqueness) and the ``only if'' part of (Existence). For the (Uniqueness) step, we use that the underlying set is countable so the measure is determined by its value on each singleton.
Theorem~\ref{amc-23} gives the ``if'' part of (Existence).  Theorem~\ref{T:Robust}, using (Uniqueness) to satisfy the hypothesis, gives (Robustness). 
\end{proof}

Since well-behavedness is a crucial input to our main results, we sometimes want to be able to check it more flexibly.  For that we have the following.
\begin{lemma}\label{la-input} Let $\mathcal T$ be a collection of  formations in $\category$. Assume that for every finite tuple of objects $G_1,\dots, G_n \in \category$, there exists $\W\in \mathcal T$ containing $G_1,\dots, G_n$.  

If the tuple $(M_G)_{G}\in \R^{\category_{/\isom}}$ is well-behaved at $\W$ for all $\mathcal W\in \mathcal T$, then
 $(M_G)_{G}\in \R^{\category_{/\isom}}$  is well-behaved.  \end{lemma}

\begin{proof} Each level $\C$ of $\category$ is generated by finitely many objects $G_1,\dots, G_n$. Let $\W\in \mathcal T$ contain those objects. Then since $\W$ is downward-closed, join-closed, and contains the minimal objects, $\C \subseteq \W$ so for $F\in \C$,
$$
\sum_{G \in \C}     \sum_{ \pi \in\Epi(G,F) }\frac{|{\mu}(F,G)|}{|\Aut(G)|} Z ( \pi)^3 M_G\leq \sum_{G \in \W}     \sum_{ \pi \in\Epi(G,F) }\frac{|{\mu}(F,G)|}{|\Aut(G)|} Z ( \pi)^3 M_G< \infty.$$

Thus $(M_G)_{G}\in \R^{\category_{/\isom}}$ is well-behaved at $\C$. Since this holds for each level $\C$, the tuple $(M_G)_{G}\in \R^{\category_{/\isom}}$ is well-behaved. \end{proof}

\section{The pro-object measure}\label{S:Pro}
This section has three parts. In the first, we  prove Theorem \ref{analytic-main-measure} and Theorem~\ref{T:weak-convergence} on measures of pro-isomorphism classes. In the second, we connect pro-isomorphism classes to the classical notion of pro-objects. In the third, we discuss pro-objects concretely in concrete categories such as rings and groups. 

\subsection{The pro-isomorphism class case}

 The following result is very useful for constructing measures on inverse limits of discrete spaces such as $\Prf$.
  
  \begin{lemma}\label{L:countadd}
Let $\nu$ be a finitely  additive non-negative function on  an algebra $\mathcal{A}$ of subsets of a set $X$ (with $\nu(X)=b$ finite).  
For each $\ell \in \mathbb{N}$,  let  $\mathcal{L}_\ell$  be a countable set of 
elements of $\mathcal{A}$ whose disjoint union is $X$.

(1) Suppose for every element $A\in \mathcal{A}$, there is some $\ell$ such that for all $\ell'\geq \ell$,  the set $A$ is a countable disjoint union of 
elements of $\mathcal{L}_{\ell'}$.

(2) Suppose any sequence $L_n\in \mathcal{L}_n$ with $L_j \cap L_i$ non-empty for all $i,j$ has $\cap_j L_j$ non-empty.  

(3) For $i\leq j$, and $U_i\in \mathcal{L}_i$, and $U_j\in \mathcal{L}_j$, then $U_i\cap U_j$ is either empty or $U_j$.

(4) Suppose for each $\ell$, we have $\sum_{L\in \mathcal{L}_\ell} \nu(L)=b$.

Then $\nu$ is countably additive on $\mathcal{A}$.
\end{lemma}

\begin{proof}
  This proof follows the arguments proving  \cite[Theorem 9.1]{Liu2020} from \cite[Theorem 9.2]{Liu2020}, but we spell out the argument here since we are working in a much more general setting.

First we show the following claim, the analog of \cite[Corollary 9.7]{Liu2020}.
Claim: If $B\in \mathcal{A}$ and $B$ is a disjoint union of $U_{\ell,j} \in L_\ell$ (for a fixed $\ell$), then
\begin{equation}\label{E:97}
\nu(B)=\sum_j \nu(U_{\ell,j}).
\end{equation}
Since the elements of $L_\ell$ are a countable partition of $X$, we also have that $X\setminus B$ 
is a countable disjoint union of the $V_{\ell,i}\in L_\ell$ not among the $U_{\ell,j}$.  
Hence for every $M$,  by finite additivity and the non-negativity of $\nu$, we have
$$
\sum_{j=1}^M \nu(U_{\ell,j}) \leq \nu(B) \leq  \nu(X\setminus \bigcup_{i=1}^M V_{\ell,i})=b-\sum_{i=1}^M \nu(V_{{\ell},i}).
$$
Taking limits as $M\ra\infty$ gives
$$
\sum_{j=1}^\infty \nu(U_{\ell,j}) \leq \nu(B) \leq  b-\sum_{i=1}^\infty \nu(V_{{\ell},i}) =\sum_{j=1}^\infty \nu(U_{\ell,j})
$$
where the last equality is by assumption (4).  This proves \eqref{E:97}.

 If we have disjoint sets $A_n \in\mathcal{A}$ with $A=\cup_{n\geq 1 } A_n \in \mathcal{A}$, by taking $B_n=A\setminus \cup_{j=1}^n A_j$, 
to show countable additivity, 
 it suffices to show that for $B_1 \supset B_2 \supset \dots$ (with $B_n\in \mathcal{A}$) with $\cap_{n\geq 1} B_n=\emptyset$ we have $\lim_{n\ra\infty} \nu(B_n)=0$.

We can assume, without loss of generality, that 
 for each $\ell\geq 1$, we have $B_\ell=\cup_{j} U_{\ell,j}$, with $U_{\ell,j}\in L_\ell$ (using (1)).
 (We can always insert redundant $B_i$'s if the required $\ell$'s  increase too quickly.) 
We will show by contradiction that $\lim_{\ell \ra \infty} \nu(B_\ell)=0$.

Suppose, instead that there is an $\epsilon>0$ such that for all $\ell$, we have $\nu(B_\ell)\geq \epsilon$.
It follows from \eqref{E:97}  that for each $\ell$ we have a subset $K_\ell\sub B_\ell$ such that $\nu(B_\ell\setminus K_\ell)<\epsilon/2^{\ell+1}$ and $K_\ell$ is a \emph{finite} union of $U_{\ell,j}$. 

Next, let $C_\ell=\cap_{j=1}^\ell K_j$.    Then 
$\nu(B_\ell\setminus C_\ell)<\epsilon/2$, since
\begin{align*}
\nu(B_\ell\setminus C_\ell)=& \nu(B_\ell\setminus K_\ell)+\nu(K_\ell\setminus K_\ell \cap K_{\ell-1})+\cdots 
+\nu(K_\ell \cap \cdots \cap K_2 \setminus K_\ell \cap \cdots \cap K_1) \\
< &\epsilon/2^{\ell+1} + \nu(B_{\ell-1} \setminus K_{\ell-1}) + \cdots   +\nu(B_1\setminus K_{1}) \\
< &\epsilon/2^{\ell+1} + \epsilon/2^{\ell} + \cdots   +\epsilon/2^{2}.
\end{align*}
So $\nu(C_\ell)\geq \epsilon/2$ for each $\ell$ and in particular it is non-empty.  Note $C_{\ell+1}\sub C_\ell$ for all $\ell$.  Pick $x_\ell\in C_\ell$ for all $\ell$.

Note that $C_\ell$ is a finite union of the $U_{\ell,j}$ (using (3)).
Pick an $H_1\in \mathcal{L}_1$ so that infinitely many of the $x_\ell$ are in $H_1$ (this is possible since all $x_\ell$ are in $C_1$ and  $C_1$ is a finite union of elements of $\mathcal{L}_1$), and then disregard the $x_\ell$ that are not in $H_1$.  In particular note $H_1\subset C_1$. Then pick $H_2\in \mathcal{L}_2$ so that infinitely many of the remaining $x_\ell$ are in $H_2$, and disregard the $x_\ell$ that are not. Since all of the remaining $x_\ell$ are in $H_1$, we have $H_1\cap H_2$ is non-empty,  and also $H_2\subset C_2$.
We continue this process and at each stage $H_1\cap H_2\cap \cdots H_N$ is non-empty and a subset of $C_N$, and hence $\cap_j H_j$ is non-empty (by (2))
and gives an element of $\cap_j C_j \subset \cap_j B_j$, which is a contradiction.  
 \end{proof} 

We will need one more result before we can apply Lemma \ref{L:countadd} to prove Theorem \ref{analytic-main-measure} and Theorem~\ref{T:weak-convergence}.

\begin{lemma}\label{L:Clim}
Let $\category$ be a \Diamantine category and  $\W_0\sub\W_1\sub \cdots$ be formations. Let $G_i$ for $i\in \N$ be objects of $\category$ such that for $j\geq i$, we have
$G_j^{\W_i}\isom G_i$.  Then given any narrow formation $\W$ of finite complexity, there exists an $N$ such that for $i,j\geq N$, we have $G_i^\W\isom G_j^\W$.
\end{lemma}

\begin{proof}
We induct on the complexity of $\W$.  If $\W$ has complexity $0$, then $G_m^\W$ is the minimal quotient of $G_m$, and for $n\geq m$, since $G_m$ is a quotient of $G_n$, they have the same minimal quotient.  Suppose the claim is true for any narrow formation  of complexity $k$.  Let $\W$ be a narrow formation of complexity $k+1$ and let $\W'$ be the elements of $\W$ of complexity $\leq k$.

Consider only $i$ such that $G_i^{\W'}$ has stabilized to $G$, and note that there is a semisimple epimorphism $G_i^\W\ra G$ by Lemma~\ref{L:complexity}, and thus $G_i^\W$ is the join of all the simple $X\ra G$ such that $G_i^\W\ra G$ factors through $X\ra G$ by Lemma~\ref{semisimple-equivalence}.  Given $G$, there are only finitely many $X\ra G$ such that $X\in \W$ since $\W$ is narrow.  Some of these $X$ eventually appear in a $\W_j$ and some never do.  Consider only $i$ large enough such that all $X\in \W$ with a simple $X\ra G$ that will ever appear in a $\W_j$ appear in $\W_i$. 
If $X$ is not in any $\W_i$, then it is not ever a quotient of $G_i$ for any $i$.
 If $X\in \W_i$, we have that all epimorphisms $G_j \ra X$ for $j\geq i$ factor through $G_j^{\W_i}\isom G_i$.  Thus $G_j^\W$ is the join of the same elements as $G_i^\W$, which proves the lemma.
\end{proof}

 \begin{proof}[ Proof of Theorem \ref{analytic-main-measure} and Theorem~\ref{T:weak-convergence}] 
Let $\nu$ be a measure on $\Prf$ (for the given $\sigma$-algebra) with the specified moments.
 For a level $\C$ let $\nu_\C$ be the push-forward of $\nu$ along the map $X\mapsto X^\C$.
For $G\in \C$, we have $|\Sur(X,G)|=|\Sur(X^\C,G)|$, and thus
$\nu_\C$ has the same $G$-moments as $\nu$, for $G\in\C$.  
We have $\nu(U_{\C,F})=\nu_\C(\{F\})=\sum_{G \in \C} \hat{\mu}(F,G) m_G$ by Theorem \ref{T:general}, and so the sum is non-negative.   This shows the ``only if'' of (Existence) and the equalities claimed in (Uniqueness).

Note that well-behavedness implies that the  sums $\sum_{G \in \C} \hat{\mu}(F,G) m_G$ are finite.
By definition, $C$ has only countably many isomorphism classes of minimal objects so, given a level $\C$, by Lemma~\ref{dim-C-finite}, we have that $\C$ is at most countable.

Consider the set $\tilde{\mathcal{A}}$ of sets that are a union of $U_{\C,H}$ for a single level $\C$ and varying $H$.
Note such a union is necessarily disjoint and at most countable.
For levels $\C \sub \C'$ we have $U_{\C,H}$ is a union of the $U_{\C',H'}$ for all $H' \in \C'$ with $(H')^\C\isom H$.
So for two sets in $\tilde{\mathcal{A}}$, we can always refine to a level containing the levels at which each was defined, and assume they are defined at the same level.  From this we can see that $\tilde{\mathcal{A}}$ is an algebra of sets. 
Also, $U_{\C,H}$ and the $\tilde{\mathcal{A}}$ generate the same $\sigma$-algebra.
Moverover,  if we have a measure $\nu$ with the specified moments, we have that 
$\nu$ is determined on $\tilde{\mathcal{A}}$.  Since for any level $\C$, we have that $\mathcal{P}$ is a countable union of 
$U_{\C,H}$, each with finite measure, we have that $\nu$ is $\sigma$-finite.  
Carath\'{e}odory's extension theorem then tells us that there is a unique measure 
on the $\sigma$-algebra generated by the $U_{\C,H}$ that agrees with 
that agrees $\nu$ on  $\tilde{\mathcal{A}}$, finishing the proof of (Uniqueness).

For simplicity, we will replace our category $C$ with the category with the same objects, but whose morphisms are the epimorphisms from $C$.  Nothing in the theorems ever considers morphisms that are not epimorphisms. We will now call this new category $C$.  The fact that any two quotients have a meet implies that any two objects with a morphism between them have a morphism to the same minimal object.  
This category $C$ is the disjoint union of at most countably many categories $C_i$ each with a unique isomorphism class of minimal object.   Similarly, $\mathcal{P}$ is an at most countable disjoint union of sets $\mathcal{P}_i$ such that each $U_{\C,H}$ is a subset of one of the $\mathcal{P}_i$.
Each moment only depends on the measure restricted to one of the $\mathcal{P}_i$.  Thus to prove existence of
a measure with the given moments, it suffices to prove existence for each $\mathcal{P}_i$.  Here we crucially use that there are most countably many $\mathcal{P}_i$ so that we can define the measure of measurable subset $S$ of $\mathcal{P}$ to be the sum over $i$ of the measures of $S\cap \mathcal{P}_i$. 

So now we assume that $C$ has a unique isomorphism class of minimal object.  
We assume that for every level $ \C$ and $G\in \C$, we have $\sum_{G \in \C} \hat{\mu}(F,G) m_G\geq 0$, and we will show a measure $\nu$ with the indicated moments actually exists. 
Theorem \ref{T:general} implies that for each level $\C$ a  measure $\nu_\C$ exists with the specified moments.
The uniqueness in Theorem \ref{T:general} implies for $\C\sub \C'$ the measure $\nu_{\C'}$ pushes forward to $\nu_\C$ along the map $F\mapsto F^{\C}$.

We will define a pre-measure $\nu$ on $\tilde{\mathcal{A}}$ by letting, for any subset $V\sub \C$, 
$$\nu\left( \bigcup_{H\in V} U_{\C,H} \right)=\nu_\C\left( V \right).$$ 
If we have $A\in  \tilde{\mathcal{A}}$ that can be defined at two different levels $\C,\C'$, i.e. 
$$
A=\bigcup_{H\in V} U_{\C,H} =\bigcup_{H'\in V'} U_{\C',H'} ,
$$
then for any level  $\C'' \supset \C,\C'$, let $V''=A\cap \C''$.
We have that $X\in A$ if and only if $X^{\C''}\in V''$.  
We have that $V''$ is the preimage of $V$ under the map $F\mapsto F^\C$ from $\C''$ to $\C$ (and similarly for $V'$, $\C'$).
Thus the claim about push-forwards above gives
$$
\nu_\C\left(V \right)=\nu_{\C''}\left( V'' \right)=\nu_{\C'}\left( V' \right),
$$
and our definition above of $\nu$ on $\tilde{\mathcal{A}}$ is well-defined.
A similar argument refining to a common level shows that $\nu$ is finitely additive.

We have $\nu$ a finitely additive non-negative function on the algebra $\tilde{\mathcal{A}}$ of sets.
We now wish to show that it is countably additive.  Suppose we have that $A_0\in \tilde{\mathcal{A}}$ is the disjoint union of $A_1,A_2,\dots \in \tilde{\mathcal{A}}$.
Let $A_i$ be defined at level $\C_i$, and let $\C_{\leq i}$ be the level generated by $\C_j$ for $j\leq i$.
Let ${\mathcal{A}}$ be the algebra of sets that are (necessarily at most countable) disjoint unions of $U_{\C_{\leq i},H}$ for some fixed $i$.
So $A_0,A_1,\dots\in {\mathcal{A}}$.
Then we apply Lemma~\ref{L:countadd} to ${\mathcal{A}}$, with $\mathcal{L}_\ell$ the set $\{U_{\C_{\leq \ell},H}\,|\,H\in \C_{\leq \ell}\}$.
To show hypothesis (2) in Lemma~\ref{L:countadd}, we need to show that 
for a collection $G_i$ with $G_j^{\C_{\leq i}}=G_i$ for $j\geq i$, 
there is $X\in \Prf$ with
$X^{\C_{\leq i}}=G_i.$  We construct such an $X$ by letting $X^\C$
be $G_i^\C$ for $i$ sufficiently large such that this isomorphism class has stabilized (by Lemma~\ref{L:Clim}).
We can see that (3) is satisfied since for $j\geq i$, we have $X^{\C_{\leq i}}=(X^{\C_{\leq j}})^{\C_{\leq i}}$.

Since $\nu_\C$ is a measure, we have that $\nu_\C(\C)=\sum_{G \in \C} \nu_\C(G),$
and the claim about push-forwards above gives that $\nu_\C(\C)=\nu_{\{0\}}(0)$, where $0$ is the minimal object of $C$.  
  Thus hypothesis (4) in Lemma~\ref{L:countadd}
is satisfied and we have that $\nu$ is countably additive on ${\mathcal{A}}$.
Thus $\sum_{i\geq 1} \nu(A_i)=\nu(A_0)$, which implies $\nu$ is countably additive on $\tilde{\mathcal{A}}$.
Hence by Carath\'{e}odory's extension theorem $\nu$ extends to a measure on $\Prf$. 
As above, the moments of $\nu$ are determined by the push-forwards $\nu_\C$, which have the desired moments.  

Finally, we prove Theorem~\ref{T:weak-convergence}.
First, we will show that if $I$ is countable, then every open set in the topology generated by the $U_{\C,F}$ is a disjoint union of sets of the form $U_{\C,F}$.  We consider $I$ as a subset of $\mathcal P$, by mapping $G\in I$ to $\{G^\C\}$.
Since there are countably many finite subsets of $I$, there are countably many levels $\C_1,\dots$, and let $\C_{\leq i}$ be the level generated by $\C_j$ for $j\leq i$.  
We note that the topology on $\mathcal{P}$ is actually generated by the $U_{\C_{\leq i},H}$, since $U_{\C_i,F}$ is a union of $U_{\C_{\leq i},H}$ over varying $H$.
Given an open $V$, around each $F\in V\cap I$ we consider the minimal $i_F$ for which  $U_{\C_{\leq i_F},F^{\C_{\leq i_F}}}\sub V$.
If $X\in V$, then $X\in U_{\C_{\leq i},H} \sub V$ for some $i$ and $H\in \C_{\leq i}$.  Also $i_H\leq i$, so 
$X\in U_{\C_{\leq i_H},H^{\C_{\leq i_H}}} \sub V.$
So $V$ is the union of the $U_{\C_{\leq i_F},F^{\C_{\leq i_F}}}$ over $F\in V\cap I$.   If the two sets $U_{\C_{\leq i_F},F^{\C_{\leq i_F}},}$ and $U_{\C_{\leq i_H},H^{\C_{\leq i_H}},}$ intersect,  then one is contained in the other, and by the minimality of $i_F$ and $i_H$, the two sets must be equal.  
This shows that $V$ is a countable disjoint union of $U_{\C,F}$.  Also, we can put a metric
on $\mathcal{P}$ such that the distance between two objects $X,Y$ is the $1/i$ for the minimal $i$ such that $X^{\C_{\leq i}}\isom Y^{\C_{\leq i}}$, or $2$ if such an $i$ does not exist.  We can check this metric induces the topology that is generated by the $U_{\C_{\leq i},H}$

By Theorem \ref{T:general}, we have, for every level $\C$ and $F\in \C$,
 $$ \lim_{t\to\infty} \nu^t(U_{\C,F}) = v_{\C,F} ,$$
and thus  $v_{\C,F}\geq 0.$ Thus a measure $\nu$ on $\mathcal{P}$ with moments $M_G$ exists and
   $$ \lim_{t\to\infty} \nu^t(U_{\C,F}) =\nu(U_{\C,F})$$
   for every level $\C$ and $F\in \C$.
Since any open set in our topology is a disjoint union of basic opens,  using Fatou's lemma and then the Portmanteau theorem proves Theorem~\ref{T:weak-convergence}.
\end{proof}

\subsection{Pro-isomorphism classes and pro-objects}

We now explain the relationship between pro-isomorphism classes and the usual notion of pro-object.

One defines a \emph{pro-object} of $\category$ to be a functor $\mathcal F$ from a small category $I$ to $\category$, where the category $I$ is \emph{cofiltered} in the sense that for any two objects of $I$ there is an object with a morphism to both, and for any two morphisms $f,g$ between objects $A,B\in I$ there is an object $C\in I$ with a morphism $h\colon C\to A$ such that $f\circ h = g\circ h$ \cite[Definition 6.1.1 and Definition 3.1.1]{KashiwaraSchapira}.
We think of a pro-object as the formal inverse limit of the diagram $(I, \mathcal F)$. 
The morphisms between two pro-objects $(I, \mathcal F) \to (J, \mathcal G)$ is given by the formula \cite[(2.6.4)]{KashiwaraSchapira} \[ \Hom((I,\mathcal F), (J, \mathcal G)) = \varprojlim_{j \in J} \varinjlim_{i\in I}  \Hom ( \mathcal F(i), \mathcal G(j) ) .\] It is not hard to define the identity morphism and composition for morphisms, making the pro-objects into a category (the pro-completion of $C$).  We can view an object of $\category$ as a pro-object using a category $I$ with one object and only the identity morphism.

\begin{definition}\label{def-small} We say a pro-object is \emph{small} if it arises (up to isomorphism) from a diagram $(I, \mathcal F)$ where every $\mathcal F(f)$ is an epimorphism for every morphism $f$ of $I$, and if for each $G\in C$, the set of epimorphisms from $(I, \mathcal F)$ to $G$ is finite. \end{definition}

 This definition is in analogy to the definition of small profinite groups (see \cite[Section 16.10]{Fried2008}, and also Lemma~\ref{small-characterization} below).

We now construct the bijection between isomorphism classes of small pro-objects and pro-isomorphism classes. We first prove a quick lemma describing the quotients of a pro-object
\begin{lemma}\label{quotient-profinite-modular} 
Let $\category$ be a \Diamantine category and $X$ be a small pro-object of $\category$. 
The poset of quotients of $X$ in $\category$ is a modular lattice. \end{lemma}

\begin{proof}  Let $X$ arise from a diagram $(I, \mathcal F)$ with all morphisms in the image $\mathcal F$ epimorphisms. We first note that, for $G \in C$, we have  $\Hom (X,G) = \varinjlim_{ i\in I} \Hom(\mathcal F(i), G)$ so every morphism $X \to G$ arises from some morphism $\mathcal F(i) \to G$.  Next observe that a morphism $X \to G$ is an epimorphism if and only if it arises from an epimorphism $\mathcal F(i) \to G$, i.e. that $\Sur( X, G) =  \varinjlim_{ i\in I} \Sur(\mathcal F(i), G)$. The case ``only if" is immediate and ``if" uses that all morphisms in the image of $\mathcal F$ are epimorphisms.

It follows that the poset of quotients of $X$ is the forward limit over $i\in I$ of the poset of quotients of $\mathcal F(i)$. For $\mathcal F(f) \colon \mathcal F(i) \to \mathcal F(j)$ an epimorphism, the induced map from the poset of quotients of $\mathcal F(j)$ to the poset of quotients of $\mathcal F(i)$ is an inclusion of modular lattices (in particular, respects meets and joins). Since $I$ is cofiltered, it follows that the limit is a modular lattice.
\end{proof}

 \begin{lemma}\label{small-finiteness} Let $\category$ be a \Diamantine category and $X$ be a small pro-object of $\category$. Then, for any narrow formation $\W$ of finite complexity the object $X$ has a maximal quotient in $\W$, i.e. an object $X^\W$ of $\W$ that is a quotient of $X$ such that every other such quotient factors through it. \end{lemma}
 
 \begin{proof}  Let $\W_n$ be the formation of elements of $\W$ of complexity at most $n$.
  We will prove for all $n$ that $X$ has a maximal quotient in $\W_n$ by induction on $n$, and then the claim follows from the finite complexity of $\W$. For the base case, we note that  the cofiltered and modular lattice properties imply that all $\mathcal{F}(i)$ have the same minimal quotient, and 
 we see that this is $X^{\W_0}$, the maximal  quotient of $X$ in $\W_0$.
 
 For the induction step, let $X^{\W_{n-1}}$ be a maximal  quotient of $X$ in $\W_{n-1}$. 
  Let $K\in \W_n$ be a quotient of $X$. Then, by Lemma \ref{L:complexity}, $K \to K^{\W_{n-1}}$ is semisimple. 
    Since $K^{\W_{n-1}}$ is in $\W_{n-1}$ and is a quotient of $X$, 
    in the lattice of quotients of $X$ we have
    $K^{\W_{n-1}}\leq X^{\W_{n-1}}$
Thus
$K\ra K\wedge X^{W_{n-1}}$ is semisimple by   Lemma \ref{semisimple-equivalence},  since  
$K\wedge X^{W_{n-1}} \geq K^{\W_{n-1}}$.
  Thus
        $K \vee X^{\W_{n-1}} \to X^{\W_{n-1}} $ is semisimple by Lemma  \ref{dit}. It suffices to show that there is a maximal quotient of $X$ in $\W_n$ with a semisimple morphism to $X^{\W_{n-1}} $, as then $K \vee X^{\W_{n-1}} $, which has those properties, will factor through it, and then $K$ will as well.
 
 To do this, consider a quotient $H\in \W_n$ of $X$ such that $H\ra X^{\W_{n-1}}$ is semisimple of fixed dimension $d$.
   There are finitely many possible isomorphism classes of such objects $H$ by Lemma \ref{dim-C-finite}. Each one admits finitely many morphisms from $X$ by the smallness assumption, so there are finitely many such quotients of $X$. 
   In particular, taking $d=1$, there are finitely many quotients of $X$ in $\W_n$ with a simple epimorphism to $X^{\W_{n-1}}$.  If $H \to X^{\W_{n-1}}$ is semisimple of dimension $d$, then it factors through at least $d$ different simple morphisms $G_i \to X^{\W_{n-1}}$ of quotients of
$H$ (by Lemma~\ref{semisimple-lower-bound}) and hence quotients of  $X$.
  So, the dimension of any semisimple morphism from a quotient of $X$ in $\W_n$ to $X^{\W_{n-1}}$ is bounded in terms of $X$, $\W$, and $n$. 
Consider the quotients of $X$ in $\W_n$ with a semisimple morphism to $X^{\W_{n-1}} $.  There are finitely many possible dimensions, and then finitely many of each dimension. Thus, such quotients form a finite set and the join of two elements of this set is an element of this set
(by Lemmas~\ref{quotient-profinite-modular} and \ref{complemented-equivalence}),
 it must have a maximal member, completing the induction step.
  \end{proof}

\begin{lemma}\label{pro-object-existence} 
Let $\category$ be a \Diamantine category. 
For each pro-isomorphism class $X$ of $\category$, there exists a small pro-object $Y$ of $\category$, unique up to isomorphism, such that $Y^\C \cong X^\C$ for each level $\C$.

In other words, the map from the set of isomorphism classes of small pro-objects to $\Prf$ that sends $Y $ to $(Y^\C)$ is a bijection. \end{lemma}

\begin{proof} 
First, for each level $\C$ we chose an object in the isomorphism class $X^\C$, and by slight abuse of notation we will now write $X^\C$ for this object, and we make analogous choices throughout this proof.
We will next show the existence of an assignment to each pair $\C_1,\C_2$ of levels with $\C_1 \sub\C_2$ an epimorphism $\pi_{\C_1}^{\C_2} \colon X^{\C_2} \to X^{\C_1}$ such that whenever $\C_1 \leq \C_2 \leq \C_3$ we have $ \pi_{\C_1}^{\C_2} \circ \pi_{\C_2}^{\C_3} = \pi_{\C_1}^{\C_3}$. Since there are finitely many epimorphisms $X^{\C_2} \to X^{\C_1}$ for each $\C_2, \C_1$, 
to check that these infinitely many conditions can be satisfied, by the compactness of an infinite product of finite sets it suffices to check that every finite subset of them can be satisfied. Any finite set of equations of the form $ \pi_{\C_1}^{\C_2} \circ \pi_{\C_2}^{\C_3} = \pi_{\C_1}^{\C_3}$ involves only finitely many levels $\C_i$, and these finitely many levels $\C_i$ are contained in a single level $\overline{\C}$.

For each $\C_i$, fix an epimorphism
$\pi_i:X^{\overline{\C}} \ra ( X^{\overline{\C}} )^{\C_i}$ realizing the definition of $( X^{\overline{\C}} )^{\C_i}$.
Then for every $i,j$ with $\C_i\sub \C_j$, there is a unique epimorphism 
$\pi^j_i: ( X^{\overline{\C}} )^{\C_j} \ra ( X^{\overline{\C}} )^{\C_i}$ such that $\pi^j_i \pi_j=\pi_i$.
We define $\pi_{\C_i}^{\C_j}$ for each $i,j$ with $\C_i \sub \C_j$ by composing $\pi^j_i$
with  fixed isomorphisms  $( X^{\overline{\C}} )^{\C_k} \isom X^{\C_k}$ chosen for each $k$. One can check these  $\pi_{\C_i}^{\C_j}$ satisfy all of the required equations involving only $\C_i \subseteq \overline{\C}$, completing the  argument for existence of the $\pi_{\C_1}^{\C_2}$.

Any epimorphism $X^{\C_1}\ra X^{\C_1}$ is an isomorphism by the finiteness of quotient posets.  
Thus any epimorphism $\pi_{\C_1}^{\C_2} \colon X^{\C_2} \to X^{\C_1}$ realizes $X^{\C_1}$ as the maximal quotient in $\C_1$ of $X^{\C_2}$,
since it factors through the maximal  quotient in $\C_1$ of $X^{\C_2}$ (which we know, by the definition of a pro-isomorphism class, is abstractly isomorphic to 
$X^{\C_1}$).

Define a category $I$ to have one object for each level $\C$ and a single morphism $\C_2 \to \C_1$ if and only if $\C_1 \sub \C_2$. 
Let $Y$ be the pro-object defined using the index category $I$ and the functor sending $\C $ to $X^\C$, and the unique morphism $\C_2\to \C_1$ to $\pi_{\C_1}^{\C_2}$. Let's check that $Y$ is a small pro-object and that $Y^{\C} \isom X^{\C}$. It suffices to check that $Y^{\C} \isom X^{\C}$ for each level $\C$, since taking $\C$ to be generated by ${H}$, the number of epimorphisms $Y \to H$ is then bounded by the number of epimorphisms $X^{\C}\to H$, which is finite.

To do this, let $F$ be any quotient in $\C$ of $Y$. Then $F$ must be a quotient of $X^{\C'}$ for some $\C'$ by definition of $Y$.
  Let $\tilde{\C}$ be the level generated by $\C'$ and $\C$.  
  Every epimorphism $X^{\tilde{\C}} \ra F$ factors through $\pi_{\C}^{\tilde{\C}}: X^{\tilde{\C}} \ra X^{{\C}}$ by our observation above that this is the maximal
quotient in $\C$ of $X^{\tilde{\C}}$.
  Thus every quotient in $\C$ of $Y$ factors through $Y \ra X^{{\C}}$, which
completes the proof of existence.

It remains to show uniqueness up to isomorphism. In other words, we fix a small pro-object $Z$ with $Z^\C \cong X^\C \cong Y^\C$ for all $\C$ and must construct an isomorphism $Y \to Z$. 
As above, we can  choose compatible maps $\pi^{\C_2}_{\C_1}: X^{\C_2} \to X^{\C_1}$ for all $\C_1\sub \C_2$.
For each level $\C$, we choose a map $\tau_\C : Z \ra Z^\C$ (which realizes $Z^\C$ as the maximal quotient in $\C$ of $Z$), and then define
 $\tau^{\C_2}_{\C_1}: Z^{\C_2} \to Z^{\C_1}$ for all $\C_1\sub \C_2$ so that $\tau_{C_1}=\tau^{\C_2}_{\C_1} \tau_{C_2}$.
 Next we will 
choose a system of isomorphisms $X^\C \to Z^\C$ (for each $\C$) that are compatible in the sense that, for $\C_2 \sub \C_1$, they form a commutative diagram with $\pi^{\C_2}_{\C_1}$ and  $\tau^{\C_2}_{\C_1}$. 
 This again is choosing an element of an infinite product of finite sets satisfying infinitely many conditions that each depend on only finitely many terms of the product, so it suffices to check that any finite subset of the conditions can be satisfied.
To do this, we chose a level $\C$ containing all levels appearing in a finite list of conditions,  fix an isomorphism $\phi: X^\C \to Z^\C$, and note that since $\pi^{\C}_{\C_i}: X^\C \ra X^{\C_i}$and $\tau^{\C}_{\C_i} \phi: X^\C \ra Z^{\C_i}$ both give the maximal quotient in $\C_i$ of $X^\C$, then they must be isomorphic, and there is a unique isomorphism $\phi_{\C_i}: X^{\C_i} \ra Z^{\C_i}$ such that $\phi_{\C_i} \pi^{\C}_{\C_i}=\tau^{\C}_{\C_i} \phi$.  Then these $\phi_{\C_i}$ are compatible in the above sense.

Now we construct the isomorphism $Y \to Z$. 
Writing $Z = (I, \mathcal F)$, we must for each $i\in I$ choose a level $ \C$ and a map $X^\C \to   \mathcal F(i)$, in a way compatible with all the morphisms of $I$. We can choose $\C$ to be the level generated by $\mathcal F(i)$.  Then there is a unique epimorphism $Z^{\C} \to \mathcal F(i)$
commuting with $\tau_\C$ and the natural map $Z\ra F(i)$.  Taking the composition with $\phi_{\C}$ constructed above, we have
 $X^{\C}  \to \mathcal F(i)$, and one can check these are compatible with the morphisms of $I$. 
 
 To map $Z \to Y$, we must for each level $\C$ choose $i\in I$ and a map $\mathcal F(i) \to X^\C$, compatible with all inclusions between levels. We 
have $\phi_C \tau_C: Z \ra X^\C$, so  there exists $i \in I$ with a morphism $\mathcal F(i) \to X^{\C}$, and choose that $i$ and morphism. 

It is straightforward to check that these two maps are inverses, and so $Y\isom Z$.
\end{proof}

Using Lemma~\ref{pro-object-existence}, we can view $\mathcal P$ as a set of small pro-objects.  
For a formation $\W$ of $C$ and
 $G \in \W$, let  $U_{\W, G}$ be the set of $X\in \mathcal P$ such that the maximal quotient of $X$ in $\W$ exists and is isomorphic to $G$.  By Lemma~\ref{small-finiteness}, these maximal quotients will always exist when $\W$ is a narrow formation of finite complexity.
 For a general formation, they may not exist, e.g. $\Z_p=\varprojlim_{n}\Z/p^n\Z$ does not have a maximal quotient in the narrow formation of finite abelian $p$-groups. 
We can deduce more general formulas for our unique measures on these $U_{\W, G}$.

\begin{corollary}\label{la-output} Let $\category$ be a \Diamantine category.
Let $(M_G)_{G}\in \R^{\category_{/\isom}}$ be well-behaved.
Let $\W$ be narrow formation of $C$.
Then $U_{\W,F}$ is open and closed in $\mathcal P$,  and $(M_G)_{G}$ is well-behaved at $\W$, and
the unique measure $\nu$ of Theorem~\ref{analytic-main-measure}, if it exists, 
 satisfies $\nu(U_{\W,F})=v_{\W,F}$ for all $F \in \W$. \end{corollary} 

\begin{remark}Even if $\W$ is not narrow, but merely is a union of levels $\C_i$ for $\C_0\sub \C_1\sub\cdots$, a part of the conclusion holds: $U_{\W, F}$ is a closed set because it is the intersection of the $U_{ \C_i, F^{\C_i}}$, and $\nu ( U_{\W, F}) = \lim_i \nu ( U_{\C_i, F^{\C_i}}) = \lim_i v_{\C_i, F^{\C_i}}$.
As long as the sum defining $v_{\W,F}$ is absolutely convergent, we have 
$ \lim_i v_{\C_i, F^{\C_i}} = v_{\W,F}$ and hence $\nu ( U_{\W, F})=v_{\W,F}$.\end{remark}

\begin{proof} Fix $F \in \W$.  Let $\C$ be the level generated by $F$ and all the objects in $\W$ with a simple epimorphism to $F$ (by assumption, a finite set). 
So $\C\subseteq\W$.

We can apply Theorem~\ref{analytic-main-measure} after checking that $U_{\C, F} = U_{\W, F}$ and $v_{\C, F} = v_{\W,F}$.

To check $U_{\C, F} = U_{\W,F}$, let $X$ have maximal quotient in $\C$ isomorphic to $F$. Then $F$ is a quotient of $X$ in $\W$, and if it is not the maximal such, there must be a quotient  $G$ of $X$ with $G>F$ and $G\in \W$. Factoring $G \to F$ into a composition of simple morphisms, we see that there must exist a quotient $G'$ of $X$ with $G' > F$ with $G' \to F$ simple. Then $G'$ lies in $\C$ by construction, contradicting the assumed maximality of $F$. Conversely, if $X$ has a maximal quotient in $\W$ isomorphic to $F$, then any quotient of $X$ in $\C$ must lie in $\W$ unless it is a minimal object, in which case it must be the minimal quotient of $F$ and hence lie in $\W$ anyways. Thus we must also have $F$ the maximal quotient in $\C$ of $F$. Hence $X \in U_{\C, F}$ if and only if $X \in U_{\W, F}$, as desired.  In particular since $U_{\C, F}$ is open and closed so is $U_{\W, F}.$

To check $v_{\C, F} = v_{\W,F}$, it suffices to check that, for $G \in \category$ with a semisimple morphism to $F$, we have $G \in \C$ if and only if $G \in \W$. The ``only if" direction is clear since $\C \subseteq\W$. For ``if", since $G$ is semisimple, it is the join of objects with simple morphisms to $F$. Since these objects are quotients of $G$, they lie in $\W$, so since they have simple morphisms to $F$, they lie in $\C$, and since $G$ is their join it lies in $\C$ also, as desired.   Similarly, each expression that needs to converge absolutely for the well-behavedness of $(M_G)_{G}$ at $\W$ converges absolutely because of the well-behavedness at some level.  Theorem~\ref{analytic-main-measure} says that
$\nu(U_{\C,F})=v_{\C,F}$, which then implies $\nu(U_{\W,F})=v_{\W,F}$.
\end{proof}

\subsection{Pro-objects concretely}

 \begin{lemma}\label{small-characterization} Let  $\category$ be a \Diamantine category such that 
  every morphism can be written as a composition $\iota \pi$, where $\pi$ is an epimorphism and $\iota$ is a monomorphism,
 and every object has finitely many subobjects, up to isomorphism. Then for $(I,\mathcal F)$ a pro-object of $\category$, the following are equivalent
 
 \begin{itemize}
 
 \item $(I,\mathcal F)$ is small.
 
 \item $(I,\mathcal F)$ has finitely many epimorphisms to each object of $\category$.
 
 \item $(I,\mathcal F)$ has finitely many morphisms to each object of $\category$.
 \end{itemize}
 
  \end{lemma}
 
 \begin{proof} We will first show that every pro-object of $\category$ is isomorphic to one of the form $(J, \mathcal G)$ where all morphisms in the image of $\mathcal G$ are epimorphisms.
  
 Fix a pro-object $(I,\mathcal F)$. For every $j \in I$ mapping to $i$, we can express $\mathcal F(j) \to \mathcal F(i)$ in at least one way as an epimorphism followed by a monomorphism.
 The monomorphism gives a subobject of  $\mathcal F(i)$. Call a subobject of $\mathcal F(i)$ arising this way a \emph{valid} subobject. We are interested in the valid subobjects that are minimal for the natural ordering on subobjects (restricted to valid subobjects). 
 
  Let $J$ be the category of triples of of an object $i\in I, $ an object $H \in \category,$ and a monomorphism  $\pi \colon H \to \mathcal F(i)$,  such that $(H, \pi)$ is a minimal valid subobject of $\mathcal F(i)$.  We take the morphisms in $J$ from $(i_1, H_1,\pi_1) \to (i_2, H_2, \pi_2)$ to be given by pairs of a morphism $f \colon i_1 \to i_2$ and an epimorphism $g \colon H_1 \to H_2$ such that $ \mathcal F(f) \circ \pi_1 = \pi_2 \circ g$. Let $\mathcal G$ be the functor $J \to \category$ sending $(i, H, \pi) $ to $H$ and the morphism given above to $g$. 
 
 We claim that $J$ is cofiltered and $(J, \mathcal G)$ is isomorphic to $(I, \mathcal F)$. 
 
 To show this, we first check that for each $i \in I$ there exists a minimal valid subobject of $\mathcal{F}(i)$.  This is because there is always at least one valid subobject $\mathcal F(i)$, and by assumption the number of subobjects is finite, so there is always a minimal element.
 
Next we check that for a minimal valid subobject $H$ of $\mathcal{F}(i)$, for $j$ mapping to $i$ such that $\mathcal F(j) \to \mathcal F(i)$ factors through an epimorphism $\mathcal F(j) \to H$,  for all $k \in I$ mapping to $j$, the given map $\mathcal F(k) \to \mathcal F(i)$ factors through an epimorphism $\mathcal F(k) \to H$.   (We call such a map $j\ra i$ a witness of the validity of $H$.)
Indeed, for every $k$ mapping to $j$, the map $\mathcal F(k) \to \mathcal F(i)$ factors through $H$. So, as a morphism, $\mathcal F(k) \to H$ is the composition of an epimorphism followed by a monomorphism.
If that monomorphism is not an isomorphism then it defines a proper subobject of $H$, hence a subobject of $\mathcal F(i)$ smaller than $H$, which $\mathcal F(k)$ factors through, contradicting minimality. Thus that monomorphism is an isomorphism so $\mathcal F(k) \to H$ is an epimorphism.

 We now check that $J$ is cofiltered. Given two objects $(i_1,H_1,\pi_1)$ and $ (i_2, H_2,\pi_2)$, we need to find an object mapping to both. 
 Choose $j_1$ mapping to $i_1$ witnessing the validity of $H_1$ and 
 and choose $j_2$ mapping to $i_2$ witnessing the validity of $H_2$. By the cofiltered property of $I$, we can choose $j_3$ mapping to $j_1$ and $j_2$. By the above, we can find $(j_3, H_3, \pi_3) \in J$.  Let's check that any such $(j_3, H_3,\pi_3)$ maps to $(i_1,H_1,\pi_1)$ (lifting the chosen map $f: j_3\ra i_1$). By the above, $  \mathcal F(f) \colon  \mathcal F(j_3) \to  \mathcal F(i)$ factors through an epimorphism $\mathcal F(j_3) \to H_1$. Taking $g$ to be the composition of this epimorphism with $\pi_3$, we see that $\pi_1 \circ g  = \mathcal F(f) \circ \pi_3$. To see that $g$ is an epimorphism, we choose $j_4$ such that $\mathcal F(j_4) \to \mathcal F(j_3)$ factors through $H_3$, and note that $\mathcal F(j_4) \to H_1$ is an epimorphism since otherwise we would obtain a valid subobject smaller than $H_1$, so $H_3 \to H_1$ is an epimorphism, as desired.  Then $(j_3, H_3, \pi_3) $ maps to $(I_2, H_2,\pi_2)$ by a similar argument, verifying the first part of the cofiltered assumption.
 
 Now we check the second part of the cofiltered assumption. 
Consider two maps $(f,g)$ and $(f',g')$ from $(i_1,H_1,\pi_1)$ to $ (i_2, H_2,\pi_2)$. 
  Since $I$ is cofiltered, we can pick $j_5$ and a map $f_{51}: j_5 \ra i_1$ such that $f f_{51}=f' f_{51}$, and further by the cofiltered assumption we can choose such a $j_5$ and $f_{51}$ such that $j_5$ maps to $j_1$ and $j_2$.  By the above, any $(j_5, H_5,\pi_5)\in J$ maps to $(i_1,H_1,\pi_1)$ via $(f_{51},h)$ for some $h$.
  We have 
  $$\pi_2gh=\mathcal{F}(f) \pi_1 h=\mathcal{F}(f) \mathcal{F}(f_{51}) \pi_5=\mathcal{F}(f') \mathcal{F}(f_{51}) \pi_5=
  \mathcal{F}(f') \pi_1 h=\pi_2g'h.$$
   Since $\pi_2$ is a monomorphism, it follows that $gh=g'h$ and $(f_{51},h)(f,g)=(f_{51},h)(f',g'),$ which shows the second part of the cofiltered assumption.

 Now we check that $(I, \mathcal F)$ and $(J, \mathcal G)$ are isomorphic. To construct a morphism $(I, \mathcal F) \to (J,\mathcal G)$, we must choose for each $(i, H, \pi) \in J$ a $j \in I$ and a map $ \mathcal F(j) \to \mathcal G ( i,H, \pi) = H$. We can take $j$ and this map to be a witness to the validity of $(H,\pi)$. To construct a morphism $(J,\mathcal G) \to (I,\mathcal F)$, we choose for each $i \in I$ a $(j, H , \pi) \in J$ and a map $H = \mathcal G( j,H, \pi)  \to \mathcal F(i)$. We choose $j = i$ and $(H,\pi)$ a minimal valid subobject, and the map $\pi$. It is straightforward to check these two maps are well-defined and are inverses.

We have now shown that every pro-object of $\category$ is isomorphic to one of the form $(J, \mathcal G)$ where all morphisms in the image of $\mathcal G$ are epimorphisms.
 
Thus, a pro-object is small if and only if it has finitely many epimorphisms to each object of $\category$. It remains to show this condition is equivalent to having finite many morphisms to each object of $\category$.  Having finitely many morphisms implies having finitely many epimorphisms, so it suffices to show the reverse implication. Each morphism $(I,\mathcal F) \to G$ arises from a map $\mathcal F(i) \to G$ which factors as an epimorphism $\mathcal F(i) \to H$ followed by a monomorphism $H \to G$. The morphism $(I,\mathcal F) \to H$ is manifestly an epimorphism. There are finitely many monomorphisms $H \to G$, up to isomorphism, by assumption, and for each $H$, finitely many epimorphisms from $(\mathcal F, I)$, proving that there are then finitely many morphisms from $(I,\mathcal F) \to G$.
This concludes the proof of the equivalence of the three conditions in the lemma.
 \end{proof}
 
 \begin{remark} Let $\category$ be the category of finite groups, which clearly satisfies the hypotheses of Lemma~\ref{small-characterization}.
The category of pro-objects in the category of finite groups is equivalent to the category of profinite groups \cite[Example (b) on p. 254 and Theorem on p. 245]{Johnstonestone}. A profinite group is small in our sense if and only if it has finitely many continuous surjective homomorphisms to any finite group by Lemma \ref{small-characterization}. By a standard argument, this is equivalent to having finitely many open subgroups of index $n$ for each $n$, which is the usual definition of a small profinite group (e.g. see  \cite[Section 16.10]{Fried2008}).
  In this case Lemma \ref{small-finiteness} says that a small profinite group has a maximal quotient in $\C$. This case appeared previously in \cite[Lemma 8.8]{SW}, while the case of a topologically finitely generated profinite group is more classical \cite[Corollary 15.72]{Neumann1967}. 
\end{remark}

 \begin{lemma}\label{L:prolocalRmod}
 \begin{enumerate} 
 \item  Let $R$ be a Noetherian local commutative ring with finite residue field, let $\mathfrak m$ its maximal ideal, and $\hat{R}$ its completion. Then the category of small pro-finite $R$-modules 
(i.e.  pro-objects for the category of finite $R$-modules) 
 is equivalent to the the category of finitely generated $\hat{R}$-modules. The sets $U_{S, L } = \{ N \in \Prf \mid  N\otimes_R S \cong L \}$ for $S$ a finite quotient of $R$  and $K$ an $R$-module form a basis for the topology on the set of isomorphism classes.
 
 \item Let $R$ be a Noetherian commutative ring. The the category of small pro-finite $R$-modules is equivalent to the product over maximal ideals $\mathfrak m$ of $R$ with finite residue field of the category of small pro-finite $R_{\mathfrak m}$-modules
 (where $R_{\mathfrak m}$ is the completion of $R$ at $\mathfrak{m}$). The topology on its set of isomorphism classes is given by the product topology (of the topology described above). 
 \end{enumerate}
 \end{lemma}

\begin{proof} This is a standard argument so we will be somewhat brief.
We first consider the local case. Given a finitely generated $\hat{R}$-module $N$, we obtain a pro-finite $R$-module $\tilde{N}$ by taking the inverse limit of the finite modules $N/\mathfrak m^k$. This is small since $N$ is finitely generated.
Conversely, 
any pro-finite $R$-module $\lim N_i$ gives a $\hat{R}$-module $N$ by taking the inverse limit, and the smallness gives that $N$ is finitely generated. 

One can check these constructions give an equivalence of categories.  
By Lemma \ref{mod-diamond}, the sets $U_{S,L}$ are exactly the usual basis for the topology on $\Prf$.

Now let $R$ be a general Noetherian commutative ring.
The natural map from the direct sum over maximal ideals $\mathfrak m$ with finite quotient of the category of finite $R_{\mathfrak m}$ modules to the category of finite $R$-modules is an equivalence, and this can be used to check the rest of the lemma.
\end{proof}

\section{Examples and applications}\label{s-examples}

We have already mentioned two fundamental examples of categories to which the main theorem can be applied: finite groups and finite $R$-modules for a ring $R$.  We will first investigate these cases in more detail, showing particular distributions that can be obtained, and proven to be uniquely determined by their moments, using the main theorem. We will then provide a toolkit for constructing more general examples, which has two parts. First, we give a general construction using the language of universal algebra that includes these two fundamental examples as well as many more. Second, we explain how new examples can be produced from old. We show that the product of any \Diamantine two categories is a \Diamantine category, that the objects in any \Diamantine category with an action of a fixed finite group $\Gamma$ form a \Diamantine category, and similar statements allowing us to put on extra structure.  Finally, we end with a simple examples, such as the opposite of the category of finite sets, and the opposite of the category of finite graphs, to show that there are interesting results even in these categories.

\subsection{Modules}\label{SS:Modules}

In this subsection, we apply our main results to the category of finite $R$-modules for a ring $R$.   
We will show how the constructed measures and the condition for well-behaved can be computed more explicitly, and give some examples of important measures and moments.  First we see that finite $R$-modules form a diamond category and give another interpretations of levels in this case.

\begin{lemma}\label{mod-diamond} Let $R$ be a ring (not necessarily commutative).  Then 

\begin{enumerate}

\item The category of finite $R$-modules is a \Diamantine category.

\item Each level of the category of finite $R$-modules is the set of isomorphism classes of $S$-modules for some finite quotient ring $S$ of $R$, and conversely each set of isomorphism classes of $S$-modules for some finite quotient ring $S$ of $R$ is a level of the category of finite $R$-modules.

\end{enumerate}

\end{lemma}

\begin{proof} We prove the first two assumptions of the definition of a \Diamantine category first. 
The poset of quotients of a module is dual to the poset of submodules, which is easily checked to be modular. 
Finiteness of automorphism groups follows from the finiteness of the modules.

Let $T$ be a finite set of $R$-modules. The action of $R$ on the objects in $T$ factors through a minimal finite quotient ring $S$. Let's prove that the level $\C$ generated by $T$ is equal to the category of $S$-modules.  Since fiber products exist in the category of $R$-modules, joins are given by fiber products over the meet. The property of being an $S$-module is preserved by quotients and fiber products, thus join-closed and downward-closed, so every module in $\C$ is an $S$-module. 

To show the converse, since every finite $S$-module is a quotient of $S^n$, and $S^n$ is a fiber product of copies of $S$, it suffices to show that $S \in \C$. Now $S$ is a submodule of the product over, for each $N_i$ in $T$, the module $N_i^{|N_i|}$ via the map that records where $s\in S$ sends each element of $N_i$.  So it suffices to show that $\C$ is closed under taking submodules.  If $K$ is a submodule of $N$ then $K$ is the quotient of $N \times_{N/K} N$ by the diagonal $N$, so $\C$ is closed under taking submodules as well.

 To finish the proof of (1), we observe that, since $S$ is finite, every $S$-module has finitely many extensions by a simple $S$-module, meaning there are finitely many simple morphisms from an object in $\C$.  Clearly,  there are countably many isomorphism classes of minimal objects. 
 
 To finish the proof of (2), we have already shown that every level is the category of $S$-modules for some $S$ and we must show the converse, but this is clear as we can take the level generated by $S$. 
  \end{proof}

 For $S$ a finite quotient of $R$, we write $\C_S$ for the level of finite $S$-modules.   We now give the sums defining $v_{\C_S, N}$, and the condition for well-behaved, more explicitly.
 
 \begin{proposition}\label{module-ext-formula} 
Let $R$ be a ring.  
 Let $S$ be a finite quotient ring of $R$, and let $K_1,\dots, K_n$ be representatives of the isomorphism classes of finite simple $S$-modules. Let $q_1,\dots, q_n$ be the cardinalities of the endomorphism fields of $K_i$.
 
 For $N$ an $S$-module and $\alpha \in \operatorname{Ext}^1(N, \prod_{i=1}^n K_i^{e_i})$, let $N_\alpha$ be the associated extension of $N$ by $\prod_{i=1}^n K_i^{e_i}$.

 Let $(M_N)_N\in\R^{\C_S}$.  Then for $N \in \C_S$ we have
  
 \[ v_{\C_S, N}  = \frac{1}{ \abs{ \Aut(N)}} \sum_{e_1,\dots, e_n=0}^{\infty}  \prod_{i=1}^n \left( \frac{ (-1)^{e_i} }{  \abs{\Hom (N, K_i)}^{e_i} \prod_{j=1}^{e_i} (q_i^j-1) }  \right)  \sum_{ \alpha \in \operatorname{Ext}^1(N, \prod_{i=1}^n K_i^{e_i})}  M_{N_\alpha} .\]
 
 We have that $(M_N)_N$ is well-behaved 
 at level $\C_S$  if and only if the sum above is absolutely convergent for every finite $S$-module $N$.
 
   \end{proposition}
 
 \begin{proof}   By definition we have
  \[ v_{\C_S, N}  =  \sum_{L \in \C_S }  \sum_{ \pi \in \Sur ( L, N)}  \frac{ \mu(\pi)  M_L }{ \abs{\Aut(N) } \abs{\Aut(L)}}.\]

The $\mu(\pi)$ term vanishes unless $\pi$ is semisimple, in which case $\ker \pi$ is a product of finite simple $S$-modules, so we can write $\ker\pi$ as $\prod_{i=1}^n K_i^{e_i}$ for a unique tuple $e_1,\dots, e_n$ of natural numbers.  In this case, $\mu(\pi) = \prod_{i=1}^n (-1)^{e_i} q_i^{ \binom{e_i}{2}} $ by Lemma \ref{module-semisimple-description}. Thus

  \[ v_{\C_S, N}  =\frac{1}{ \abs{ \Aut(N)}} \sum_{e_1,\dots, e_n=0}^{\infty} \prod_{i=1}^n \left( (-1) ^{e_i} q_i^{ \binom{e_i}{2}}   \right)  \sum_{L \in \C_s} \sum_{ \substack{ \pi \in \Sur(L,N) \\ \ker \pi \cong \prod_{i=1}^n K_i^{e_i} }}  \frac{M_L}{ \abs{\Aut(L)}} .\]

Now $ \ker \pi \cong \prod_{i=1}^n K_i^{e_i}$ if and only if there exists an injection $f \colon \prod_{i=1}^n K_i^{e_i} \to N$ such that $\operatorname{im} f = \ker \pi$, and the number of such $f$'s is 
\[\abs{ \Aut( \prod_{i=1}^n K_i^{e_i})} = \prod_{i=1}^n \abs{ \Aut(  K_i^{e_i}) }= \prod_{i=1}^n \abs{ GL_{e_i}(\mathbb F_{q_i})} = \prod_{i=1}^n \prod_{j=1}^{e_i}  (q^{e_i} - q^{ e_i-j}) =  \prod_{i=1}^n  \left( q^{ \binom{e_i}{2}} \prod_{j=1}^{e_i}  (q^{j}- 1) \right) .\] 

Thus
  \[ v_{\C_S, N}  =\frac{1}{ \abs{ \Aut(N)}} \sum_{e_1,\dots, e_n=0}^{\infty} \prod_{i=1}^n \left( \frac{ (-1) ^{e_i} }{  \prod_{j=1}^{e_i} (q_i^j-1) } \right)  \sum_{L \in \C_s} \sum_{ \substack{ \pi \in \Sur(L,N) \\ f\colon \prod_{i=1}^{n } K_i^{e_i} \to L \\  \operatorname{Im} f = \ker \pi \\ f \textrm{ injective} }}  \frac{M_L}{ \abs{\Aut(L)}} .\]
  
  Associated to each triple $L, f, \pi$ there is a class $\alpha \in \operatorname{Ext}^1(N, \prod_{i=1}^n K_i^{e_i})$, the class of the exact sequence $0 \to \prod_{i=1}^n K_i^{e_i} \to_f  K \to_{\pi} N \to 0$.  For this $\alpha$ we have $N_\alpha \cong L$ so $M_{N_{\alpha} } = M_{L}$. 
  
  Next we check that each class $\alpha$ arises from exactly $ \Aut(N_\alpha) / \prod_{i=1}^n \abs{\Hom (N, K_i)}^{e_i} $ triples $L, f,\pi$. 
  Two triples give the same $\alpha$, if and only if have the same $L$ and are related by an automorphism of $L$.
  Thus the number of different triples giving $\alpha$ is the number of orbits of $\Aut(L)$ on maps $f,\pi$.
   The stabilizer is the group of 
   automorphisms fixing the submodule $\prod_{i=1}^n K_i^{e_i}$ and its quotient module. 
   The group of such automorphisms is $\Hom(N, \prod_{i=1}^n K_i^{e_i})$ with cardinality
  \[ \abs{ \Hom(N, \prod_{i=1}^n K_i^{e_i})} =  \prod_{i=1}^n \abs{\Hom (N, K_i)}^{e_i} \]
so the number of choices of $(f,\pi)$ is $\Aut(N_\alpha) / \prod_{i=1}^n \abs{\Hom (N, K_i)}^{e_i}$, as desired.

Thus, we can replace the sum over $L, f,\pi$ with a sum over $\alpha$ after multiplying by \[\Aut(L) / \prod_{i=1}^n \abs{\Hom (N, K_i)}^{e_i}  ,\] which cancels the $\Aut(L)$ term and gives the desired statement.

The $M_N$ are well-behaved  if and only if the sum defining $v_{\C_S,N}$ with an additional $Z(\pi)^3$ factor is absolutely convergent for all $N$. 
We have $2\leq Z(\pi)\leq 2^n$,  so that absolute convergence happens if and only if the sum defining $v_{\C_S,N}$ (in which $n$ is constant) is absolutely convergent for all $N$. 
\end{proof}

Now we give a commonly occurring example of moments (e.g.  the Cohen-Lenstra-Martinet distributions have such moments \cite[Theorem 6.2]{Wang2021}), for which we find the measure explicitly.

\begin{lemma}\label{mod-card-wb} Let $R$ be a ring,  $\category$ be the category of finite $R$-modules and $u$ be a real number.   
For each $N\in \category$, let $M_N=\abs{N}^{-u}$. 
Then $(M_N )_N$ is well-behaved.

Let $S$ be a finite quotient ring of $R$, and let $K_1,\dots, K_n$ be representatives of the isomorphism classes of finite simple $S$-modules. Let $q_1,\dots, q_n$ be the cardinalities of the endomorphism fields of $K_i$. For $N$ a finite $S$-module, we have
\[ v_{\C_S,N}  = \frac{ 1}{ \abs{\Aut(N) } \abs{N}^{u}}  \prod_{i=1}^n \prod_{j=1}^{\infty} \left( 1 -    \frac{\abs{ \operatorname{Ext}^1_{ S} ( N,  K_i ) }}{ \abs{ \operatorname{Hom}(N, K_i)} \abs{K_i}^u  }\abs{q_i }^{-j} \right) \]

If $u$ is an integer then $v_{\C_S,N} \geq 0$ for all finite $S$-modules $N$. 
\end{lemma}

\begin{remark}
If $R$ is a DVR with residue field $k$, then for a finite $S$-module $N$, we have $\abs{ \operatorname{Ext}^1_{ S} ( N,  k ) }\leq \abs{ \operatorname{Hom}(N, k)}$,
and then for all real $u\geq -1$,  we have $v_{\C_S,N} \geq 0$ for all finite $S$-modules $N$. 
\end{remark}

Thus, by Theorem~\ref{analytic-main-measure}, for $u\in \Z$ (or $u\geq -1$ when $R$ is a DVR),  there is a unique measure on the set of isomorphism classes of pro-finite $R$-modules with moments $|N|^{-u}$ and it is given by the formulas in Lemma~\ref{mod-card-wb}.  We write $\mu_u$ for this measure (when $R$ is understood).

\begin{proof} 
From Proposition \ref{module-ext-formula} we have
 \[ v_{\C_S, N}  = \frac{1}{ \abs{ \Aut(N)}} \sum_{e_1,\dots, e_n=0}^{\infty}  \prod_{i=1}^n \left( \frac{ (-1)^{e_i} }{  \abs{\Hom (N, K_i)}^{e_i} \prod_{j=1}^{e_i} (q_i^j-1) }  \right)  \sum_{ \alpha \in \operatorname{Ext}^1(N, \prod_{i=1}^n K_i^{e_i})}  M_{N_\alpha} .\]
and \[M_{N_\alpha} = \abs{N_\alpha}^{-u} = \abs{N}^{-u} \prod_{i=1}^n \abs{K_i}^{- e_i u }\] so 
 \[ v_{\C_S, N}  = \frac{1}{ \abs{ \Aut(N)} \abs{N}^{u} } \sum_{e_1,\dots, e_n=0}^{\infty}  \prod_{i=1}^n \left( \frac{ (-1)^{e_i} }{  \abs{\Hom (N, K_i)}^{e_i} \abs{K_i}^{e_i u } \prod_{j=1}^{e_i} (q_i^j-1) }  \right)  \abs{ \operatorname{Ext}^1(N, \prod_{i=1}^n K_i^{e_i})}\]
 \[ = \frac{1}{ \abs{ \Aut(N)} \abs{N}^{u} } \sum_{e_1,\dots, e_n=0}^{\infty}  \prod_{i=1}^n \left( \frac{ (-1)^{e_i} }{  \abs{\Hom (N, K_i)}^{e_i} \abs{K_i}^{e_i u } \prod_{j=1}^{e_i} (q_i^j-1) }  \right)  \prod_{i=1}^n \abs{ \operatorname{Ext}^1(N,  K_i)}^{e_i} \]
\[ =  \frac{1}{ \abs{ \Aut(N)} \abs{N}^{u} } \sum_{e_1,\dots, e_n=0}^{\infty}  \prod_{i=1}^n \left( \frac{ (-1)^{e_i} \abs{ \operatorname{Ext}^1(M,  K_i)}^{e_i} }{  \abs{\Hom (N, K_i)}^{e_i} \abs{K_i}^{e_i u } \prod_{j=1}^{e_i} (q_i^j-1) }  \right) \]
\[ =  \frac{1}{ \abs{ \Aut(N)} \abs{N}^{u} }   \prod_{i=1}^n \sum_{e=0}^{\infty} \left( \frac{ (-1)^{e} \abs{ \operatorname{Ext}^1(N,  K_i)}^{e} }{  \abs{\Hom (N, K_i)}^{e} \abs{K_i}^{e u } \prod_{j=1}^{e} (q_i^j-1) }  \right) \] \[= \frac{ 1}{ \abs{\Aut(N) } \abs{N}^{u}}  \prod_{i=1}^n \prod_{j=1}^{\infty} \left( 1 -    \frac{\abs{ \operatorname{Ext}^1_{ S} ( N,  K_i ) }}{ \abs{ \Hom(N, K_i)} \abs{K_i}^u  }\abs{q_i }^{-j} \right) \]
by the $q$-series identity $\sum_{e=0}^{\infty} \frac{  v^e}{ \prod_{j=1}^{e} (q^j-1)} = \prod_{j=1}^{\infty} ( 1 + v q^{-j} ) $ (from the $q$-binomial theorem).

This $q$-series sum is absolutely convergent because the denominator grows superexponentially in $e$ while the numerator grows only exponentially. All our manipulations preserve absolute convergence in the reverse direction because we never combine two terms in the sum with opposite signs -- in fact, we never combine two terms except in the very first step, where we group together identical terms.  Thus  Proposition \ref{module-ext-formula} gives well-behavedness.

If $u$ is an integer, then $ \frac{\abs{ \operatorname{Ext}^1_{ S} ( N,  K_i ) }}{ \abs{ \operatorname{Hom}(N, K_i)} \abs{K_i}^u  }$ is an integer power of $q_i$ since $\operatorname{Ext}^1_{ S} ( N,  K_i )$,  $\operatorname{Hom}(N, K_i)$, and $K_i$ are all vector spaces over $\mathbb F_{q_i}$ and hence have cardinality an integer power of $q_i$. If  $ \frac{\abs{ \operatorname{Ext}^1_{ S} ( N,  K_i ) }}{ \abs{ \operatorname{Hom}(N, K_i)} \abs{K_i}^u  }$ is a positive integer power of $q_i$ then one of the terms in the product over $j$ vanishes, so the product is $0$ and thus is nonnegative, and if it is a nonnegative integer power of $q_i$ then all the terms in the product over $j$ are positive, so the product is nonnegative. Thus $ v_{\C_S, N}  \geq 0$.
\end{proof}

\begin{corollary}\label{C:wbbound} Let $R$ be a ring and let $\category$ be the category of finite $R$-modules.  Let $(M_N)_N\in \R^{\category/\isom}$.  If $M_N = O ( \abs{N}^n)$ for some real $n$, then $(M_N)_N$ is well-behaved. \end{corollary}

\begin{proof} This follows immediately form Lemma \ref{mod-card-wb}, taking $u=-n$, and Lemma \ref{big-O-wb}. \end{proof}

Many special cases of both the uniqueness and robustness that follows from Corollary~\ref{C:wbbound}, and the formulas for particular measures in Lemma~\ref{mod-card-wb} have appeared in previous work.  For example, Fouvry and Kl\"{u}ners prove uniqueness for the case $\R=\Z/p\Z$ \cite[Proposition 3]{Fouvry2006} and give formulas for the probabilities when the moments are all 1 \cite[Proposition 2]{Fouvry2006}.  Ellenberg, Venkatesh, and Westerland prove uniqueness and give formulas for $R=\Z_p$ when all moments are 1, as well as prove robustness \cite[Section 8.1]{Ellenberg2016}.
When $R=\Z$, the  second author proved uniqueness and robustness given by Corollary~\ref{C:wbbound}, as well as giving the formulas for moments $|N|^{-u}$
(see \cite[Theorem 8.3]{Wood2017} 
and \cite[Section 3]{Wood2019a}).  When $R$ is a maximal order in a semi-simple $\Q$-algebra, Wang and the second author did the same \cite[Section 6]{Wang2021}. 
For a $R$ with finitely many (isomorphism classes of) simple $R$-modules and such that $\operatorname{Ext}^1_R$ between any two finite $R$-modules is finite,  the first author \cite[Theorem 1.3]{Sawin2020} proved the uniqueness and robustness given by Corollary~\ref{C:wbbound}.

Under mild hypotheses, the measures from Lemma~\ref{mod-card-wb} have a nice description in terms of a limit of random modules from generators and relations.

\begin{lemma}\label{L:genandrel}
Let $R$ be a ring such that there
 are only countable many isomorphism classes of finite $R$-modules (such as a Noetherian local commutative ring with finite residue field).  Let $\hat{R}$ be the profinite completion of $R$.

As $n\to \infty$, the distribution of the quotient of $\hat{R}^n$ by $n+u$ random relations taken from Haar measure on $\hat{R}^n$ weakly converges to $\mu_u$
(defined above as the measure on pro-finite $R$-modules with moments $|N|^{-u}$).
\end{lemma}

\begin{proof} 
A finitely generated $\hat{R}$-module $N$ naturally gives a small pro-finite $R$-module.  Each continuous finite quotient ring 
$\hat{R}/I$ is a finite quotient of $R$,  and the functor from the category of continuous finite quotients of $\hat{R}$ taking $\hat{R}/I$ to the finite $R$-module $\hat{R}/I \tensor N$ gives the pro-object.  
Thus  we can apply Theorem~ \ref{T:weak-convergence} to show that the distribution of the quotient of $\hat{R}^n$ by $n+u$ random relations taken from Haar measure converges to $\mu_u$ as long as its $N$-moment converges to $\abs{N}^{-u}$ for all $N$.

This is true since, for each surjection $\pi$ from $\hat{R}^n$ to $N$, the probability that $\pi$ descends to the quotient of $\hat{R}^n$ by $n+u$ Haar random relations is $\abs{N}^{-n-u}$, since this is equivalent to the probability that the images of the $n+u$ random relations in $N$ are all trivial. So the expected number of surjections from the quotient of $\hat{R}^n$ by $n+u$ Haar random relations is $\Surj(\hat{R}^n, N) \abs{N}^{-n-u}$. As $n$ goes to $\infty$, the ratio $\Surj(\hat{R}^n, N)  / \abs{N}^n$ converges to $1$ since it is equal to the probability that $n$ random elements of $N$ generate $N$, which converges to $1$. Hence the expected number of surjections from the quotient of $N$ converges to $\abs{N}^{-u}$, as desired.
\end{proof}

For the rest of Section~\ref{SS:Modules}, we restrict attention to  commutative rings.

When $R$ is a local ring, the formulas for $\mu_u$ simplify even further, and we can make more computations about the measure.

\begin{lemma}\label{L:Localmeasform}
 Let $R$ be a Noetherian local commutative ring with finite residue field $k$. 
We have\[ \mu_u( U_{S,L} ) = \abs{\Aut(L)}^{-1}\abs{L}^{-u} \prod_{j=1}^{\infty} \left( 1 -    \frac{\abs{ \operatorname{Ext}^1_{S} ( L,  k) }}{ \abs{ \operatorname{Hom}(L, k)} }\abs{k}^{-j-u} \right) \] for any nonzero finite quotient $S$ of $R$.
\end{lemma}

\begin{proof} 
Since $R$ is a local ring, the only simple $R$-module is the residue field $k$, and for $S\neq 0$, the residue field is an $S$-module, so we can take $n=1$ and $N_1=k$ in Lemma \ref{mod-card-wb}.  In this case, the formula of Lemma \ref{mod-card-wb}  simplifies to the stated one.
\end{proof}

One interesting set whose measure we can calculate in the local rings case is the set of finite $R$-modules. This is a measurable set since it is a countable union of closed points.

\begin{lemma} Let $R$ be an infinite Noetherian local commutative ring with finite residue field $k$.  We have
\[ \mu_u( \{ N \mid N \textrm{ finite} \} ) = \begin{cases} 0 & \textrm{if } u<0 \\ \mu_u(\{0\})=  \prod_{j=1}^\infty (1 - \abs{k}^{-j-u} )  & \textrm{if } 0 \leq u < \dim R-1 \\ 1 & \textrm{if } u \geq \dim R-1 \end{cases} \]
where $\dim R$ is the Krull dimension. In the middle case, saying that the measure is $\mu_u( \{0\})$ is equivalent to saying each nonzero finite $R$-module has  measure $0$. \end{lemma}

\begin{proof} Let $N$ be a finite $R$-module. Then $N$ has a minimal presentation in terms of $d$ generators and $m$ relations for  $d = \dim \Hom(N, k)$ and $m = \dim \operatorname{Ext}^1_R ( N,k) $. We can choose a finite quotient $S$ of $R$ such that $N$ factors through $S$, as do all finite extensions of $N$ by $k$, so that $\operatorname{Ext}^1_S( N,k)  = \operatorname{Ext}^1_R ( N,k) $. It follows that for an $R$ module $L$, we have $L \otimes_R S \cong N$ if and only if $L \cong N$, with ``if" obvious and ``only if" because otherwise some quotient of $L$ would be an extension of $N$ by $k$ and thus an $S$-module, a contradiction.

Then \[ \mu_u ( \{ N \}) = \mu_u ( U_{S,N}) = \abs{\Aut(N)}^{-1} \abs{N}^{-u} \prod_{j=1}^{\infty} (1 - \abs{k}^{m-d-j-u } ) \]
which is $0$ if $m-d-j-u =0$ for some $j >0$, equivalently, if $m-d - u >0$.  So we see the probability of attaining $N$ is nonzero only if $m-d\leq u$.

We will now show that if $N$ is a nonzero finite $R$-module, then $m-d\geq \dim R-1$.  
The $m$ relations of $N$ form an $d \times m$ matrix with entries in $R$, and let $I$ be the ideal generated by the $d\times d$ minors of that matrix. 
If $I$ is the unit ideal, then $N$ is trivial.
If $d-n\leq \dim R-2$ and $I$ is not the unit ideal,  then $I$ is contained in a prime of height $\dim R-1$ \cite[Theorem 2.1]{Bruns1988}, so a non-maximal prime $\wp$.
If $S$ is the fraction field of $R/\wp$, then $N\tensor_R S\isom S^k$ for some $k>1$ since $I\sub \wp$ implies that that the relations defining $N\tensor_R S$ are rank at most $d-1$ over $S$.  Thus, in this case $N$ must be infinite, as if $N$ were finite, then it would be annihilated by some power of an element in the maximal ideal but not $\wp$, which would contradict $N\tensor_R S\isom S^k$

Thus, if $u< \dim R-1$, then we cannot have $m-d \leq u$ for $N$ a nonzero $R$-module, and so the probability of each nonzero finite $R$-module is zero.

For the zero module, we have $d=m=0$ so the probability of the zero module is also zero for $u<0$. If $u\geq 0$, the probability of the zero $R$-module is given by $1^{-u} \prod_{j=1}^{\infty} (1 - \abs{k}^{0-0-j-u } ) = \prod_{j=1}^\infty (1 - \abs{k}^{-j-u} )  $,  finishing the proof of the middle case.

We finally consider the case $u \geq \dim R-1$.  In this case, we first check that the quotient of $\hat{R}^n$ by $n+u$ random relations is finite with probability $1$. 
Let $A_m$ be the random matrix over $\hat{R}$ who columns are the first $m$ relations.  Let $I_t(m)$ be the ideal generated by $t\times t$ minors of $A_m$.
We say a prime $\wp$ in $\hat{R}$ of height $h$ is \emph{bad} if $\wp$ is not maximal, and if for some $h+1\leq k\leq n+h$ we have that the rank of $A_k$ mod $\wp$ is $<k-h,$ or equivalently,  $I_{k-h}(k) \sub \wp$.  We will show by induction on $h$ and $k$ that the probability that there exists a bad prime is $0$.

We condition on the first $k-1$ columns of the matrix.  If $I_{k-h}(k-1)\not\sub \wp,$ then  $I_{k-h}(k)\not\sub \wp$.  If $h>0$, then we know by induction that
$I_{(k-1)-(h-1)}(k-1)$ is not contained in any prime of height $h-1$.  Thus any prime of height $h$ containing $I_{k-h}(k-1)$ is a prime of height $0$ in $R/I_{k-h}(k-1)$.  There are only finitely many such primes (given $A_{k-1}$).  For each of these finitely many primes $\wp$ of height $h$, the matrix
$A_{k-1}$ mod $\wp$ has rank at least $k-1-h$ with probability $1$ by the inductive hypothesis, and we choose some $(k-1-h)\times(k-1-h)$ minor of $A_{k-1}$ that is non-zero mod $\wp$.  We condition on all entries of $A_k$ except for one entry in the $k$th column and a row not in the chosen minor.  Then there is at most one value that entry can be mod $\wp$ so that $A_k$ mod $\wp$ is rank $<k-h$.   Since $\hat{R}/\wp$ is infinite (since it is dimension at least $1$), 
  a Haar random element of  $\hat{R}$ takes that value mod $\wp$ with probability $0$.  Since there are only finitely many primes of height $h$ containing $I_{k-h}(k-1)$ (and the rest cannot contain $I_{k-h}(k)$), this completes the induction.

Let $N$ be the quotient of $\hat{R}^n$ by $n+\dim R -1$ random relations.  
With probability $1$, we have that in the quotient by every non-maximal prime $\wp$ of $\hat{R}$, the matrix $A(n+\dim R -1)$ is rank $n$, which implies that the localization $N_\wp=0, $ which implies that the annihilator $I_N$ of $N$ is not contained in $\wp$.  So, with probability 1, we have that 
$\hat{R}/I_N$ is $0$-dimensional and hence finite since the residue is finite, and the finitely generated $\hat{R}/I_N$-module $N$ is thus finite. 

To prove that $\mu_u$ of the set of finite $R$-modules is $1$, it remains to prove that 
\[ \lim_{n\to\infty}  \sum_{N  \textrm{ finite } R\textrm{-module}} \operatorname{Prob} ( \hat{R}^n / \langle n+u \textrm{ Haar random relations} \rangle \cong N )\]\[=  \sum_{N  \textrm{ finite } R\textrm{-module}}  \lim_{n\to\infty}\operatorname{Prob} ( \hat{R}^n / \langle n+u \textrm{ Haar random relations} \rangle \cong N)\] 
as we have just seen the left side is $1$. 

We do this using the dominated convergence theorem. The probability that $\hat{R}^n$ modulo $n+u$ Haar random relations is isomorphic to $N$ is bounded above by $ \abs{\Aut(N)}^{-1}\abs{N}^{-u}$ (see the proof of Lemma~\ref{L:genandrel}), and is $0$ if the minimal number of relations $m$ to define $N$, minus the minimal number of generators $d$, is greater than $u$.

On the other hand, if then $m-d\leq u$, then \[\mu_u(\{N\}) = \abs{\Aut(N)}^{-1} \abs{N}^{-u} \prod_{j=1}^{\infty} (1 - \abs{k}^{m-d-j-u } )  \geq  \abs{\Aut(N)}^{-1} \abs{N}^{-u} \prod_{j=1}^{\infty} (1 - \abs{k}^{-j} ) \] so in either case we have
\[ \operatorname{Prob} ( \hat{R}^n / \langle n+u \textrm{ Haar random relations} \rangle \cong N)  \leq  \mu_u(\{N\})  \prod_{j=1}^{\infty} (1 - \abs{k}^{-j} )^{-1} \]
and this upper bound satisfies the assumption of the dominated convergence theorem since
\[ \sum_{N  \textrm{ finite } R\textrm{-module}}    \mu_u(\{N\})  \prod_{j=1}^{\infty} (1 - \abs{k}^{-j} )^{-1} \leq \mu_u ( \mathcal P ) \prod_{j=1}^{\infty} (1 - \abs{k}^{-j} )^{-1}  =  \prod_{j=1}^{\infty} (1 - \abs{k}^{-j} )^{-1}  < \infty.\]

\end{proof}

For a discrete valuation ring, we can give a simple condition for a collection of moments to be well-behaved that is almost sharp. For an $R$-module $N$, we write $\wedge^2 N$ for the quotient of the $R$-module  $N\tensor N$ by the $R$-module generated by elements of the norm $n\tensor n$ for $n\in N$.

\begin{lemma}\label{dvr-sharp-moments} Let $R$ be a discrete valuation ring with finite residue field.  Let $M_N$ be an assignment of a real number to each isomorphism class of $R$-modules. Then $M_N$ is well-behaved if there exists $\epsilon>0$ and $c$ such that  $ M_N \leq  c  \abs{N}^{1-\epsilon} \abs{ \wedge^2(N) } $ for all finite $R$-modules $N$. \end{lemma}

We then obtain corollaries for uniqueness of measures from Theorems~\ref{T:general} and \ref{analytic-main-measure}.
A slightly weaker uniqueness result appears in \cite[Theorems 8.2 and 8.3]{Wood2017}, without the $|N|^{1-\epsilon}$ in the upper bound
(though stated there over $\Z_p$, the argument works over other DVR's with finite residue field, e.g. see \cite[Theorem 6.11]{Wang2021}).

Lemma~\ref{dvr-sharp-moments} is almost sharp because, if we set $M_N = \abs{N} \abs{ \wedge^2(N) } =\abs{\Sym^2(N)}$ then the moments are attained for two distinct measures $\mu$: The probability distribution the quotient of $\hat{R}^n$ by the columns of a Haar random $n \times n$ skew-symmetric matrix converges as $n \to \infty$ for $n$ even to one measure, and as $n\to\infty$ for $n$ odd to a different measure, which both have the same moments $\abs{\Sym^2(N)}$.
(See \cite[Theorems 10 and 11]{Clancy2015} for an argument showing moments in a different, but analogous case, and see \cite[Theorems 3.9 and 3.11]{Bhargava2015b} 
for the limiting distributions in the case $\hat{R}=\Z_p$.)
So if we allow $\epsilon=0$ then  $M_N \leq  c  \abs{N} \abs{ \wedge^2(N) } $ are not necessarily well-behaved since the conclusion of Theorem \ref{analytic-main-measure} is not necessarily true. 

\begin{proof} Let $k$ be the residue field of $R$ and $q$ its cardinality. Let $\pi$ be the uniformizer of $R$.

From Proposition \ref{module-ext-formula}, since $k$ is the only simple $R$-module,  it suffices to prove for each finite $R$-module $N$ that    \[  \frac{1}{ \abs{ \Aut(N)}} \sum_{e=0}^{\infty}   \left( \frac{ 1 }{  \abs{\Hom (N, k) }^{e} \prod_{j=1}^{e} (q^j-1) }  \right)  \sum_{ \alpha \in \operatorname{Ext}^1(N, k^e )}  c \abs{N_\alpha }^{1-\epsilon} \abs{ \wedge^2(N_\alpha) }   <\infty.\]

By the classification of modules over PIDs, we can write $N \cong \prod_{\pi=1}^n  R/\pi^{d_i}$ for some tuple $d_1,\dots, d_n$ of positive integers. Then using the exact sequence $0 \to R^n \to R^n \to N \to 0$, we can write  $ \operatorname{Ext}^1(N, k^e ) \cong k^{ne}$. Explicitly, we can represent an extension class as an $e \times n$ matrix where the $i$th column is $\pi^{d_i}$ times a lift of the generator of $R/\pi^{d_i}$ from $N$ to $N_\alpha$.

We now describe $\wedge^2 (N_\alpha)$. This has a filtration into submodules where  $F_1 \sub \wedge^2 (N_\alpha)$ is generated by all $a \wedge b$ where $a,b, \in k^e$, the second term $F_2 \sub \wedge^2(N_\alpha)$ is generated by all $a \wedge b$ where $a \in k^e$, and the last term $F_3= \wedge^2(N_\alpha)$.  Then
\[F_3  / F_2   \cong \wedge^2 N ,\]
while $F_2  / F_1 $ is a quotient of $N\otimes k^{e} \cong k^{ne}$, and  $F_1 $ is a quotient of $\wedge^2 (k^e)$. Furthermore, if $a \in k^e$ lies in the image of the matrix representing $\alpha$ then $a \in \pi N_\alpha$ so for $b\in k^e$, we have  \[a \wedge b \in \pi N_\alpha \wedge b = N_\alpha \wedge \pi b = N_\alpha \wedge 0 =0 ,\] and thus $F_1$ is a quotient of $\wedge^2 ( k^e/ \operatorname{Im}(\alpha))$. Thus
\[ \abs{ \wedge^2(N_\alpha) } = \abs{ F_3 / F_2   } \abs{ F_2 / F_1  } \abs{ F_1 }  \leq \abs{ \wedge^2(N)}  q^{en}  \abs{ \wedge^2 ( k^e/ \operatorname{Im}(\alpha))}  = \abs{ \wedge^2(N)}  q^{en}  q^{ \frac{ (e- \operatorname{rank}(\alpha) ) ( e- \operatorname{rank}(\alpha)-1)}{2}} .\]

We also note that $\abs{N_\alpha} =  \abs{N} q ^{e}$ and $\abs{ \Hom(N,k)} = q^n $. Thus 

\begin{align*}
  &\frac{1}{ \abs{ \Aut(N)}} \sum_{e=0}^{\infty}   \left( \frac{ 1 }{  \abs{\Hom (N, k) }^{e} \prod_{j=1}^{e} (q^j-1) }  \right)  \sum_{ \alpha \in \operatorname{Ext}^1(N, k^e )}  c \abs{N_\alpha }^{1-\epsilon} \abs{ \wedge^2(N_\alpha) } \\
   \leq &  \frac{1}{ \abs{ \Aut(N)}} \sum_{e=0}^{\infty}   \left( \frac{ 1 }{ q^{ne}  \prod_{j=1}^{e} (q^j-1) }  \right)  \sum_{ \alpha \in \operatorname{Ext}^1(N, k^e )}  c  ( \abs{N} q^e )^{1-\epsilon}  \abs{ \wedge^2(N)}  q^{en}  q^{ \frac{ (e- \operatorname{rank}(\alpha) ) ( e- \operatorname{rank}(\alpha)-1)}{2}} \\ 
   = &  \frac{c\abs{ \wedge^2(N)}    \abs{N}^{1-\epsilon}  }{ \abs{ \Aut(N)}} \sum_{e=0}^{\infty}   \left( \frac{ 1 }{   \prod_{j=1}^{e} (q^j-1) }  \right)  \sum_{ \alpha \in \operatorname{Ext}^1(N, k^e )}    (  q^e )^{1-\epsilon}  q^{ \frac{ (e- \operatorname{rank}(\alpha) ) ( e- \operatorname{rank}(\alpha)-1)}{2}} \\
   = &   \frac{c\abs{ \wedge^2(N)}    \abs{N}^{1-\epsilon}  }{ \abs{ \Aut(N)}} \sum_{e=0}^{\infty}   \left( \frac{ 1 }{   \prod_{j=1}^{e} (q^j-1) }  \right) \sum_{r=0}^n  (  q^e )^{1-\epsilon}  q^{ \frac{ (e- r ) ( e- r-1)}{2}} \# \{ \alpha \in \Hom ( k^n, k^e) \mid \operatorname{rank}(\alpha)=r \} \\
    \leq  &C_N \sum_{e=0}^{\infty}   \left( \frac{ 1 }{   \prod_{j=1}^{e} (q^j-1) }  \right) \sum_{r=0}^n (  q^e )^{1-\epsilon}   q^{ \frac{ (e- r ) ( e- r-1)}{2}} q^{ re }  
   = \sum_{e=0}^{\infty}   \left( \frac{ 1 }{   \prod_{j=1}^{e} (q^j-1) }  \right) \sum_{r=0}^n   q^{   \frac{e (e+1-2\epsilon) }{2} + \frac{r^2+r}{2}}\\
    \leq & C'_N   \sum_{e=0}^{\infty}   \left( \frac{ q^ {\frac{e (e+1-2\epsilon) }{2}}}{   \prod_{j=1}^{e} (q^j-1) }  \right)   \leq \frac{1}{ \prod_{j=1}^{\infty} ( 1-q^{-j} ) }  \sum_{e=0}^{\infty} q^{ - \epsilon e} <\infty , 
  \end{align*}
  
where $C_N,C_N'$ are constants only depending on $c,N$.
 \end{proof}

\subsection{Groups}
As in the example of modules, we show that the category of finite groups is a \Diamantine category, simplify the formulas for the measure for general moments, and simplify them further and verify positivity for certain nice moments.  The reader will notice the proofs are quite similar, and indeed one can make analogous computations in many \Diamantine categories.   
Instead of carrying out these computations in further examples, we hope that the examples given here will give the reader a blueprint for how to do the computations in further examples of interest.  One can see see \cite{SW} and our forthcoming paper \cite{SWROUCL} for more examples of analogous computations in other cases.

\begin{lemma}\label{L:groupsD}
 The category of finite groups is a \Diamantine category. \end{lemma}

\begin{proof} 

The poset of quotients of a group is dual to the poset of normal subgroups, and it is a classical fact (easy to check) that the normal subgroups form a modular lattice \cite[Ch.  I, Thm. 11 on p. 13]{Birkhoff1967} (a property that is manifestly stable under duality). The finiteness of automorphism groups is trivial, and the third assumption follows from \cite[Corollary 6.13]{Liu2020}.  Clearly,  there are countably many isomorphism classes of minimal objects. \end{proof}

\begin{proof}[Proof of Theorem~\ref{T:groups}]
Using Lemma~\ref{L:groupsD}, we have that
Theorems~\ref{analytic-main-measure} and \ref{T:weak-convergence} specialize to
Theorem~\ref{T:groups}, except that we need a translation from pro-isomorphism classes to profinite groups with finitely many continuous homomorphisms to any finite group, which is provided by Lemmas~\ref{pro-object-existence} and \ref{small-characterization}.
\end{proof}

We now introduce some notation that will help to classify simple surjections of finite groups. Let $G$ be a finite group.  A \emph{$[G]$-group} is a group $H$ together with a homomorphism from $G$ to $\Out(H)$.  A morphism of $[G]$-groups is a homomorphism $f: H\ra H'$ such that for each element $g\in G$, for each lift $\sigma_1$ of the image of $g$ from $\Out(H)$ to $\Aut(H)$, there is a lift $\sigma_2$ of the image of $g$ from $\Out(H')$ to $\Aut(H')$ such that $\sigma_2\circ f=f\circ \sigma_1$.  
A $[G]$-group isomorphism is then defined a $[G]$-group morphism with an inverse $[G]$-group morphism.
 We write $\Aut_{[G]}(H)$ for the group of automorphisms of $G$ as an $[H]$-group.
Note that $G$ acts on the set of normal subgroups of a $[G]$-group $H$, and we say a nontrivial $[G]$-group is \emph{simple} if it has no nontrivial proper fixed points for this action.  
If we have an exact sequence $1\ra N \ra H \ra G\ra 1$, then $N$ is naturally a $[G]$-group.
 A normal subgroup $N'$ of $N$ is fixed by the $\Out(G)$ action if and only if it is normal in $H$, and so $N$
 is a product of simple $H$-groups if and only if it is a product of simple $[G]$-groups.

We can now give a formula for measures in the category of groups that is more cumbersome to state than the general formula but will be easier to use in applications.

 \begin{proposition}\label{group-ext-formula} Let $\C$ be a level in the category of groups and $(M_G)_G\in \R^{\C}$.
 
 Fix a finite group $F$.  Let $V_1,\dots, V_n$ be representatives of all isomorphism classes of irreducible representations of $F$ over prime fields such that $V_i \rtimes F \in \C$. For each $i$, let $W_i$ be the subset of $H^2 ( F, V_i)$ consisting of classes whose associated extension of $F$ by $V_i$ is in $\C$ 
 and let $q_i$ be the order of the endomorphism field of $V_i$. 
 Let $N_1,\dots, N_m$ be representatives of all isomorphism classes of pairs of  finite simple $[F]$-groups $N$ such that $F \times_{\Out(N)} \Aut(N)$ is in $\C$.
  
 For nonnegative integers $e_1,\dots, e_n, f_1,\dots, f_m$ and $\alpha_i \in  W_i^{e_i} \subseteq H^2(F, V_i)^{e_i} = H^2 (F, V_i^{e_i})$, let $G_{{ \mathbf e}, {\mathbf f,} {\mathbf \alpha}}$ be the fiber product over $F$ of, for each $i$ from $1$ to $n$, the extension of $F$ by $V_i^{e_i}$ associated to $\alpha_i$, with, for each $i$ from $1$ to $m$, $F \times_{\Out(N_i)^{f_i}} \Aut(N_i)^{f_i}$.
 
 Then we have
 \[  \hspace{-.5in} v_{\C, F}  = \frac{1}{ \abs{ \Aut(F)}} \sum_{{\mathbf e} \in \mathbb N^n, {\mathbf f}\in \mathbb N^m}  \prod_{i=1}^n \left( \frac{ (-1)^{e_i} \abs{H^0(F,V_i)}^{e_i} }{ \abs{V_i}^{e_i}   \abs{H^1 (F, v_i) }^{e_i} \prod_{j=1}^{e_i} (q_i^j-1) }  \right) \prod_{i=1}^m \left( \frac{(-1)^{f_i}  } { \abs{ \Aut_{[F]}(N_i)}^{f_i} (f_i)! } \right) \sum_{ {\mathbf \alpha}  \in \prod_{i=1}^n W_i^{e_i}  } M_{ G_{{ \mathbf e}, {\mathbf f,} {\mathbf \alpha}}}  .\]
 
 The $M_G$ are well-behaved at level $\C$  if and only if the similar sum
  \[  \sum_{{\mathbf e} \in \mathbb N^n, {\mathbf f}\in \mathbb N^m}  \prod_{i=1}^n \left( \frac{ \abs{H^0(F,V_i)}^{e_i} }{ \abs{V_i}^{e_i}   \abs{H^1 (F, v_i) }^{e_i} \prod_{j=1}^{e_i} (q_i^j-1) }  \right) \prod_{i=1}^m \left( \frac{8^{f_i}  } { \abs{ \Aut_{[F]}(N_i)}^{f_i} (f_i)! } \right) \sum_{ {\mathbf \alpha}  \in \prod_{i=1}^n W_i^{e_i}  } M_{ G_{{ \mathbf e}, {\mathbf f,} {\mathbf \alpha}}}  \]
is absolutely convergent.
   \end{proposition}
 
 \begin{proof}   By definition we have
  \[ v_{\C, F}  =  \sum_{G \in \C }  \sum_{ \pi \in \Sur ( G,F )}  \frac{ \mu(\pi)  M_G }{ \abs{\Aut(G) } \abs{\Aut(F)}}.\]

The $\mu(\pi)$ term vanishes unless $\pi$ is semisimple, in which case $\ker \pi$ is a product of finite simple $[F]$-groups, 
of which the abelian ones are irreducible representations of $F$ over prime fields,  So, we can write $\ker\pi$ as $\prod_{i=1}^n V_i^{e_i} \times \prod_{i=1}^m N_i^{f_i} $ for  unique tuple $e_1,\dots, e_n, f_1,\dots f_m$ of nonnegative integers.  In this case, $\mu(\pi) = \prod_{i=1}^n (-1)^{e_i} q_i^{ \binom{e_i}{2}} \cdot \prod_{i=1}^m (-1)^{f_i} $ by Lemma \ref{group-semisimple-description}. 
Thus
  \[ v_{\C, F} =\frac{1}{ \abs{ \Aut(F)}} \sum_{{\mathbf e} \in \mathbb N^n, {\mathbf f}\in \mathbb N^m}  \prod_{i=1}^n \left( (-1) ^{e_i} q_i^{ \binom{e_i}{2}}   \right) \prod_{i=1}^m (-1)^{f_i}   \sum_{G \in \C} \sum_{ \substack{ \pi \in \Sur(G,F) \\ \ker \pi \cong \prod_{i=1}^n V_i^{e_i} \times \prod_{i=1}^m N_i^{f_i}  }}  \frac{M_G}{ \abs{\Aut(G)}} ,\]
where $\cong$ here denotes an $[H]$-group isomorphism.
Now $ \ker \pi \cong \prod_{i=1}^n V_i^{e_i} \times \prod_{i=1}^m N_i^{f_i} $ if and only if there exists an $[F]$-group
isomorphism  $\iota \colon \prod_{i=1}^n V_i^{e_i} \times \prod_{i=1}^m N_i^{f_i} \to \ker \pi$. 
 By \cite[Lemmas 2.3, 2.4, 2.5]{Sawin2020} and the easily-checked fact that bijective $[F]$-morphisms are $[F]$-isomorphisms,
when  one such isomorphism exists, then there are exactly
\[ \prod_{i=1}^n \abs{ \GL_{e_i}(\mathbb F_{q_i})}  \prod_{i=1}^m\abs{ \Aut_{[F]}(N_i)}^{f_i} f_i! =  \prod_{i=1}^n  \left( q^{ \binom{e_i}{2}} \prod_{j=1}^{e_i}  (q^{j}- 1) \right)  \prod_{i=1}^ m\abs{ \Aut_{[F]}(N_i)}^{f_i} f_i!  \] 
such isomorphisms.
Thus
  \[ v_{\C, F}  =\frac{1}{ \abs{ \Aut(F)}} \sum_{{\mathbf e} \in \mathbb N^n, {\mathbf f}\in \mathbb N^m}  \prod_{i=1}^n \frac{ (-1) ^{e_i} }{ \prod_{j=1}^{e_i}  (q^{j}- 1)} \prod_{i=1}^m \frac{ (-1)^{f_i} }{ \abs{ \Aut_{[F]}(N_i)}^{f_i} f_i! }   \sum_{G \in \C} \sum_{ \substack{ \pi \in \Sur(G,F) \\ \iota \colon  \prod_{i=1}^n V_i^{e_i} \times \prod_{i=1}^m N_i^{f_i}  \to G \\  \operatorname{Im} \iota= \ker \pi \\ \iota \textrm{ an $[F]$-group morphism } }}  \frac{M_G}{ \abs{\Aut(G)}} .\]

For each triple $G, \pi, \iota  $ we can quotient $G$ by all the factors of $\ker \pi$ except $\iota(V_i^{e_i})$, obtaining an extension of $F$ by $V_i^{e_i}$, with an extension class $\alpha_i \in H^2(F, V_i)^{e_i}$. We can quotient further to obtain $e_i$ individual extensions of $F$ by $V_i$, each of which must lie in $\C$, hence $\alpha_i \in W^{e_i}$.

There is a natural isomorphism from $G$ to the fiber product, over $F$, of each of the minimal extensions of $F$ obtained by quotienting $G$ by all of the $V_i$ and $N_i$ factors except one.   One can check that the minimal extension by $N_i$ is naturally isomorphic to $ \Aut(N_i) \times_{\Out(N_i)} F$,  where $G$ maps to $\Aut(N_i)$ by its conjugation action and to $F$ by the given map (using that the center of $N_i$ is trivial).
 Thus we have an isomorphism 
$G \ra G_{ {\mathbf e}, {\mathbf f}, {\mathbf \alpha}},$ respecting both the map from $\prod_{i=1}^n V_i^{e_i} \times \prod_{i=1}^m N_i^{f_i}$ and the maps to $F$.

We next check that,  for given ${\mathbf e}, {\mathbf f}$, each class $\alpha$ arises from exactly \[\Aut( G_{ {\mathbf e}, {\mathbf f}, {\mathbf \alpha}})  \prod_{i=1}^n \left(\frac{ \abs{H^0(F,V_i)}}{ \abs{V_i}\abs{H^1 ( F, V_i)}}\right)^{e_i}\] triples $G, \pi, \iota$. 
First, we note that each class $\alpha$ arises from $G= G_{ {\mathbf e}, {\mathbf f}, {\mathbf \alpha}}$ and taking $\pi,\iota$ to be the natural maps,
as $G_{ {\mathbf e}, {\mathbf f}, {\mathbf \alpha}}\in\C$.
We take the group $ G_{ {\mathbf e}, {\mathbf f}, {\mathbf \alpha}}$ to be the  representative of its isomorphism class we use for the sum.  From the last paragraph, we see that if $G_{ {\mathbf e}, {\mathbf f}, {\mathbf \alpha}}, \pi',\iota'$
gives $\alpha$, then there is some automorphism $\beta$ of $G_{ {\mathbf e}, {\mathbf f}, {\mathbf \alpha}}$ such that  $\iota'=\beta\iota$ and $\pi'=\pi\beta$.
So $\Aut(G_{ {\mathbf e}, {\mathbf f}, {\mathbf \alpha}})$ acts transitively on the non-empty set of triples $G, \pi, \iota$ that give $\alpha$, and the stabilizer is the subgroup of automorphisms preserving $\iota$ and $\pi$.

So it suffices to prove the number of automorphisms of $G_{ {\mathbf e}, {\mathbf f}, {\mathbf \alpha}}$ that preserve $\iota$ and $\pi$ is  $\prod_{i=1}^n \left(\frac{ \abs{V_i}\abs{H^1 ( F, V_i)}}{ \abs{H^0(F,V_i)}}\right)^{e_i}$. To do this, note that any such automorphism sends $g \in G$ to the product of $g$ with an element of  $\prod_{i=1}^n V_i^{e_i}  \times \prod_{i=1}^m N_i^{f_i}$, which since the automorphism fixes $N_i^{f_i}$ and thus fixes the action of $g$ on $N_i^{f_i}$ must commute with $N_i^{f_i}$ for all $i$ and thus must lie in $\prod_{i=1}^n V_i^{e_i}  $, giving a cocycle in $C^1(F, \prod_{i=1}^n V_i^{e_i}) $, and every cocycle gives such an automorphism. The number of cocycles is the size of the cohomology group  $\abs{ H^1(F, \prod_{i=1}^n V_i^{e_i}) } = \prod_{i=1}^n \abs{H^1(F, V_i)}^{e_i}$ times the number of coboundaries, which is $\prod_{i=1}^n \abs{V_i}^{e_i}$ divided by the number of cocycles in degree $0$, which is  $\prod_{i=1}^n \abs{H^0(F, V_i)}^{e_i}$, giving the stated formula.

Thus, we can replace the sum over $G,\pi,\iota$ with a sum over $\alpha$ after multiplying by 
$$\Aut( G_{ {\mathbf e}, {\mathbf f}, {\mathbf \alpha}})  \prod_{i=1}^n \left(\frac{ \abs{H^0(F,V_i)}}{ \abs{V_i}\abs{H^1 ( F, V_i)}}\right)^{e_i},$$ which cancels the $\Aut(G)$ term and gives the desired statement.

By definition, the assignment $M_G$ is well-behaved if and only if the sum, with an additional $Z(\pi)^3$ factor, is absolutely convergent. We have $Z(\pi) = 2^{\omega(\pi)}$ where $\omega(\pi) =  \#\{ 1\leq i \leq n \mid e_i >0 \} + \sum_{i=1}^m f_i$, by Lemma~\ref{group-semisimple-description}, and thus $Z(\pi)$ is bounded by a constant $2^n$ times $2^{\sum_{i=1}^m f_i}$. Thus, the $M_G$ are well-behaved if and only if the sum, with the signs removed and an $8^{\sum_{i=1}^m f_i}$ factor inserted, is absolutely convergent. We may also ignore the $\Aut(F)$ factor as it is bounded.
\end{proof}

\begin{lemma}\label{group-card-wb} Let $\category$ be the category of finite groups.   Let $u$ be a real number, and
for each $G\in \category$, let $M_G=\abs{G}^{-u}$. 
Then $(M_G )_G$ is well-behaved.

Keeping the notation of Proposition \ref{group-ext-formula}, we have
\[v_{\C,F} = \frac{\abs{F}^{-u}  }{ \abs{ \Aut(F)}} \prod_{i=1}^n \left( \prod_{j=1}^{\infty}  \left( 1- \frac{ \abs{H^0(F,V_i) }\abs{W_i}}{ \abs{V_i}^{u+1} \abs{H^1(F,V_i) }} q_i^{-j} \right) \right) \prod_{j=1}^m e^{ - \frac{1}{ \abs{N_i}^u \abs{ \Aut_{[F]}(N_i)}}}. \]

If $u$ is an integer then $v_{\C,F} \geq 0$ for all levels $\C$. \end{lemma}

Thus, by Theorem~\ref{analytic-main-measure}, for $u\in \Z$,  there is a unique measure on the set of isomorphism classes of pro-finite groups with moments $|G|^{-u}$ and it is given by the formulas in Lemma~\ref{group-card-wb}.

For a fixed $u$, if one takes $X_g$ to be the free profinite group on $g$ generators, modulo $g+u$ independent Haar random relations,  it is easy to compute that
$$\lim_{g\ra\infty} \E(\Sur(X_g,G))=|G|^{-u}$$ for every finite group $G$.
So then Theorem~ \ref{T:weak-convergence} implies that the $X_g$ weakly converge in distribution to a probability measure on the space of profinite groups, indeed the one given by the formulas in Lemma~\ref{group-card-wb}.
This recovers \cite[Theorem 1.1]{Liu2020}.
One can reconcile the formulas in Lemma~\ref{group-card-wb} with \cite[Equation~(3.2)]{Liu2020} using observations from the proof of Proposition~\ref{group-ext-formula}.

\begin{proof}   We observe that 
\[M_{ G_{{ \mathbf e}, {\mathbf f,} {\mathbf \alpha}}} = \abs{G_{{ \mathbf e}, {\mathbf f,} {\mathbf \alpha}}}^{-u} = ( \abs{F} \cdot \prod_{i=1}^n \abs{V_i}^{e_i}\cdot \prod_{i=1}^m \abs{N_i}^{f_i} )^{-u} = \abs{F}^{-u} \cdot \prod_{i=1}^n \abs{V_i}^{-u e_i}\cdot \prod_{i=1}^m \abs{N_i}^{-u f_i}\]
so by Proposition~\ref{group-ext-formula},  $(M_G )_G$ is well-behaved if and only if 
\[  \abs{F}^{-u}  \sum_{{\mathbf e} \in \mathbb N^n, {\mathbf f}\in \mathbb N^m}  \prod_{i=1}^n \left( \frac{\abs{H^0(F,V_i)}^{e_i}  \abs{V_i}^{ - u e_i} }{ \abs{V_i}^{e_i}   \abs{H^1 (F, v_i) }^{e_i} \prod_{j=1}^{e_i} (q_i^j-1) }  \right) \prod_{i=1}^m \left( \frac{8^{f_i}   \abs{N_i}^{-u f_i} } { \abs{ \Aut_{[F]}(N_i)}^{f_i} (f_i)! } \right) \sum_{ {\mathbf \alpha}  \in \prod_{i=1}^n W_i^{e_i}  }  1\]
is absolutely convergent. Since the length of the sum over ${\mathbf \alpha}$ is $\prod_{i=1}^n \abs{W_i}^{e_i}$, this is absolutely convergent if and only if
\[  \abs{F}^{-u}  \sum_{{\mathbf e} \in \mathbb N^n, {\mathbf f}\in \mathbb N^m}  \prod_{i=1}^n \left( \frac{\abs{H^0(F,V_i)}^{e_i}  \abs{V_i}^{ - u e_i}  \abs{W_i}^{e_i} }{ \abs{V_i}^{e_i}   \abs{H^1 (F, v_i) }^{e_i} \prod_{j=1}^{e_i} (q_i^j-1) }  \right) \prod_{i=1}^m \left( \frac{8^{f_i}   \abs{N_i}^{-u f_i} } { \abs{ \Aut_{[F]}(N_i)}^{f_i} (f_i)! } \right) \]
is.

Now note that every term in the ratio $\frac{ \abs{H^0(F,V_i)}^{e_i}  \abs{V_i}^{ - u e_i} \abs{W_i}^{e_i}  }{ \abs{V_i}^{e_i}   \abs{H^1 (F, v_i) }^{e_i} \prod_{j=1}^{e_i} (q_i^j-1)} $ is exponential in $e_i$, except for $ \prod_{j=1}^{e_i} (q_i^j-1)$, which is superexponential, so the ratio decays superexponentially. Similarly, every term in $\frac{8^{f_i}   \abs{N_i}^{-u f_i} } { \abs{ \Aut_{[F]}(N_i)}^{f_i} (f_i)! } $ is exponential in $f_i$ except for $(f_i)!$, which is superexponential, so the ratio decays superexponentially. Thus, the term being summed decays superexponentially in all the $e_i,f_i$ variables, and hence the sum is absolutely convergent.

To calculate $v_{\C,F}$, we similarly obtain
 \[  \hspace{-.5in} v_{\C, F}  = \frac{\abs{F}^{-u} }{ \abs{ \Aut(F)}} \sum_{{\mathbf e} \in \mathbb N^n, {\mathbf f}\in \mathbb N^m}  \prod_{i=1}^n \left( \frac{ (-1)^{e_i} \abs{H^0(F,V_i)}^{e_i} \abs{V_i}^{-u e_i}  }{ \abs{V_i}^{e_i}   \abs{H^1 (F, v_i) }^{e_i} \prod_{j=1}^{e_i} (q_i^j-1) }  \right) \prod_{i=1}^m \left( \frac{(-1)^{f_i}  \abs{N_i}^{-u f_i}  } { \abs{ \Aut_{[F]}(N_i)}^{f_i} (f_i)!  } \right) \sum_{ {\mathbf \alpha}  \in \prod_{i=1}^n W_i^{e_i}  } 1  \]
\[ = \frac{\abs{F}^{-u}  }{ \abs{ \Aut(F)}} \sum_{{\mathbf e} \in \mathbb N^n, {\mathbf f}\in \mathbb N^m}  \prod_{i=1}^n \left( \frac{ (-1)^{e_i} \abs{H^0(F,V_i)}^{e_i} \abs{V_i}^{-u e_i}  \abs{W_i}^{e_i}  }{ \abs{V_i}^{e_i}   \abs{H^1 (F, v_i) }^{e_i} \prod_{j=1}^{e_i} (q_i^j-1) }  \right) \prod_{i=1}^m \left( \frac{(-1)^{f_i}  \abs{N_i}^{-u f_i}  } { \abs{ \Aut_{[F]}(N_i)}^{f_i} (f_i)!  } \right)   \]
\[ = \frac{\abs{F}^{-u}  }{ \abs{ \Aut(F)}} \prod_{i=1}^n \left(  \sum_{e=0}^{\infty}  \frac{ (-1)^{e} \abs{H^0(F,V_i)}^{e} \abs{V_i}^{-u e}  \abs{W_i}^{e}  }{ \abs{V_i}^{e}   \abs{H^1 (F, V_i) }^{e} \prod_{j=1}^{e} (q_i^j-1) }  \right) \prod_{i=1}^m \left( \sum_{f=0}^{\infty} \frac{(-1)^{f}  \abs{N_i}^{-u f}  } { \abs{ \Aut_{[F]}(N_i)}^{f} f!  }\right)\]
\[ = \frac{\abs{F}^{-u}  }{ \abs{ \Aut(F)}} \prod_{i=1}^n \left( \prod_{j=1}^{\infty}  \left( 1- \frac{ \abs{H^0(F,V_i) }\abs{W_i}}{ \abs{V_i}^{u+1} \abs{H^1(F,V_i) }} q_i^{-j} \right) \right) \prod_{j=1}^m e^{ - \frac{1}{ \abs{N_i}^u \abs{ \Aut_{[F]}(N_i)}}} \]
by the $q$-series identity  $\sum_{e=0}^{\infty} \frac{  v^e}{ \prod_{j=1}^{e} (q^j-1)} = \prod_{j=1}^{\infty} ( 1 + v q^{-j} ) $ and the Taylor series for the exponential function.

For $u$ an integer, this is nonnegative because, on the one hand, the exponential terms are nonnegative, and, on the other hand,  $\frac{ \abs{H^0(F,V_i) }\abs{W_i}}{ \abs{V_i}^{u+1} \abs{H^1(F,V_i) }}$ is an integer power of $q$ because every term is the cardinality of a vector space over $\mathbb F_{q_i}$ raised to an integer power. 
(One can check that  $\C$  closed under fiber products and quotients implies that $W_i$ is a vector space over the endomorphism field of $V_i$.)
If this power is positive then one of the terms in the product over $j$ is zero, making this factor nonnegative, and otherwise, each term in the product over $j$ is nonnegative, making this factor nonnegative.
\end{proof}

\begin{corollary}Let $\category$ be the category of finite groups.  Let $(M_G)_G\in \R^{\category/\isom}$. If $M_G = O ( \abs{G}^n)$ for some real $n$, then  $(M_G)_G$ is well-behaved. \end{corollary}

\begin{proof} This follows immediately form Lemma \ref{group-card-wb}, taking $u=-n$, and Lemma \ref{big-O-wb}. \end{proof}

\begin{remark} It may be possible to prove, following the calculations in~\cite[\S7]{SW}, that $(M_G)_G$ is well-behaved as long as \[ M_G = O \left( \frac{ \abs{G}^{1-\epsilon} \abs{H_2(G,\mathbb Z)}}{ \abs{H_1(G,\mathbb Z)} }\right)\] for some $\epsilon>0$, but we do not pursue that here. 

\end{remark}

\subsection{General tools for constructing examples}
In this section, we give a general theorem which gives a large class of examples of \Diamantine categories, as well as tools for constructing new
\Diamantine categories from old ones.

A Mal'cev variety
\cite[142 Theorem]{Smith1976}
 is a variety in universal algebra (a set of abstract $n$-ary operations for different natural numbers $n$ together with a set of equations they satisfy) such that some composition of the $n$-ary operations is an operation $T(x,y,z)$ which satisfies $T(x,x,z)= z$ and $T(x,z,z)=x$. The most familiar example of such an operation is given by $T(x,y,z) = x y^{-1} z$ in the variety of groups, or in any variety which contains a group operation. 
 Given a Mal'cev variety, one has a category of finite algebras for that Mal'cev variety.  An object in this category is a finite set $S$ with functions $S^n\ra S$ for each abstract $n$-ary operation of the Mal'cev variety (for each $n$)  such that the functions satisfy the equations of the Mal'cev variety.  A morphism in this category is a function $S\ra S'$ that respects all the operations.
The congruences in any algebra for a Mal'cev variety (i.e. equivalence relations respecting the algebra structure, analogous to normal subgroups of a group or ideals of a ring) form a modular lattice \cite[Ch. VII: Theorem 4 and Theorem of Malcev]{Birkhoff1967}.

\begin{lemma}\label{malcev-construction}
For any Mal'cev variety with operations of finitely many different arities, the category of finite algebras for that Mal'cev variety, with morphisms given by surjections, is a \Diamantine category. \end{lemma}

This includes the cases of finite rings, 
finite commutative rings,
finite $R$-modules for a ring $R$, finite groups, 
finite groups with the action of a fixed group $\Gamma$,
finite quasigroups, finite Lie algebras over a finite field, finite Jordan algebras over a finite field, etc.

\begin{proof} 
In this proof, we use ``algebra'' to refer to an algebra for the given Mal'cev variety.

Since we choose morphisms to be surjections, every morphism is an epimorphism, and the epimorphisms are dual to the lattice of congruences 
\cite[Ch. VI, Sec. 4]{Birkhoff1967}
which is a modular lattice as mentioned above. 
The lattice is finite since the number of congruences is at most the number of partitions of the underlying set.

The automorphism group of a finite algebra is a subgroup of the permutation group of the underlying set, and thus is finite.

The only minimal objects are one-element algebras, which are all isomorphic, so indeed there are at most countably many isomorphism classes. 

The last finiteness step is more difficult.     
If $A$ and $B$ are both quotients of $D$, then the elements of $A\vee B$ are equivalence classes of elements of $D$, where the equivalence is that the equivalence defining $A$ and the equivalence defining $B$ both hold. 
Since we are in a Mal'cev variety,  the elements of $A\wedge B$ are equivalence classes of elements of $D$, where the $x$ is equivalent to $y$ if there is some $z$ such that $x$ and $z$ are equivalent in $A$ and $z$ and $y$ are equivalent in $B$ (\cite[Ch. VII: Theorem 4]{Birkhoff1967}).
Thus, in the setwise fiber product $A\times_{A\wedge B} B$ (which has a natural algebra structure), every element is in the image of the map from $D$.
Thus $D\ra A\times_{A\wedge B} B$ is a surjection of algebras, and the equivalence relation giving the fibers is precisely the same as that defining the fibers of the surjection $D\ra A\vee B$, and hence $A\vee B=A\times_{A\wedge B} B$.

For any simple surjection between algebras, let the width be the maximum size of a fiber of the underlying map of sets. For any algebra $A$, let the width of $A$ be the maximum width of any simple surjection between two quotients of $A$. Let $n$ be the maximum width of an element of a finite generating set of $\C$.

Since $A\vee B=A\times_{A\wedge B} B$,  we have that the sizes of the fibers of $A\vee B\ra A$ are all sizes of fibers of $B\ra A\wedge B$ and vice versa.
We next will see that under certain conditions, the width of $E\ra F$ is the same as the width of  $E\vee G \ra F\vee G$ or $E\wedge G \ra F\wedge G$.
Suppose $E\ra F$ is simple such that $F\geq G \wedge E$ (all as quotients of some object).  Then $\vee G$ gives an isomorphism $[G\wedge E,E]\ra [G,E\vee G]$
by Lemma \ref{dit}, and so $E\vee G \ra F\vee G$ is simple.  Moreover,  $(F\vee G) \vee E=E\vee G$ and $(F\vee G) \wedge E=F\vee (G\wedge E)=F$ (by the modular property and then the assumption that $F\geq G \wedge E$).  Thus, $E\vee G \ra F\vee G$ is $(F\vee G) \vee E \ra F\vee G$ and has the same width as
$E\ra (F\vee G) \wedge E$, which is $E\ra F$.  
Similarly, if $E\ra F$ is simple such that $E\leq G \vee F$, then $E\wedge G \ra F\wedge G$ is simple has has the same width as $E\ra F$.

Our first step is to check that all algebras in $\C$ have width at most $n$. By the definition of level, it suffices to check that the set of algebras of width at most $n$ is downward-stable and join-stable. That it is downward-stable is clear by definition. For join-stable, consider a simple surjection $B \to A$ of quotients of $D \vee E$ (in some lattice of quotients), where $D$ and $E$ have width at most $n$.   We have $ B \geq (D \wedge B) \vee A \geq A$ so, by simplicity $ (D \wedge B) \vee A$, is either equal to $B$ or $A$.

If $(D \wedge B)\vee A =B$, then 
$B = (D \wedge B) \vee A \leq D \vee A$, so by the above $D \wedge B\ra  D\wedge A$ is simple and has the same width as  $B \to A$.  Since $D \wedge B\ra  D\wedge A$ is a simple surjection of quotients of $D$,  we conclude $B \to A$ has width at most $n$.

If $( D \wedge B) \vee A =A$, then $A\geq (D \wedge B)$.   By the above,  $B \to A$ has the same width as  $B\vee D \to A \vee D$.
Moreover, since $B\vee D \leq E\vee D= E\vee (A\vee D)$,  we have that $(B\vee D) \wedge E \ra (A\vee D) \wedge E$ is simple and also has the same width as $B\ra A$.
However,  $(B\vee D) \wedge E \ra (A\vee D) \wedge E$ is a simple map of quotients of $E$, and hence has width at most $n$.  Thus $B \to A$ has width at most $n$.

This concludes the proof that all algebras in $\C$ have width at most $n$. It follows that, for $G$ in $\C$, any $H$ in $\C$ with a simple surjection $H\to G$ has cardinality $|H| \leq |G| n$. If there are finitely many operations in the variety, we are done, as there are finitely many algebra structure on any given finite set.

Otherwise, we need a slightly more complicated argument. Let $K$ be the product of all the elements of a finite generating set of $\C$. Any equation (involving the algebra operations) that holds in $K$ holds in every algebra in $\C$, since equations are preserved by quotients and fiber products. Since $K$ is a finite set, there can only be finitely many distinct $r$-ary maps $K\to K$ for each $r$, so sufficient equations hold in $K$ so that there is some finite set of operations, such that the equations holding on $K$, and hence for algebras in $\C$,  determine the value of every operation in terms of that finite set of operations.
It then holds that there are finitely many algebras in $\C$ on sets of cardinality $\leq |G| n$, completing the proof of the final finiteness step. \end{proof}

One we have a \Diamantine category, there are many ways to modify it and get further \Diamantine categories.

\begin{lemma}\label{product-category} Let $\category_1$ and $\category_2$ be \Diamantine categories. Then the product category $\category_1 \times \category_2$, whose objects are ordered pairs of an object of $\category_1$ and an object of $\category_2$ and whose morphisms are ordered pairs of morphisms, is a \Diamantine category. \end{lemma}

\begin{proof} We verify the four assumptions in order:

\begin{enumerate}

\item The poset of quotients of an ordered pair $(G_1,G_2)$ is the product of the posets of quotients of $G_1$ and $G_2$, and the product of two finite modular lattices is a finite modular lattice.

\item The automorphism group of an ordered pair $(G_1,G_2)$  is the product of the automorphism groups, and the product of two finite groups is a finite group. 

\item Since the poset of quotients is the product of quotients, the product of two downward-closed sets is downward-closed, and the product of two join-closed sets is join-closed. Thus, the objects in $\C$ are contained in the set of ordered pairs $(G_1,G_2)$ where $G_1\in \C_1$ and $G_2 \in C_2$, for $\C_1$ the set of first terms of objects in $\C$ and $\C_2$ the set of second terms. A morphism to $(G_1,G_2)$ is simple if and only if it is simple in one factor and an isomorphism to the other, so the finiteness of simple morphisms from objects in $\C$ follows from the finiteness of simple morphisms from objects in $\C_1$ and $\C_2$.

\item A minimal object of $\category_1 \times \category_2$ is the ordered pair of a minimal object of $\category_1$ and a minimal object of $\category_2$, so the set of isomorphism classes of minimal objects is the product of the sets of isomorphism classes of minimal objects of $\category_1$ and $\category_2$ and hence is countable.
\end{enumerate}
\end{proof}

For a product category, when the moments themselves split as products over the two factors, we have a simple criterion to check well-behavedness.

\begin{lemma} Let $\category_1$ and $\category_2$ be \Diamantine categories.  Let $(M_G^1)_{G\in \category_1}$ and $(M_G^2)_{G\in\category_2}$ be tuples of nonnegative real numbers indexed by the isomorphism classes, respectively, of $\category_1$ and $\category_2$.  Define $(M_{(G_1,G_2)})_{(G_1,G_2)\in\category_1\times \category_2}$ by the formula $M_{(G_1,G_2)}= M^1_{G_1} M^2_{G_2}$. If $(M_G^1)_{G\in \category_1}$ and $(M_G^2)_{G\in\category_2}$ are well-behaved then so is $(M_{(G_1,G_2)})_{(G_1,G_2)\in\category_1\times \category_2}$. \end{lemma}

\begin{proof} Let $\C$ be a level of $\category_1\times \category_2$. Let  $\C_1$ be the set of first terms of objects of $\C$ and $\C_2$ the set of second terms of objects in $\C$.  As in the proof of Lemma \ref{product-category}, $\C \subseteq \C_1\times \C_2$.

For $\pi \colon (G_1,G_2) \to (F_1,F_2)$ the product of epimorphisms $\pi_1\colon G_1 \to F_1$ and $\pi_2 \colon G_2 \to F_2$, the lattice $[(F_1,F_2), (G_1,G_2)]$ is the product of the lattices $[F_1,G_1]$ and $[F_2,G_2]$ so $Z(\pi) = Z(\pi_1) Z(\pi_2)$ and ${\mu}((F_1,F_2),(G_1,G_2)) = \mu(F_1,G_1) \mu(F_2,G_2)$. Thus
\[\sum_{(G_1,G_2)  \in \C}     \sum_{ \pi \in\Epi((G_1,G_2),(F_1,F_2)) }\frac{|\mu ((F_1,F_2),(G_1,G_2)|)}{|\Aut((G_1,G_2))|} Z ( \pi)^3 M_(G_1,G_2) \]
\[ = \sum_{(G_1,G_2)  \in \C}     \sum_{ \pi_1 \in\Epi(G_1,F_1) }  \sum_{ \pi_2 \in\Epi(G_1,F_1) } \frac{|{\mu}(F_1,G_1) | | \mu(F_2,G_2) |}{|\Aut(G_1)|  |\Aut(G_2)|}  Z ( \pi_1)^3 Z(\pi_2)^3 M_{G_1}^1 M_{G_2}^2 \]
\[ \leq  \sum_{(G_1,G_2)  \in \C_1 \times \C_2}     \sum_{ \pi_1 \in\Epi(G_1,F_1) }  \sum_{ \pi_2 \in\Epi(G_1,F_1) } \frac{|{\mu}(F_1,G_1) | | \mu(F_2,G_2) |}{|\Aut(G_1)|  |\Aut(G_2)|}  Z ( \pi_1)^3 Z(\pi_2)^3 M_{G_1}^1 M_{G_2}^2 \]
\[= \Bigl( \sum_{G_1\in \C_1}  \sum_{ \pi_1 \in\Epi(G_1,F_1) } \frac{ |{\mu}(F_1,G_1) |}{|\Aut(G_1)|  } Z ( \pi_1)^3 M^1_{G_1} \Bigr)\Bigl( \sum_{G_2\in \C_2}  \sum_{ \pi_2 \in\Epi(G_2,F_2) } \frac{ |{\mu}(F_2,G_2) |}{{|\Aut(G_2)|  } }Z ( \pi_2)^3 M^2_{G_2} \Bigr)< \infty\]
so $(M_{(G_1,G_2)})_{(G_1,G_2)\in\category_1\times \category_2}$ is well-behaved.\end{proof}

The following result will cover many kinds of additional structure we may add to a \Diamantine category to get another \Diamantine category.

\begin{proposition}\label{stability-master-theorem} Let $\mathcal F\colon \category_1 \ra \category_2$ be a faithful functor that sends epimorphisms to epimorphisms. Assume:

\begin{itemize}

\item  $\category_2$ is a \Diamantine category.

\item  The natural map from the poset of quotients of $G$ to the poset of quotients of $\mathcal F(G)$ is the inclusion of a sublattice (i.e. we have $H_1 \leq H_2$ if $\mathcal F(H_1) \leq \mathcal F(H_2)$ and the image is closed under joins and meets).

\item  That there is some $n$ such that  $\mathcal F$ applied to any simple epimorphism gives a morphism of dimension at most $n$.

\item Given a level $\C$ of $\category_1$, there are finitely many isomorphism classes of objects $G\in \C$   with $\mathcal F(G)$ isomorphic to a given object in $\category_2$.

\item $\category_1$ admits at most countably many isomorphism classes of minimal objects.

\end{itemize}

Then $\category_1$ is a \Diamantine category. 
\end{proposition}

\begin{remark}
When applying Proposition~\ref{stability-master-theorem}, it is very helpful when checking the conditions to note that  $\mathcal F$ faithful implies that if
 $\mathcal F(\pi)$ is an epimorphism then $\pi$ is an epimorphism. 
\end{remark}

\begin{proof} We verify the four assumptions in order:

\begin{enumerate}

\item The poset of quotients of $G$ is a sublattice of the poset of quotients of $\mathcal F(G)$. A sublattice of a finite modular lattice is modular since a modular lattice is defined by an algebraic equation involving joins and meets.

\item The automorphism group of $G$ is a subgroup of the automorphism group of $\mathcal F(G)$, hence a finite group.

\item Let $\C$ be a level of $C_1$ and let $\C'$ be the level of $C_2$ generated by the application of $\mathcal F$ to a generating set of $\C$.  Then the objects $G \in \category_1$ such that $\mathcal F(G) \in \C'$ are downward-closed and join-closed by the assumption on the natural map of posets, so $\mathcal{F}(\C)\subset \C'$. So it suffices to prove for an object $H$ that there are finitely many simple morphisms $G \to H$ with $\mathcal F(G)\in \C'$. By Lemma \ref{dim-C-finite} and assumption, there are finitely many possibilities for $\mathcal F(G)$ up to isomorphism. By assumption, there are finitely many possibilities for $G\in \C$ up to isomorphism, and then each one has finitely many surjections to $H$ by the previous two properties.

\item  For $G$ minimal, since the poset of quotients $G$ is a sublattice of the poset of quotients of $\mathcal F(G)$ and hence contains the minimal and maximal element, which since $G$ has only one quotient must be equal to each other, $\mathcal F(G)$ must be minimal. Considering the level $\C$ consisting of the minimal objects, we see there are finitely many isomorphism classes of minimal objects $G$ for each minimal object $\mathcal F(G)$, and thus countably many in total.

\item $\category_1$ admits at most countably many isomorphism classes of minimal objects by assumption.
\end{enumerate}
\end{proof}

One such structure is adding the action of a finite group, where $\mathcal{F}$ is the forgetful functor.  

\begin{lemma}\label{group-action-construction} Let $\category$ be a \Diamantine category. Let $\Gamma$ be a finite group. Define the category $\Gamma$-$\category$ in which objects  are pairs of an object $G\in \category$ and a homomorphism $\Gamma \to \Aut(G)$ and morphisms are epimorphisms in $\category$ compatible with the action of $\Gamma$. Then $\Gamma$-$\category$ is a \Diamantine category. \end{lemma}

\begin{proof}
We verify the conditions of Proposition \ref{stability-master-theorem} for the forgetful functor $\mathcal F: \Gamma$-$\category  \ra\category$ forgetting the $\Gamma$ action, and then the lemma will follow. 
We have that $\mathcal{F}$ is faithful and takes all morphisms to epimorphisms (by definition of the morphisms in $\Gamma$-$\category$).
We have assumed that $\category$ is a \Diamantine category.

Since $\Gamma$ and the automorphism group of any object of $\category$ are finite, there are finitely many isomorphism classes of objects in 
$\Gamma$-$\category$ mapping to an object of $\category$.

We can check that the poset of quotients of an object $G$ with an action of $\Gamma$ is the subposet of $\Gamma$-invariant elements of the poset of quotients of $G$.  Since the join and meet of $\Gamma$-invariant elements is $\Gamma$-invariant, the $\Gamma$-invariant elements of the poset of quotients of an object of $\category$ is a sublattice as required.

Furthermore, the minimal element of the poset is $\Gamma$-invariant, and thus always represents a $\Gamma$-equivariant quotient of $G$. If $G$ is not minimal and thus not equal to this quotient then $G $ with an action of $\Gamma$ is not minimal either. The isomorphism classes of minimal objects of $\Gamma$-$\category$ are thus the isomorphism classes of minimal objects of $\category$ endowed with an action of $\Gamma$, of which there are finitely many for each minimal object of $\category$ and thus countably many in total.

A simple morphism of objects with an action of $\Gamma$ has an underlying morphism of objects of dimension at most $\abs{\Gamma}$. Indeed, let $G \to H$ be a morphism of objects, and suppose we have compatible actions of $\Gamma$ on $G$ and $H$. The induced morphism of objects with an action of $\Gamma$ is simple if there is no $K \in [ H,G]$ that is $\Gamma$-invariant other than $H$ and $G$. Take any element of $[H,G]$ that is simple over $H$ and take the join of its orbit under $\Gamma$. Since this join depends only on the orbit, it is invariant under $\Gamma$, and it cannot be $H$, so it must be $G$. Thus $G$ is the join of at most $\abs{\Gamma}$ objects simple over $H$ and thus must have dimension at most $\abs{\Gamma}$ over $H$.

\end{proof}

Several previous works have considered $\Gamma$-groups.   
If we let $\category$ be the category of finite $\Gamma$-groups whose order is relatively prime to $|\Gamma|$,
Liu, Zureick-Brown and the second author proved the existence of, and gave explicit formulas for, a measure on the pro-category with moments $[G:G^\Gamma]^{-u}$ 
\cite[Theorem 6.2]{Liu2019}.
The first author \cite[Theorem 1.2]{Sawin2020} proved uniqueness and robustness in this setting for moments that are $O(|G|^n)$ for some $n$.
The  results of the current paper, along with following the ``blueprints'' above allow one to give a different proof of these explicit formulas and uniqueness and robustness results. 

Another structure is the addition of a map to a fixed object.

\begin{lemma}\label{slice-construction} Let $\category$ be a \Diamantine category. Let $K$ be an object of $\category$. The category $K$-$\operatorname{Ext}$ of $K$-extensions whose objects are objects of $\category$ together with an epimorphism to $K$ and where morphisms are epimorphisms of $\category$ respecting the epimorphisms to $K$ is a \Diamantine category. \end{lemma}

This construction is a variant of the notion of slice category, restricted to epimorphisms.

Applied to the category of finite groups, this produces the category of groups with a surjection onto a fixed finite group $K$. This is closely related to the category of groups with an action of $K$, as if $G$ is a group with an action of $K$ then $G \rtimes K$ is a group with a surjection onto $K$. 

\begin{proof} We verify the conditions of Proposition \ref{stability-master-theorem} for the forgetful functor $\mathcal F$ taking an object of $\category$ together with an epimorphism to $K$ to the underlying object, and then the lemma will follow.  This is faithful and takes all morphisms to epimorphisms (by construction of the category).
We have assumed that $\category$ is a \Diamantine category.
   For $G$ an object and $f \colon G \to K$ a surjection, the poset of quotients of $(G, f)$ is simply the interval $[K,G]$, which is indeed a sublattice.
 A morphism $(G,f) \to (H, g)$ is simple if and only if $G \to H$ is simple, so $\mathcal F$ applied to a simple morphism is always simple.
 There are finitely many isomorphism classes of pairs $(G,f)$ for any $G$ because there are finitely many epimorphisms $f \colon G\to K$.
in $\category$ since $\category$ being a \Diamantine category implies that quotient lattices and automorphism groups are finite. Finally, the only minimal objects of $K$-$\operatorname{Ext}$ are isomorphic to $K$ together with the identity map to $K$, so there is a single isomorphism class.
\end{proof}

We can also add any kind of finite data that can pushed forward along morphisms.  

\begin{lemma}\label{decorated-construction} Let $\category$ be a \Diamantine category. Let $\mathcal G$ be a functor from $\category$ to the category of finite sets. Then the category $(C,\mathcal{G})$ of pairs of an object $G\in \category$ together with an element $s\in \mathcal G(G)$, with morphisms $(G ,s_G) \to (H,s_H)$ given by morphisms $f\colon G \to H$ in $\category$ such that $ \mathcal G(f) (s_G) = s_H$, is a \Diamantine category.
\end{lemma}

One example of such a functor on the category of finite groups is provided by the $k$th group homology for any fixed $k$. The $k=3$ case was important in \cite{SW}. The ``bilinearly enhanced groups" studied by Lipnowski, Tsimerman, and the first author in \cite{Lipnowski2020} can also be expressed this way, as the set of bilinearly enhanced group structures on a finite abelian group is a functor from the category of finite abelian groups to finite sets.  Our results then can be used to give a new proof of  \cite[Theorem 8.17]{Lipnowski2020}, which proves uniqueness and robustness for finite abelian $\ell$-groups with bilinearly enhanced group structures.
It should also be possible to give a new proof of the existence result and formulas of  \cite[Theorem 8.14]{Lipnowski2020} using our main results and following the ``blueprints'' above.   {In forthcoming work, the authors will use the results from this paper in the category of $\Gamma$-groups with additional finite data as in Lemma~\ref{decorated-construction} (a class in $H^3(G,\mathbb Z/n)$) to give conjectures about the distribution of Galois groups of maximal unramified extensions of $\Gamma$-number fields, as well as $q\ra\infty$ results in the function field case.

Note that, if $\category$ has fiber products, then the category of pairs constructed here need not, unless $\mathcal F(G \times_K H) = \mathcal F(G) \times_{\mathcal F(K)} \mathcal F(H)$. For example, the category of finite groups together with a $k$th homology class does not have fiber products. It is for this and similar reasons that we avoid using fiber products and use joins instead, even though the familiar examples like groups, rings, and modules have fiber products. 

\begin{proof}We verify the conditions of Proposition \ref{stability-master-theorem} for the (faithful) forgetful functor $\mathcal F$ taking a pair $(G, s_G)$ to $G$, and then the lemma will follow. 

If a morphism  $f\colon G \to H$ in $\category$ such that $ \mathcal G(f) (s_G) = s_H$ is an epimorphism $(G, s_G) \to (H, s_H)$,  $F$ is another object of $\category$, and $h_1,h_2$ are two morphisms $H \to F$ such that $h_1 \circ f= h_2 \circ f$. Then  \[ \mathcal G( h_1) (s_H) = \mathcal G(h_1) ( \mathcal G(f)(s_G)) = \mathcal G( h_1 \circ f) (s_G) = \mathcal G( h_2 \circ f) (s_G) = \mathcal G(h_2) ( \mathcal G(f)(s_G)) = \mathcal G( h_2) (s_H).\] Let $s_F =  \mathcal G( h_1) (s_H) = \mathcal G( h_2) (s_H)$. Then $h_1,h_2$ are two morphisms $(H, s_H) \to (F, s_F)$ whose compositions with $f$ are equal, so they are equal, thus $h_1=h_2$. Hence $f$ is an epimorphism $G \to H$ in $\category$.

We have assumed that $\category$ is a \Diamantine category.
The poset of quotients of $(G, s_G)$ is identical to the poset of quotients of $G$ since each quotient $H$ has a unique element $ \mathcal G( \pi_G^H)(s_G) \in \mathcal G(H)$ compatible with the morphism $\pi_G^H$ from $G$. 
A morphism $(G,s_G) \to (H,s_H)$ is simple if and only if $G \to H$ is simple, and the dimension of such morphisms is bounded by one.
 Since $\mathcal G$ is valued in finite sets, there are finitely many isomorphism classes of pairs $(G,s_G)$ for a fixed $G$. 
 
Since the poset of quotients of $(G,s_G)$ is identical to the poset of quotients of $G$, we see that $(G,s_G)$ is minimal if and only if $G$ is minimal, and since there are finitely many isomorphism classes of pairs $(G,s_G)$ for each $G$, there are finitely many isomorphism classes of minimal $(G,s_G)$ in total.
\end{proof}

Moreover, in the situation of Lemma~\ref{decorated-construction}, we can  relate well-behavedness between 
$(\category,\mathcal{G})$ and $\category$.

\begin{lemma}\label{decorated-well-behaved}  Let $\category$ be a \Diamantine category. Let $\mathcal G$ be a functor from $\category$ to the category of finite sets.

Let $(M_{ H})_H\in \R^{(\category,\mathcal{G})/\isom}$ with $M_H\geq 0$ 
for all $H\in(\category,\mathcal{G})$.
 Then $(M_{ H})_H$ is well-behaved (for $(\category,\mathcal{G})$)  if and only if $(\sum_{s\in \mathcal G(G)} M_{(G,s)})_G\in \R^{\category/\isom}$ is well-behaved (for $\category$). \end{lemma}

Note that an analogous statement need not be true for other notions of ``categories of objects with extra structure" described in Proposition \ref{stability-master-theorem}.

\begin{proof}
The key observation is that for $\pi \colon (G, s_G) \to (H, s_H)$ a map of ordered pairs arising from a map $f \colon G\to H$, the interval $[H,G]$ is naturally isomorphic to the interval $[(H,s_H), (G,s_G)]$ and thus $Z(\pi) = Z(f)$ and $\mu (H, G) = \mu( (H,S_H), (G,s_G))$. 
We write $f_*$ for $\mathcal{G}(f)$.

We first show that if  $ (\sum_{s\in \mathcal G(G)} M_{(G,s)})_G$ is well-behaved then $(M_{H})_H$ is as well.  For $\C$ a level of the category of pairs, writing $\C'$ for the level generated by the underlying groups of a finite set of generators of $\C$, and $(F, s_F) \in \C$, we have
\begin{align*}
 &\sum_{(G,s_G)  \in \C}     \sum_{ \pi \in\Epi((G,s_G),(F,s_F)) }\frac{\abs{ {\mu}((F,s_F) ,(G,s_G) )}}{|\Aut((G,s_G) )|} Z ( \pi)^3 M_{(G, s_G)} \\
 = &\sum_{(G,s_G)  \in \C}     \sum_{ \substack{ f \in\Epi(G, F)  \\ f_* s_G = s_F }}\frac{\abs{{\mu}((F,s_F) ,(G,s_G) )}}{|\Aut((G,s_G) )|} Z ( \pi)^3 M_{(G, s_G)}\\
   =&\sum_{(G,s_G)  \in \C}     \sum_{ \substack{ f \in\Epi(G, F)  \\ f_* s_G = s_F }}\frac{\abs{{\mu}(F ,G )}}{|\Aut((G,s_G) )|} Z ( f)^3 M_{(G, s_G)}\\ 
 = & \sum_{ G \in \C'} \sum_{ \substack{ s_G \in \mathcal G(G) \\  (G, s_G) \in \C}}  \sum_{ \substack{ f \in\Epi(G, F)  \\ f_* s_G = s_F }}\frac{\abs{{\mu}(F ,G )}}{|\Aut(G)|} Z ( f)^3 M_{(G, s_G)} \\
 \leq &\sum_{ G \in \C'} \sum_{  s_G \in \mathcal G(G) }  \sum_{ \substack{ f \in\Epi(G, F) }}\frac{\abs{{\mu}(F ,G )}}{|\Aut(G)|} Z ( f)^3 M_{(G, s_G)} \\
 = &\sum_{ G \in \C'}  \sum_{ \substack{ f \in\Epi(G, F)) }}\frac{\abs{{\mu}(F ,G )}}{|\Aut(G)|} Z ( f)^3   \sum_{ s_G \in \mathcal G(G) } M_{(G, s_G)}  < \infty,
 \end{align*}
showing that $(M_{H})_H$ is well-behaved.

For the converse, we take $\C$ a level of $\category$, fix $F \in \C$, and let $ \overline{\C}$ be the level of $(\category,\mathcal{G})$ generated by $F$, and every object with a simple morphism to $F$, together with all possible values of $s_G$. Closure of $\overline{C}$ under joins implies that for every $G\in \C$ with a semisimple morphism to $F$ and $s_G \in \mathcal {G}(G)$, we have $(G, s_G) \in \C$. We now assume that $(M_{H})_H$ is well-behaved and have
\begin{align*}
 &\sum_{ G \in \C}   \sum_{ \substack{ f \in\Epi(G, F) }}\frac{\abs{{\mu}(F ,G )}}{|\Aut(G)|} Z ( f)^3   \sum_{ s_G \in \mathcal G(G) } M_{(G, s_G)}
 \\=&   \sum_{ G \in \C} \sum_{ \substack{ f \in\Epi(G, F) }}\frac{\abs{{\mu}(F ,G )}}{|\Aut(G)|} Z ( f)^3   \sum_{  s_G \in \mathcal G(G) }  \sum_{ \substack{s_F \in \mathcal G(F) \\  f_* s_G = s_F}}  M_{(G, s_G)} 
 \\= &\sum_{ s_F \in \mathcal G(F)} \sum_{ G \in \C} \sum_{ s_G \in \mathcal G(G) } \sum_{ \substack{ f \in\Epi(G, F)  \\  f_* s_G = s_F }}\frac{\abs{{\mu}(F ,G )}}{|\Aut(G)|} Z ( f)^3   M_{(G, s_G)}\\
  \leq &  \sum_{ s_F \in \mathcal G(F)} \sum_{ (G, s_G) \in \overline{\C} } \sum_{ \pi\in \Epi( (G, s_G),(F, s_F)) } \frac{\abs{ {\mu}((F,s_F) ,(G,s_G) )}}{|\Aut((G,s_G) )|} Z ( \pi)^3 M_{(G, s_G)}   < \infty  ,
\end{align*}
   showing that $\sum_{ s_G \in \mathcal G(G) } M_{(G, s_G)} $ is well-behaved.
\end{proof}

In the setting of Lemma~\ref{decorated-construction}, we also have another convenient way to check well-behavedness and existence  on
$(\category,\mathcal{G})$.
\begin{corollary} Let $\category$ be a \Diamantine category. Let $\mathcal G$ be a functor from $\category$ to the category of finite sets. Recall the category $(\category,\mathcal{G})$ from Lemma~\ref{decorated-construction}.

For $\C$ a level of $\category$, the set $\W_{\C}$ of isomorphism classes of $(G, s_G) \in (C,\mathcal{G})$ with $G \in \C$ is a narrow formation of  finite complexity.
 For each finite set of objects of $(\category,\mathcal{G})$, there exists a level $\C$ such that $\W_{\C}$ contains all those objects. 
 
If $(M_{ H})_H\in \R^{(\category,\mathcal{G})/\isom}$ is well-behaved at $\W_\C$ for each level $\C$ of $C$, then $(M_H)_H$ is well-behaved.

If $(M_{ H})_H\in \R^{(\category,\mathcal{G})/\isom}$ is well-behaved at $\W_\C$ for every level $\C$ of $C$, 
and $v_{\W_\C,F}\geq 0$ for every level $\C$ and $F\in\W_\C$, then 
for every level $\mathcal{D}$ of $(\category,\mathcal{G})$
a measure $\nu$ exists on $\mathcal{D}$ with moments  $(M_{ H})_H\in \R^{\mathcal{D}}$
and $v_{\mathcal{D},H}\geq 0$ for every $H\in \mathcal{D}$.
 \end{corollary}

\begin{proof} It is helpful to recall the facts from the proof of Lemma~\ref{decorated-construction} that a morphism $(G,s_G)\to (H,s_H)$ is an epimorphism if and only if the underlying map $G\to H$ is, and that the poset of quotients of $(G,s_G)$ is thereby isomorphic to the poset of quotients of $G$.  Thus the complexity of $(G,s_G)$ is equal to the complexity of $G$.  Thus $\W_\C$ is finite complexity by Corollary~\ref{C:levelcom}.

In particular, if $(H, s_H)$ is a quotient of $(G,s_G)$ then $H$ is a quotient of $G$, so $\W_{\C}$ is downward-closed because $\C$ is. Similarly, the join of quotients $(H_1,s_{H_1})$ and $(H_2,s_{H_2})$ of $(G_2,s_{G_2})$ has underlying object the join of $H_1$ and $H_2$, so $\W_{\C}$ is join-closed because $\C$ is. Finally, $(G,s_G)$ is minimal only if $G$ is minimal, so $\W_{\C}$ contains all the minimal objects.  For $(G',s_{G'})$ an object of $\W_{\C}$ with a simple epimorphism to $(G, s_G)$, the induced morphism $G' \to G$ must be a simple epimorphism. There are finitely many isomorphism classes of such $G'\in \C$ because $\C$ is a level and $\category$ is a \Diamantine category, and then for each $G'$ there are finitely many choices of $s_{G'}$  since $\mathcal G(G')$ is a finite set.  So $\W_\C$ is a narrow formation.

For $(G_1,s_{G_1}),\dots, (G_n, s_{G_n})$ a finite set of objects of $(\category,\mathcal{G})$, let $\C$ be the level generated by $G_1,\dots, G_n$. Then clearly $\W_{\C}$ contains all these objects.  
It also follows that every level $\mathcal{D}$ of $(\category,\mathcal{G})$ is contained in $\W_{\C}$ for some level $\C$ of $C$.
 The well-behavedness of $(M_H)_H$  then follows from
  Lemma~\ref{la-input}. Using Theorem \ref{amc-23}, we can then  check the positivity criterion for existence on  $\W_\C$ and obtain existence on level $\mathcal{D}$.  The (Uniqueness) part of Theorem~\ref{T:general} then gives $v_{\mathcal{D},H}\geq 0$.
\end{proof}

The following special case of Lemma~\ref{decorated-construction} is often useful.

\begin{lemma}\label{subcategory} Let $\category$ be a \Diamantine category. Let $\category'$ be a full subcategory of $\category$ such that every morphism of $\category$ with source in $\category'$ has target in $\category'$. (For example, if all morphisms in $\category$ are epimorphisms, we can take any downward-closed full subcategory of $\category$.) Then $\category'$ is a \Diamantine category. \end{lemma}

\begin{proof} There is a unique functor from $\category$ to finite sets that takes each object of $\category'$ to the one-element set and every other object to the empty set, with the assumption that every morphism with source in $\category'$ has target in $\category'$ being the condition needed for the functor to exist.

Applying Lemma~\ref{decorated-construction} to this functor, we obtain $\category'$, showing that $\category'$ is \Diamantine.\end{proof}

We have seen some examples, but there are many more constructions allowed by Proposition~\ref{stability-master-theorem}.
Proposition~\ref{stability-master-theorem} will be useful because the variations of extra data that may need to be considered are endless.  
Liu \cite{Liu2022} works with $\Gamma$-groups with a certain kind of homology element, and the current authors plan to work in a similar setting in forthcoming work.  Proposition~\ref{stability-master-theorem} still applies in these settings.
We hope the results and methods of this paper can be used to prove
Conjectures 6.1 and 6.3 in \cite{Liu2022}, which are on existence, uniqueness, and robustness for measures with certain moments.

\subsection{Finite sets and finite trees}

We now give some very simple examples of \Diamantine categories, which still can lead to interesting results.

\begin{lemma} The opposite of the category of finite sets is a \Diamantine category. \end{lemma}

\begin{proof} This follows from Lemma \ref{malcev-construction} because the opposite of the category of finite sets is equivalent to the category of finite Boolean algebras, \cite[Prop. 7.30]{Awodey2006} which is a Mal'cev variety.

It can also be checked directly. A quotient object in the opposite of the category of finite sets is a subobject in the category of finite sets, i.e., a subset, and subsets of a fixed finite set form a modular lattice. The automorphism group of a finite set is finite, and simple morphisms to a finite set $S$ are injections of $S$ into a set of cardinality $\abs{S}+1$, of which there exists only one up to isomorphism.
\end{proof}

Since isomorphism classes of finite sets are parameterized by natural numbers, our main theorem in this context gives a statement about when probability distributions on natural numbers are determined by their moments.

 \begin{proposition} Let $M_0,M_1,\dots$ be a sequence of nonnegative numbers
  such that for all $m \geq 0$,   we have $\sum_{n \geq m}  \frac{ 2^{n-m} M_n }{ m ! (n-m)!} < \infty$. 
 
 Then there exists a sequence $\nu_m$ of nonnegative real numbers such that \[ \sum_{m=0}^{\infty} \left(\prod_{i=0}^{n-1} (m-i)  \right) \nu_m = M_n \] for all $n\geq 0$ if and only if 
 if and only if $ \sum_{n \geq m}  \frac{ (-1) ^{n-m} M_n }{ m ! (n-m)!}\geq 0$ for all $m\geq 0$. 
 
 If such $\nu_m$ exists, then we have $\nu_m =  \sum_{n \geq m}  \frac{ (-1) ^{n-m} M_n }{ m ! (n-m)!}$ for all $m\geq 0$.

 \end{proposition}
 
 This is a version of the classical moment problem. Note that the factorial moments $\sum_{m=0}^{\infty} \left(\prod_{i=0}^{n-1} (m-i)  \right) \nu_m $ carry the same information as the power moments $\sum_{m=0}^{\infty} m^n \nu_m$ since we can write either in terms of the other using the Stirling numbers. 
 
 \begin{proof} 
Let $\category$ be the opposite of the category of finite set. 
 First note that every object of this category lies in $\C$ where $\C$ is generated by the one-element set, because joins correspond to unions of sets and every set is the union of one-element sets. 
 
Since the isomorphism classes of finite sets are parameterized by natural numbers, a measure on $\C$ is the same thing as a measure on $\mathbb N$, i.e. a sequence $\nu_m$ of nonnegative real numbers.  We write $[n]$ for the object corresponding to a set with $n$ elements.

The epimorphisms from $[m]$ to $[n]$ are the injections from an $n$-element set to an $m$-element set. The number of these is $\prod_{i=0}^{n-1} (m-i)$, so $\Sur([m],[n]) = \prod_{i=0}^{n-1} (m-i)$. Thus the $[n]$-moment of the measure $\nu_m$ is given by $\sum_{m=0}^{\infty} \left(\prod_{i=0}^{n-1} (m-i)  \right) \nu_m $.
 
 For an epimorphism $\pi \colon [n]\to [m]$, the interval $[[n],[m]]$ is the lattice of subsets of an $n$-element set containing a fixed $m$-element set, i.e. the lattice of subsets of an $n-m$-element set, which is the product of $n-m$ copies of a two-element lattice. This implies $\mu( [n],[m])= (-1)^{n-m}$, and $\hat{\mu}([n],[m])  =  (-1)^{n-m} \binom{n}{m} $, as well as $Z(\pi) =2^{n-m}$.
 
 Thus \[ v_{\C, m} = \sum_{n=0}^{\infty} \frac{  (-1)^{n-m} \binom{n}{m}}{ n!}  M_n =   \sum_{n \geq m}  \frac{ (-1) ^{n-m} M_n }{ m ! (n-m)!}.\]
  Since $[[n],[m]]$ is a product of copies of the two-element lattice, we can use $a=1$ in the definition of ``behaved.''
We have
\begin{align*}&\sum_{[n] \in \C} \sum_{ \pi \in\Epi([n],[m]) }\frac{\abs{{\mu}([n],[m])}}{|\Aut([n])|} Z ( \pi) M_n = \sum_{n=0}^{\infty} \sum_{ \pi \in\Epi([n],[m]) }\frac{\abs{ (-1)^{m-n} }}{n!} 2^{n-m}  M_n  \\
= &\sum_{n=m}^{\infty} \frac{2^{n-m}}{ (n-m)!} M_n  = m! \sum_{n = m}^\infty  \frac{ 2^{n-m} M_n }{ m ! (n-m)!} < \infty,
\end{align*}
so our moments are behaved. 
 Then, the ``if'' of the first statement follows from Theorem~\ref{amc-23}.
The final statement of the proposition follows from Theorem~\ref{amc-1}. 
The ``only if'' of the first statement follows from the final statement.
   \end{proof}

   Our framework also incorporates the notion of local weak convergence of random rooted graphs, also known as Benjamini-Schramm convergence, at least in the special case of trees. We first recall this notion. For a vertex $v$ in a simple graph and a natural number $r$, the $r$th ordered neighborhood of $v$ is the induced subgraph consisting of all vertices connected by a path of length at most $r$ to $v$. The $r$th-order neighborhood can be considered as a rooted graph, with $v$ the root. Note that the $r$th-order neighborhood is always a connected graph. We say a graph is locally finite if the $r$th order neighborhood of each vertex is finite for all $r$.
   
 We can define a measurable space whose underlying set is the set of locally finite rooted graphs up to isomorphism, with the sigma-algebra generated by the set consisting of, for each rooted graph $G$ and natural number $r$, the set of graphs such that the $r$th ordered neighborhood of the root is isomorphic to $G$ as a rooted graph. (This set may be empty if not every vertex of $G$ can be connected to the root by a set of length $\leq r$.) We also consider the topology generated by these graphs.
 
A sequence of random graphs $G_n$ converges locally weakly if, for each $r$, the probability distribution of the isomorphism class of the $r$th ordered neighborhood of a vertex of $G_n$ chosen uniformly at random converges as $n \to \infty$ to a probability distribution.  For a sequence $G_n$ of random graphs, the local weak limit is the limit of, for each $n$, the measure on this space obtained as the distribution of the connected component of $G_n$ containing a vertex chosen uniformly at random, with root that random vertex. This limit exists only if the sequence converges locally weakly.

We can restrict to trees by considering a measurable space of locally finite rooted trees, with the restricted sigma-algebra and subspace topology. We can then consider local weak convergence on this space of a sequence of random trees, or, more generally, a sequence of graphs such that the probability that the $r$th order neighborhood of a uniformly random vertex is a tree converges to $1$ for as $n$ goes to $\infty$ with fixed $r$ (a mild condition which is satisfied by many examples of interest that are not trees). Let us see how this probability space arises from our framework.

  Consider the category  whose objects are finite rooted trees (i.e. connected simple graphs without cycles, together with a marked vertex), where morphisms $T_1 \to T_2$ are inclusions of $T_2$ as a subgraph of $T_1$ that sends the root to the root.  Note that every tree that is a subgraph of another tree is automatically an induced subgraph, so it does not matter whether we consider ordinary or induced subgraphs. 
  
  \begin{lemma} The category of finite rooted trees is a diamond category.\end{lemma}
  
  \begin{proof} The finiteness of automorphisms groups is obvious, since they are a subset of the automorphism group of the set of vertices. The partially ordered set of quotients of an object $T$ is the partially ordered set $S_T$ of subsets of its vertex set, containing the root, whose induced subgraph is connected. The intersection of two connected subsets of the vertices of a tree is always connected, and the union of two connected subsets containing the root is connected, so $S_T$ is stable under unions and intersections. It follows that $S_T$ is a lattice, since unions give joins and intersections give meets, and a modular lattice, because the modular identity in the lattice of sets implies the modular identity in $S_T$. Furthermore, $S_T$ is finite since there are finitely many subsets, so the partially ordered set of quotients is a finite modular lattice. To check the third condition, we observe that given any proper connected subtree as a tree, we can always find a connected subtree with exactly one more vertex, showing that the only simple epimorphisms are $T_1 \to T_2$ occur when $T_2$ has exactly one  fewer vertex than $T_1$, or in other words where $T_1$ is obtained from $T_2$ by adding a leaf to one vertex. The number of these up to isomorphism is bounded by the number of vertices of $T_2$, which is finite.  There is only one minimal object.\end{proof}
  
  A pro-object in the category of finite rooted trees is the union of an increasing  net of finite rooted trees, which is simply a possibly-infinite rooted tree. A pro-object is small if and only if it is locally finite, i.e. each vertex has finitely many neighbors, since if a vertex a distance $n$ from the root has infinitely many neighbors, there are infinitely many morphisms to a path of length $n+2$ that send the 1st vertex to the root and the $n+1$st vertex to the vertex with infinitely many neighbors.
  
  The level generated by a path of length $n+1$ consists of all finite rooted trees of height $\leq n$, i.e. where each vertex has distance at most $n$ from the root, because all such trees may be written as joins of quotients of the path, and because finite rooted trees of height $\leq n$ are stabled under joins and quotients, where ``joins" are unions preserving the root and ``quotients" are subtrees containing the root. In fact, all levels are of this form, as every rooted tree of height $m$ contains a path of length $m+1$ as a subtree, so the level generated by a set of trees of maximum height $m$ both contains and is contained in the level generated by a path of length $m+1$. We can refer to the level generated by a path of length $n$ as level $n-1$. The maximal quotient in level $n$ of a possibly-infinite rooted tree is simply the subtree consisting of all vertices of depth $\leq n$, i.e. the $n$th-order neighborhood of the root.
  
  Thus, $\Prf$ is the set of isomorphism classes of locally-finite rooted trees, with the sigma-algebra and topology generated by, for each rooted tree of  height $\leq n$, the set of trees such that the $n$th-order neghborhood of the root is that fixed tree -- exactly the sigma-algebra we saw above was relevant for local weak convergence.
  
  For a measure on $\Prf$, the $T$'th moment for a rooted tree $T$ of a measure describing a random locally-finite rooted tree is the expected number of inclusions of $T$ as a rooted subtree of the random tree. Thus, for a measure obtained by taking a uniformly random root on a random tree with $n$ vertices, the $T$-moment is simply the expected number of inclusions of $T$ as an unrooted subtree, divided by $n$. ( In particular, the moment is independent of the choice of root.) 
  
  Two different models of random trees which arise naturally as local weak limits are the Galton-Wason tree, which is the local weak limit of the sequence of, for each $n$, the Erd\"H{o}s-R\'{e}nyi random graph on $n$ vertices with edge probability $\lambda/n$, and the skeleton tree, which is the local weak limit of the sequence of, for each $n$, a uniformly random tree on $n$ vertices.
  
  It is easy to check that both these families have moments given by simple formulas, which are well-behaved, so they are uniquely determined by their moments. The first has moment $M_T = \lambda^{ \# E(T)}$, where $\# E(T)$ is the number of edges, while the second has moment $\# V(T)$, where $\#V(T)$ is the number of vertices. Thus, at least two models that are natural from the point of view of local weak convergence are also natural from our point of view.

It would be possible, by defining a more complicated category, to obtain the probability space of local weak convergence of graphs, but we would not obtain such a simple notion of moments.

\bibliographystyle{alpha}
\bibliography{../references.bib,../MyLibrary.bib}

\end{document}